\documentclass[12pt, a4paper,reqno]{amsart}

\usepackage{amsmath, amsthm, amssymb}
\usepackage{graphicx}
\usepackage{enumerate}
\usepackage[top=2.5cm, bottom=2.5cm, left=2.5cm, right=2.5cm]{geometry}
\usepackage{comment, float}

\newcommand{\x}{\mathbf x}
\newcommand{\y}{\mathbf y}

\newtheorem{lemma}{\bf Lemma}
\newtheorem{proposition}{\bf Proposition}
\newtheorem{theorem}{\bf Theorem}
\newtheorem{corollary}{\bf Corollary}

\renewenvironment{proof}{\noindent {\bf Proof: }}{\rm\\}
\theoremstyle{definition}
\newtheorem{remark}{Remark}{\rm}
\newtheorem{example}{Example}{\rm}

\usepackage{color}
\usepackage{stmaryrd}

\usepackage[T1]{fontenc}
\usepackage[export]{adjustbox}

\usepackage{algorithm, algorithmic}

\makeatletter
\renewcommand{\p@algorithm}{\arabic{algorithm}\expandafter\@gobble}
\makeatother

\newcounter{step}[algorithm]
\newcommand\STEP[2][\(\triangleright\)]{%
	\refstepcounter{step}
	\vskip 0.25\baselineskip
	\item[]\hskip -\algorithmicindent #1 \textbf{Step \arabic{step}}%
	\ifthenelse{\equal{\unexpanded{#2}}{}}{}{ (\texttt{#2})}%
	\textbf{.}%
}
\newenvironment{algo}{\algo}{}
\def\algo#1\end{%
	\noindent\fbox{%
	\begin{minipage}[b]{\dimexpr\columnwidth-\algorithmicindent\relax}
	\begin{algorithmic}
	#1
	\end{algorithmic}
	\end{minipage}
	}%
\end}

\def\diag{\mathop{ \mbox{diag}}\limits}

\usepackage{bbm}

\usepackage{mathrsfs}
\usepackage{mathtools}

\newcommand{\vertiii}[1]{{\left\vert\kern-0.25ex\left\vert\kern-0.25ex\left\vert #1 
    \right\vert\kern-0.25ex\right\vert\kern-0.25ex\right\vert}}

    \usepackage{etoolbox}
    
    \usepackage{arydshln}

\DeclarePairedDelimiterX{\normi}[1]
  {|\!|\!|}
  {|\!|\!|}
  {\ifblank{#1}{\:\cdot\:}{#1}}

\begin{document}

\title[Kreiss system norm]{Minimizing transients via the Kreiss system  norm}
\author{Pierre APKARIAN$^1$}
\author{Dominikus NOLL$^{2}$}
\thanks{$^1$ONERA, Department of System Dynamics, Toulouse, France}
\thanks{$^2$Institut de Math\'ematiques, Universit\'e de Toulouse, France}

\begin{abstract}
We introduce system norms 
which assess transient behavior of stable Linear Time-Invariant (LTI) systems. 
This allows us to address undesired responses to initial conditions, finite resource consumption signals, or persistent perturbations.  
We then consider the challenging problem  of minimizing these norms in closed loop
using structured linear feedback. The computed controllers mitigate transients in a linearized
closed loop, with the potential side effect of enlarging the region of stability of the underlying 
non-linear controlled system. In applications this helps to
prevent transition to undesired nonlinear regimes, limit cycles or chaotic behavior. 
The success of our approach is certified a posteriori using Lyapunov-like techniques and simulations, 
as we demonstrate through a variety of applications.\\[0.5cm]


\end{abstract} 

\maketitle

\section{Introduction}
It has been observed in the literature that the size of the region of attraction of 
a locally stable nonlinear system
\begin{equation}
\label{nl}
\dot{x} = Ax + \phi(x), \;\; x(0) = x_0\,,
\end{equation}
with $\phi : \mathbb{R}^n \to \mathbb{R}^n$ is a static, memoryless nonlinearity with $\phi(0)=0$, $\phi'(0)=0$, 
may strongly depend on the degree of normality of $A$.  When $A$ is far from normal, the linearization
$\dot{x}=Ax$, $x(0)=x_0$,
may have large transient peaks,  which may incite
trajectories of (\ref{nl}) to leave the region of attraction. This is known as {\it peaking}, 
\cite{sussmann1991peaking,francis1978bounded,lin2022co,Taira2017AIAA_Overview,Taira2019AIAA_Applications},
and considered a major obstacle to global stability. 

The tendency of a stable  $A$ to produce large transients or peaking 
may be assessed by its
worst-case transient growth
\begin{equation}
\label{M0}
M_0(A)= \max_{t\geq 0}
\max_{\|x_0\|_2=1} \|e^{At}x_0\|_2 = \max_{t\geq 0} \overline{\sigma}(e^{At}),
\end{equation}
and in closed loop, when $A$ depends on tunable parameters, one may minimize $M_0(A_{\rm cl})$
in order to enlarge the region of local stability of (\ref{nl}). This has been studied in \cite{an_kreiss} for structured controllers, and previously in \cite{hinrichsen} using the controller Q-parametrization.

In a continuous operating process the effect of initial values is not the appropriate  
lever, as instability is caused rather by noise, persistent perturbations, or finite-consumption disturbances. 
Moreover,
nonlinearity often arises only in some
of the states $z$,  and likewise may  affect only parts of the dynamics,   and  
we address those issues by considering as a refined version of (\ref{nl}) a nonlinear controlled system
of the form
\begin{align}
\label{NL}
\begin{split}
\dot{x} &= Ax + B\phi(z)+ Bw  + B_uu\\
z &= Cx \\
y &= C_yx
\end{split}
\end{align}
with  $x\in \mathbb R^n$, $u\in \mathbb R^m$, $y\in \mathbb R^p$, $w\in \mathbb R^{m_w}$, $z\in \mathbb R^{p_z}$,
where the non-linearity $\phi:\mathbb R^{p_z} \to \mathbb R^{m_w}$
satisfies $B\phi(0)=0$ and $B\phi'(0)C=0$, and where a tunable feedback controller $u = K(\x) y$, with $\x$ as decision variables, is sought 
which stabilizes
the system locally, rendering it as resilient  as possible with regard to these disturbances.
The latter is aimed at indirectly by tuning the
closed loop channel $w \to z$ to remain small with regard to a system norm assessing transients, the idea being
that disturbances $w$ cause the partial state $z$ to have unduly large transients. 

Closing the loop with respect to the controller $K(\x)$ in (\ref{NL}), we consider the linear closed-loop channel
$T_{wz}(\x,s) = C(sI- A_{\rm cl}(\x))^{-1} B$, which we now tune in such a way that transients in $z(t)$ due to disturbances
$w(t)$ remain small. 
Expanding on (\ref{M0}), we assess transients of $G(s)=C(sI-A)^{-1}B$ via 
\begin{equation}
    \label{M}
\mathcal M_0(G) = \sup_{t\geq 0}  \overline{\sigma}\left( Ce^{At}B\right) = \sup_{\|w\|_1\leq 1} \| G \ast w\|_\infty,
\end{equation}
a time-domain $L^1 \to L^\infty$ induced system norm,  
which measures the time-domain peak of the response $z=G \ast w$ 
to a finite consumption input $w$. For $G(s)=(sI-A)^{-1}$ we recover $\mathcal M_0(G)=M_0(A)$. 

The principal goal of this contribution  is to develop a closed-loop controller design technique,
which mitigates transients of (\ref{NL}) via the indicated heuristic, performs fast and reliably, and
at the same time can be combined with standard design specifications in robust control. In addition, this
should be achieved with simple and practically useful
controller structures commonly used in engineering designs.

In order to achieve this goal,
we rely on frequency domain techniques, which
leads us to
introduce the  Kreiss system norm $\mathcal K(G)$ as a frequency domain approximation to (\ref{M}),  
the definition being given in
Section \ref{sect-Kreiss}. The fact that $\mathcal K(G)$ is frequency-based  offers algorithmic advantages
for optimization 
and
combines favorably with classical frequency-domain specifications, such as stability margins, noise and disturbance attenuation, loop shaping constraints,
allowing realistic and practically relevant design settings. A challenge is that this leads to multi-objective optimization
programs with non-smooth criteria and constraints.

Along with disturbances of finite consumption,  $w\in L^1$, it also makes sense to consider finite energy perturbations $w\in L^2$, 
which may be thought of as representing
noise, or time-domain bounded $w \in L^\infty$, which stand for persistent perturbations, as naturally all those could
be the reason why trajectories of (\ref{nl}) or (\ref{NL})  get outside the region of attraction. 
Ability of a system to withstand destabilizing disturbances is referred to as {\it resilience}, and 
along with $\mathcal M_0(G)$ or $\mathcal K(G)$ other ways to quantify  it have been discussed, see e.g.
\cite{krakov}. 
Resilience of systems is currently a subject of broad interest  and addressed in various ways, see e.g.
\cite{ghanbari24,demmer23,bouvier23,boerner21}.

In parallel with (\ref{M}), where peaking is quantified in the time-domain $L^\infty$-norm in response to finite consumption inputs,  
dissipative system theory assesses transient responses to initial values in the energy norm
\cite{Packard2003b,khong_24,khong25,Veenmann2013}.
While related, there is no direct link between these concepts. Yet it is worth mentioning that dissipativity analysis based on quadratic storage functions may be cast as 
integral quadratic constraints (IQCs), see \cite{seiler_2015}. In analysis,
those lead to linear matrix inequalities (LMIs),  but in synthesis turn into
bilinear matrix inequalities (BMIs), which are non-convex and often cumbersome to solve due to the large number of optimization variables.
This is why alternative more successful
ways to address IQCs have been proposed, see \cite{gahinet20,simoes20,an_iqc,mixedApkarianNoll}.

There are cases where applying dissipativity theory to a nonlinear system with a static nonlinearity can be turned into LMI synthesis conditions. This occurs in the application of section 7, where the structure of the nonlinearity admits a simple characterization via quadratic constraints. In this setting, no multipliers are required and the synthesis conditions can be reduced to LMIs in the same vein as in \cite{kalur2021nonlinear,mushtaq2022feedback}. These conditions may be conservative and should therefore be evaluated in the context of each specific application. It should be emphasized that this approach does not provide a means to prescribe or restrict the structure of the controller.

Mitigating large  transients is a general concern in control design, and has been addressed e.g. in
\cite{astolfi22,whidborne2007minimization,MQMcW2011,ray21}. 
LMI approaches are discussed in \cite{boyd1994linear, whidborne2005minimization}, and  a 
comparison between minimization of (\ref{M})  and LMI techniques is \cite{quenon2021control}, suggesting that, in the case of plane Poiseuille flow, minimization (\ref{M}) may be less conservative.

The remainder of this article is organized as follows. Section \ref{sect-Kreiss} introduces 
the Kreiss system norm $\mathcal K(G)$ as a frequency domain approximation
of $\mathcal M_0(G)$, followed by Section \ref{sect_optim}, which presents the central
Kreiss optimization program.
Section \ref{sect-Norm}  gives norm estimates related to $L_1$-disturbances. Section \ref{sect_kreiss} derives 
the system norm estimate $\mathcal K(G) \leq \mathcal M_0(G)$ from Young's inequality.
In Section \ref{new_subs} we investigate attainment of the lower bound $\overline{\sigma}(CB) \leq \mathcal K(G)\leq \mathcal M_0(G)$. This is important in view of the quest whether
lack of normality of the system $A$-matrix continues to be the cause of unduly
large transients when the set-up is (\ref{NL}) and no longer (\ref{nl}).
Section \ref{sect-Persistent} addresses the case of persistent perturbations, again using Young's inequality. Experiments
in Sections \ref{sect-LCA} and \ref{sect-ChaosFP} focus on $L_1$-disturbances, where we apply the Kreiss norm
minimization  of Section \ref{sect-Kreiss}  to control nonlinear dynamics involving limit cycles, chaos or multiple fixed points, with the goal to mitigate transients and thereby increase the region of local stability or to even achieve global stability in closed loop.
Conclusions are given in Section \ref{sect-Conclusion}.

\section*{Notation} Notation is standard. Time-domain $L^p$-spaces are equipped with classical signal norms as in \cite{induced_norms}.  
Time and  Laplace variables are $t$ and $s$, and Re$(.)$ denotes the real part,  
$\ast$ is convolution. For matrices $M$ symbols $M^T$, $M^H$, $M^{-1}$, ${\rm Tr}(M)$ mean transpose, conjugate transpose, inverse and trace, 
$I_n$ stands for the identity matrix of size $n$. We use ${\rm diag}(A_1,A_2)$ to denote 
 a block-diagonal matrix with blocks $A_1$ and $A_2$. For Hermitian matrices, $M \succ N$ means $M-N$
 is positive definite, $M \succeq N$ means $M-N$ is positive semi-definite. Maximum singular values and maximum eigenvalues are denoted
 $\overline{\sigma}$ and $\overline{\lambda}$.  The matrix exponential is $e^ A$. The $H_\infty$-norm of a transfer function $G(s)$
 is denoted $\| G\|_\infty$ or
 $\|G(s)\|_\infty$. The Clarke directional derivative of a locally Lipschitz function $f$ is $f'(x,d)$,
 the Clarke sub-differential is $\partial f(x)$; \cite{clarke1990optimization}.
The adjoint of  a linear operator $T$ is $T^*$. 
Additional specific notations are introduced within the text. 

\section{Kreiss system norm \label{sect-Kreiss}}
As observed  in \cite{leveque}, it may be difficult to compute $M_0(A)$ and  $\mathcal M_0(G)$ fast and accurately enough for 
the purpose of optimization. 
In response, the authors of  \cite{leveque} propose to use the  {\it Kreiss constant} $K(A)$ of a matrix $A \in \mathbb R^{n\times n}$
as an alternative measure of normality. The latter is defined as
\begin{equation}
\label{kreiss}
K(A) = \max_{{\rm Re}(s) > 0} {\rm Re}(s) \overline{\sigma}\left((sI-A)^{-1}\right),
\end{equation}
and its computation was investigated in \cite{mitchell1,mitchell2,an_kreiss}. 
By the famous Kreiss Matrix Theorem \cite[p. 151, p.183]{trefethen_embree} the estimate
\begin{equation}
    \label{estimate}
K(A) \leq M_0(A) \leq en K(A)
\end{equation}
is satisfied with $e=2.7183..$ the Euler number and $n$ the matrix size, 
where the right hand estimate is generally pessimistic, but sharp as shown in \cite{leveque}.

In view of (\ref{estimate}) minimizing $K(A_{\rm cl})$ has an  effect similar to minimizing $M_0(A_{\rm cl})$, and this is in line with the observation
that the global minimum $K(A)=M_0(A)=1$  is the same for both criteria and occurs for normal $A$, 
and more generally, for matrices $A$ where $e^{At}$ is a
contraction in the spectral norm. In \cite{an_kreiss} we have shown that optimizing $K(A_{\rm cl})$ is numerically possible, and that
it has indeed the desired effect of driving $A_{\rm cl}$ closer to normal behavior.
This has incited a vivid interest in Kreiss constant minimization, see e.g.
\cite{Krakovska2024Resilience,LeeMarcus2023JFM,ShcherbakovDabbene2022EJC,GarciaHilares2023PhD,Lee2024Thesis,DudarenkoEtAl2023IA}.

Expanding to (\ref{NL}) requires
computation and
optimization of $\mathcal M_0(G)$, which in closed loop encounters similar difficulties. We therefore
introduce the Kreiss system norm
$$
\mathcal K(G) := \sup_{{\rm Re}(s) > 0} {\rm Re}(s) \overline{\sigma} \left( C (sI-A)^{-1} B\right),
$$
which generalizes $K(A)$ in a natural way and satisfies the same estimate
\begin{equation}
\mathcal K(G) \leq \mathcal M_0(G) \leq en\, \mathcal K(G),
\end{equation}
as we shall prove in Section \ref{sect_kreiss}.  The principled reason to use $\mathcal K(G)$
is that its computation, and for that matter, optimization, may be based on a robust control technique, first proposed in  \cite[Thm. 2.1]{an_kreiss}
for the case $B=C=I_n$:
\begin{lemma}
\label{theorem1}
Suppose $A$ is stable. Then
the Kreiss system norm $\mathcal K(G)$ can be computed through the robust $H_\infty$-performance analysis program
\begin{equation}
\label{parametric}
\mathcal K(G) = \max_{\delta \in [-1,1]} \left\| C\left(sI-\left(\textstyle\frac{1-\delta}{1+\delta} A-I\right)\right)^{-1}B\right\|_\infty,
\end{equation}
where $\|G\|_\infty$ denotes the $H_\infty$-system norm.
\hfill $\square$
\end{lemma}

The Kreiss  norm can be computed either by solving a nonsmooth max-max program, or 
by a convex Semi-Definite Program  (SDP);  see \cite[Theorems 2.1 and 2.4]{an_kreiss}  and the discussion given there. The SDP provides a certified accuracy
and accounts for the worst case complexity, 
but the nonsmooth technique is considerably faster. In numerical testing, we 
therefore use the SDP only for the final certification. 
 
\section{Kreiss norm minimization}
\label{sect_optim}
This leads us now to the following synthesis program:
\begin{eqnarray}
\label{kreiss_program}
\begin{array}{ll}
\mbox{minimize} & \mathcal K(T_{wz}(\x)) \\
\mbox{subject to} 
& K(\x) \mbox{ stabilizing}\\
&\x\in \mathbb R^n
\end{array}
\end{eqnarray}
where $\x\in \mathbb R^n$ are the finitely many tunable parameters of the structured controller $K(\x)$,
and where $T_{wz}(\x,s)$ is the linear closed-loop channel of (\ref{NL}), by which we assess transients.
In practice
program (\ref{kreiss_program})  will be complemented by adding standard $H_\infty$- or $H_2$-loop-shaping requirements as constraints
to further improve performances and robustness, as for instance explained in
\cite{ApkarianNoll2021_OptBasedControl,AN2015,apkarianNoll2017worst}. Examples are 
(\ref{eq-synth1}) and (\ref{eqsynth}) in Sections 
\ref{sect-LCA} and \ref{sect-ChaosFP}.

Program (\ref{kreiss_program}) is a special case of a much wider class of problems with parametric uncertainty
discussed in \cite{AN2015,apkarianNoll2017worst}. 
Computation of $\mathcal K(T_{wz}({\bf x}))$ in closed loop
involves system
matrices of size $N:=n+n_K$, with $n$ the order of the plant, $n_K$ the controller state dimension, and is of the order $O(N^3)$ mainly
through Hamiltonian eigenvalue computations. The same complexity applies to computing transfer functions.
Clarke subgradients of criteria and constraints use 
\cite[Sect. IV, Prop. 1]{AN2015} and are of the order $O(p_z^3+m_w^3 + n_K^3)$, which gives some speedup of (\ref{NL})
over (\ref{nl}).
Some experiments are
documented in \cite{quenon2021control,an_kreiss}.

Further information in the assessment of (\ref{kreiss_program}) concerns
the number of iterations required by the optimizer.
Since design problems
are non-convex, we content ourselves with local minima, which is beneficial as finding global minima is NP-hard and
computationally infeasible for sizable problems. 
Non-smoothness
of criteria and constraints complicates matters, and in response
is addressed by tailored optimization techniques, 
\cite{apkarian2006nonsmooth,apkarian2006nonsmooth2}.
An advantage of our approach is that it
avoids the use of Lyapunov variables, so that the number of decision variables
is typically way smaller than the matrix dimension, dim$({\bf x}) \ll n+n_K$. This is
significant, 
because each optimization step calls for a convex quadratic program with size dim$(\x)$ and computational 
complexity
$O({\rm dim}(\x)^3)$. Altogether the optimizer
succeeds for medium size problems consistently under 100 iterations.

Comparisons
in \cite{an_kreiss} show that LMI-based methods fail
much earlier when the matrix dimension increases. Recent experiments in the literature as well as our own 
in Sections \ref{sect-LCA} and \ref{sect-ChaosFP} indicate that
$N=n+n_K$ is the dominant parameter.

\section*{Organization of the theoretical contribution}
In the following two sections we discuss theoretical aspects of disturbances causing
large transients via peaking. This includes:
\begin{description}
  \item[Section~\ref{sect-Norm}] 
  Finite consumption disturbances and 
   norm estimates relating time and frequency domain via Young's inequality:
  \begin{itemize}
    \item  Estimate for the Kreiss system norm (Section~\ref{sect_kreiss}),
    \item  Question of attainment of the lower bound of $\mathcal K(G)$ and the role of normality of the system $A$-matrix (Section~\ref{new_subs}).
    \item Hausdorff's numerical abscissa extended to systems
    (Section \ref{sect_abscissa}).
\end{itemize}

    \item[Section~\ref{sect_norms2}] 
    System norms other than $\mathcal M_0(G)$, $\mathcal K(G)$ which assess peaking:
    \begin{itemize}
    \item Alternative computable frequency based system norms (Sections \ref{sect_more}, \ref{sect_other}).
  \item
  Peak gain norm for 
   persistent perturbations and 
   estimate relating it to the  $H_\infty$-norm via Young's inequality. 
    \item  $L^2 \to L^2$ operator norm to address noise (Sections~\ref{sect_noise}).
  \end{itemize}
\end{description}

\section{Norm estimates \label{sect-Norm}}
In this Section, we obtain basic estimates relating the Kreiss system norm $\mathcal K(G)$ to
the  $L^1 \to L^\infty$ induced norm $\mathcal M_0(G)$.
We
recall Young's inequality:
\begin{lemma}
\label{lemma_young}
{\rm (Young's inequality; see \cite{best_young})}.
Let $1/p+1/q+1/r=2$, $p,q,r\geq 1$, and $1/p+1/p'=1$. Then
$$
\left|
\iint f(x) g(x-y) h(y) dy dx
\right| \leq C_{p}C_qC_{r} \|f\|_p\|g\|_q\|h\|_r,
$$
where
$$
C_p= \left( p^{1/p} \big/ p'^{1/p'}  \right)^{1/2}, C_1 = C_\infty = 1.
$$
\hfill $\square$
\end{lemma}

Let $\xi,\eta$ be test vectors of appropriate dimensions and consider a one-dimensional signal $u(t)$, then Lemma \ref{lemma_young} gives
\begin{align}
\label{young}
\begin{split}
\xi^TC(sI-A)^{-1} B \eta\, u(s) &= \int_0^\infty e^{-st} (\xi^T Ce^{At}B \eta \ast u)(t) dt \\
&\leq C_{p}C_qC_{r} \|e^{-st}\|_p \|\xi^TCe^{At}B\eta\|_q \|u\|_r \\
&= C_{p}C_qC_{r} {\rm Re}(s)^{-1/p} p^{-1/p}   \|\xi^TCe^{At}B\eta\|_q \|u\|_r,
\end{split}
\end{align}
where $f(t)=e^{-st}$,  $g(t)=\xi^TCe^{At}B\eta$, and $h(t) = u(t)$ are understood to take values $0$ for $t < 0$. 
In the sequel we consider various choices of $p,q,r$.

\subsection{Kreiss system norm}
\label{sect_kreiss}
We apply Young's inequality with 
$r=1$, $q=\infty$, $p=1$, where $C_pC_qC_r=1$.
This leads to the following

\begin{theorem}\label{KreissLB}
For a stable system $G(s) = C(sI-A)^{-1}B$ we have the estimate
\begin{equation}
\mathcal K(G) := \sup_{{\rm Re}(s) > 0} {\rm Re}(s) \overline{\sigma} \left( C(sI-A)^{-1} B\right)
\leq \sup_{t\geq 0} \overline{\sigma} \left(  Ce^{At}B \right) =: \mathcal M_0(G). 
\end{equation}
\end{theorem}

\begin{proof}
From (\ref{young}) with $r=1$, $q=\infty$, $p=1$, we get
$$
{\rm Re}(s) | \xi^TC (sI-A)^{-1} B\eta\, u(s)| \leq \|\xi^TC e^{At}B\eta \|_\infty \|u\|_1.
$$
Now take $u_\epsilon(t) = \epsilon^{-1}$ on $[0,\epsilon]$, $u_\epsilon(t) = 0$ else. Then $\|u_\epsilon\|_1=1$.
On the other hand, $u_\epsilon(s) \to 1$ as $\epsilon \to 0$, hence we get
$$
{\rm Re}(s) |\xi^TC (sI-A)^{-1} B\eta | \leq \|\xi^TCe^{At}B\eta\|_\infty =\sup_{t\geq 0} |\xi^T Ce^{At}B\eta |.
$$
Now we consider test vectors $\xi\in \ell_2$, $\eta\in \ell_2$.
Passing to the supremum over  $\|\xi\|_2\leq 1$, $\|\eta\|_2\leq 1$ on the right gives
\begin{align*}
{\rm Re}(s) |\xi^TC(sI-A)^{-1}B\eta| &\leq \sup_{t\geq 0} \sup_{\|\xi\|_2,\|\eta\|_2 \leq 1} | \xi^T Ce^{At} B\eta | \\
&= \sup_{t\geq 0} \overline{\sigma} \left( Ce^{At}B \right).
\end{align*}
Then taking the supremum over $\|\xi\|_2\leq 1$, $\|\eta\|_2\leq 1$ and ${\rm Re}(s) > 0$ on the left gives
$$
\mathcal K(G)= \sup_{{\rm Re}(s) > 0} {\rm Re}(s) \overline{\sigma} \left(  C(sI-A)^{-1}B\right) \leq\sup_{t\geq 0} \overline{\sigma} \left( Ce^{At}B \right)=\mathcal M_0(G),
$$
which is the claimed estimate.
\hfill $\square$
\end{proof}

In order to interpret the expression $\mathcal M_0(G)$ on the right, we 
consider vector norms on $L^p([0,\infty),\mathbb R^n)$ defined as
$$
\|u\|_{p,q} = \left(\int_0^\infty |u(t)|_q^p dt\right)^{1/p},
$$
where $|u|_q= \left(\sum_{i=1}^n |u_i|^q\right)^{1/q}$ is the standard vector $q$-norm in $\mathbb R^n$, and where
$\|u\|_{\infty,q} = \sup_{t\geq 0} |u(t)|_q$. Then, with the terminology introduced in \cite{induced_norms},
\begin{equation}
\label{induced_norms}
\|G\|_{(q,s),(p,r)} = 
\sup_{u\not=0} \frac{\|G\ast u\|_{q,s}}{\|u\|_{p,r}}
\end{equation}
are induced norms $G:(L^p,\|\cdot\|_{p,r}) \to (L^q,\|\cdot\|_{q,s})$.
In some  cases these admit closed-form expressions, which is a prerequisite to making them amenable to computations,
and even more so, optimization. By \cite[(25)]{induced_norms} one such case is
\begin{equation}
\label{many_others}
\|G\|_{(\infty,p),(1,r)} = \sup_{t\geq 0} \|G(t)\|_{p,r},
\end{equation}
where $\|A\|_{q,p} = \sup_{x\not=0} \|Ax\|_q/\|x\|_p$ are the usual well-known induced matrix norms.
Therefore, if we choose $p=r=2$ in (\ref{many_others}), then
\begin{equation*}
\|G\|_{(\infty,2),(1,2)} = \sup_{t\geq 0} \|G(t)\|_{2,2} = \sup_{t\geq 0} \overline{\sigma}(G(t)) = \mathcal M_0(G).
\end{equation*}
We have proved

\begin{proposition}
\label{prop1}
$\mathcal M_0(G)$ is an induced system norm. Given the vector input $w(t)$ satisfying 
$\int_0^\infty |w(t)|_2 dt = \int_0^\infty \left( \sum_{k=1}^p |w_k(t)|^2 \right)^{1/2} dt = 1$, it measures the output $z=G \ast w$ by the vector signal norm
$$\sup_{t\geq 0} \|z(t)\|_2
= \sup_{t\geq 0} \left( \sum_{i=1}^m \left| z_i(t) \right|^2\right)^{1/2}.$$
\hfill $\square$
\end{proposition}

The norm $\mathcal M_0(G) = \|G\|_{(\infty,2),(1,2)}$ will be called the worst case transient peak norm, as it
measures the peak of the time-domain response of $G$ to a signal with finite resource consumption.
Here 'response to a signal of finite resource consumption' is terminology adopted from \cite{boyd_barratt}.

In consequence,
the expression $\mathcal K(G)$ is a frequency domain  lower bound of $\mathcal M_0(G)$, and it is easy to see that $\mathcal K(G)$
is a norm, which we will call the {\em Kreiss system norm}. 

\begin{remark}
We do not expect $\mathcal K(G)$ to be an induced system norm, but
it does have the property of an operator norm, as follows from Theorem \ref{theorem1}. Indeed,
let $G_\delta = C(sI-(\frac{1-\delta}{1+\delta}A-I))^{-1}B$, then $\|G_\delta\|_\infty$ is the $L^2\to L^2$ induced
system norm when we take $\|\cdot\|_{2,2}$ as vector norm. Hence 
$\|z\|_{2,2} \leq \max_{\delta \in [0,1]} \|G_\delta\|_\infty \|w\|_{2,2}$, which due to (\ref{parametric})
gives $\|G \ast w\|_{2,2}\leq \mathcal K(G)\|w\|_{2,2}$.
\end{remark}

\begin{remark}
Suppose $G=(A,B,C)$ is output controllable. Then for $y_0\in {\rm im}(C)$, $y_0 \not=0$,  there exists $u_0$ and $t_0 > 0$ such that $Ce^{At_0}B u_0 = y_0$. 
Then $\mathcal M_0(G) \geq \overline{\sigma}(Ce^{At_0}B) \geq \|Ce^{At_0} B u_0\|_2/\|u_0\|_2 = \|y_0\|_2/\|u_0\|_2 > 0$. Some such condition is of course required, because if we take
$C = [1 \; 1]$, $B= \begin{bmatrix} 1\\-1\end{bmatrix}$, $A= -I_2$, then $Ce^{At}B=0$ for all $t$.
\end{remark}

\begin{remark}
The famous
estimate (upper bound due to Spijker \cite{spijker})
\begin{equation}
\label{spijker}
K(A) \leq M_0(A) \leq n e K(A)
\end{equation}
holds for matrices $A$ of size $n \times n$, and the global minimum $K(A)=M_0(A)=1$ is attained  for matrices
where $e^{At}$ is a contraction in the spectral norm, and in particular,  for normal matrices. For this reason $M_0(A)$, and $K(A)$, have been considered as
'measures of non-normality' of a matrix.  
\end{remark}

The following extends (\ref{spijker}), obtained in   \cite{kreiss_himself,leveque,spijker}, to system norms:

\begin{theorem}\label{KreissUB}
We have
$$ 
\mathcal K(G) \leq \mathcal M_0(G) \leq en\, \mathcal K(G).
$$
\end{theorem}

\begin{proof}
We have already shown in Theorem \ref{KreissLB} that $\mathcal K(G) \leq \mathcal M_0(G)$. For the upper bound estimate, take test vectors $\xi,\eta$, then on putting
$q(s) = \xi^TC(sI-A)^{-1}B\eta$, we have
\begin{align*}
\xi^T C e^{At} B\eta &= \frac{1}{2\pi j} \int_{{\rm Re}(s)=\mu} e^{st} \xi^TC(sI-A)^{-1}B\eta\, ds   \mbox{ (inverse Laplace)}\\
&= - \frac{1}{2\pi j} \int_{{\rm Re}(s)=\mu} \frac{e^{st}}{t} q'(s) ds  \mbox{ (partial integration)}\\ &= -\frac{1}{2\pi j} \frac{e^{\mu t}}{t} \int_{-\infty}^\infty e^{j\omega t} q'(\mu+j\omega) j\, d\omega
\end{align*}
Using
 ${\rm Re}(s)=\mu = 1/t$ and taking absolute values, we obtain
 \begin{align*}
 | \xi^TCe^{At}B\eta| 
&\leq \frac{e}{2\pi} \frac{1}{t} \int_{-\infty}^\infty |q'(1/t+j \omega) | d\omega = \frac{e}{2\pi} {\rm Re}(s) \|q'({\rm Re}(s) + j \cdot)\|_1.
  \end{align*}
Since by \cite{spijker} and \cite{leveque} we have
$\|q'\|_1 \leq 2\pi n \|q\|_\infty$, we find
\begin{align*}
|\xi^T Ce^{At}B\eta| &\leq en\, {\rm Re}(s) \sup_\omega | \xi^T C(({\rm Re}(s)+j\omega)I-A)^{-1} B\eta| \\
&\leq en \sup_{{\rm Re}(s) >0} {\rm Re}(s) |\xi^T C(sI-A)^{-1} B\eta|,
\end{align*}
so that taking the supremum over $\|\xi\|_2=1$, $\|\eta\|_2=1$ gives the right hand estimate.
\hfill $\square$
\end{proof}

\subsection{Attainment of the Kreiss lower bound}
\label{new_subs}
The fact that $K(A)$ and $M_0(A)$ attain their common global lower bound $K=M_0=1$
for contraction semi-groups $e^{At}$ in the spectral norm rises the question
whether the situation for $\mathcal K(G)$ and $\mathcal M_0(G)$ is similar. This is investigated in the present
section.  In particular, we ask whether a gap between $\mathcal K(G)$, $\mathcal M_0(G)$  and their 
lower bound can still be attributed to non-normal behavior of the system $A$-matrix.

\begin{proposition}
\label{bound}
We have the 
lower bound $\overline{\sigma}(CB) \leq \mathcal K(G)\leq \mathcal M_0(G)$.
\end{proposition}

\begin{proof}
For $x > 0$ we have
$\mathcal K(G) \geq x \overline{\sigma} (C(xI-A)^{-1} B) = \overline{\sigma}(C x(xI-A)^{-1}B)$, and since the matrix $x(xI-A)^{-1}$ approaches $I$ as $x \to \infty$,
we get the lower bound $\overline{\sigma}(CB)$ all right. 
\hfill $\square$
\end{proof}

For  $G=(sI-A)^{-1}$ this reproduces the bound $K(A)\geq 1$, which as we know is attained when $e^{At}$ is
a contraction in the spectral norm, and in particular, for normal matrices.  The question is therefore
whether, or for which systems $G=(A,B,C)$, the bound $\overline{\sigma}(CB)$  is attained. 
It is clear from Proposition \ref{bound} that $\mathcal M_0(G)=\overline{\sigma}(CB)$
implies equality $\overline{\sigma}(CB)=\mathcal K(G)=\mathcal M_0(G)$. However, in the matrix case
the reverse argument is also true, i.e., $K(A)=1$ implies $M_0(A)=1$ as a consequence of the Hille-Yosida theorem \cite{engel2000one}.
The analogous result for systems is no longer valid.

\begin{example} 
If we consider a stable SISO system
$$
G(s) = \frac{c_{n-1} s^{n-1} +\dots + c_0}{s^n + a_{n-1} s^{n-1} + \dots + a_0}
$$
then in controllable companion form
$$
A= \begin{bmatrix} 0 & 1 & 0& \dots& 0\\
0 & 0 & 1 & &\\
\dots & && \ddots &\\
0 & 0 & \dots & &1 \\
-a_0&-a_1& &&-a_{n-1} \end{bmatrix},\;
B = \begin{bmatrix} 0 \\ \vdots \\ 0 \\ 1\end{bmatrix},\;
C = [c_0 \dots c_{n-1}].
$$
If the degree of the numerator is $n-1$, then we can
normalize by taking the system $G/c_{n-1}$, then $\overline{\sigma}(CB)=1$, and we may ask whether there are choices of 
the $a_i$, $c_i$ where this bound is attained. However, if the degree
of the numerator is $\leq n-2$, then always $CB=0$, so here the lower bound is never attained. 
\end{example}

This leaves now two situations. In case $\overline{\sigma}(CB)=0$ one may wonder
under what conditions $\mathcal K(G)=\mathcal M_0(G) >0$ is satisfied, and whether this holds under normality of $A$. On the other
hand, when $\overline{\sigma}(CB) > 0$ one may ask under what conditions the lower bound is attained,
whether attainment  $\overline{\sigma}(CB)=\mathcal K(G)$ implies attainment $\overline{\sigma}(CB)=\mathcal M_0(G)$, and again, whether this is linked to normality of $A$. 

The following example shows that in the case $\overline{\sigma}(CB)=0$,
normality of $A$ is no longer the correct answer. 

\begin{example} 
\label{example3}
Take $C = [1 \; 1]$, $B = \begin{bmatrix} 1\\-1\end{bmatrix}$, $A = \begin{bmatrix} -\lambda & 0 \\ 0 & -\mu\end{bmatrix}$ with $0< \lambda < \mu$. Then
$\overline{\sigma} (CB)=0$, but $Ce^{At}B = e^{-\lambda t} - e^{-\mu t}  \not=0$ for $t > 0$, so that $\mathcal M_0(G) > 0$, and by the Kreiss matrix theorem we also have
$\mathcal K(G) > 0$. This also means that neither $\mathcal M_0$ not $\mathcal K$ are monotone in $t$. For $\lambda=1$, $\mu=2$ we obtain $\mathcal K(G) = 0.1716 < \mathcal M_0(G) = 0.25$, 
\end{example}

In case $\overline{\sigma}(CB) > 0$, the situation is also fairly unsettled, as the following examples
underline.

\begin{example}  
Take $B = [ 0\; 0 \;1]^T$,  $C = [1\; 1 \;1]$,  $a_0 = 0.9608$,  $a_1 = 1$,  $a_2 = 1$, in the controllable companion
form above, which gives
$G = (s^2 + s + 1)/(s^3 + s^2 + s + 0.9608)$, then
$|CB| = 1$,  $\mathcal K(G) = \mathcal M_0(G) = 1$.  Here
the lower bound is attained, while
$K( (sI -A)^{-1} ) = 1.17$,  $M_0( (sI-A)^{-1} ) = 1.43 $,  thus with $A$ not a contraction, and in particular, not normal.
\end{example}


\begin{example}  
\label{example5}
    Now we give an example where $\mathcal K(G) = \overline{\sigma}(CB) =1$, but $\mathcal K(G) < \mathcal M_0(G)$.
    Take $A = [-q , p;0,-q]$, $B=[b_1;b_2]$, $C=[c_1,c_2]$ with $b_1c_1+b_2c_2=1$. 
    Then with the choices $q = 0.6509$, 
$p = 0.8746$,
$C = [-19.5450, -19.1251]$, 
$B = [-0.2592 ;
     0.2126 ]$, 
     we get $\mathcal K(G)=1 < \mathcal M_0(G)=1.72$.
     This situation may also arise with normal $A$.
\end{example}


\begin{example} 
\label{example7}
    Example \ref{example5} can be used to analyze the special case considered in \cite{an_kreiss}, where the $C$-matrix
    is $J=[I_n,0]$ and the $B$-matrix is $J^T$.
    Starting out from the system
    in Example \ref{example5},
    we have to find a regular $2\times 2$
 matrix $T$ such that $CT^{-1}=[1,0]=J$ and $TB = [1;0]=J^T$. That requires $t_{11}=c_1$, $t_{12}=c_2$ and
 $c_1b_1+c_2b_2=1$. Moreover, we need to fix $t_{21},t_{22}$ such that
 $t_{21}b_1+t_{22}b_2=0$. That gives for $b_1 \not=0$:
 $$
 T = \begin{bmatrix} c_1 & c_2\\ -\frac{t_{22}b_2}{b_1}&t_{22}\end{bmatrix}
 $$
 which is regular for $t_{22}\not=0$. Now $G=Ce^{At}B= CT^{-1}T e^{At}T^{-1}TB=
 J  e^{TAT^{-1}t} J^T$, where $A$ is as in the previous example.  Then we have $1=\mathcal K(G)< \mathcal M_0(G)$,
 so the special structure $C=B^T=J$ used in \cite{an_kreiss} does not help. 
\end{example}



\begin{remark}
    Reference \cite{swaroop} gives conditions, under which any induced system norm
    attains the value $\overline{\sigma}(CB)$. Since this applies to $\mathcal M_0(G)$, this case gives attainment.
\end{remark}

\subsection{Numerical abscissa}
\label{sect_abscissa}
Hausdorff's numerical abscissa $\omega(A)$ satisfies
$\|e^{tA}\| \leq e^{\omega(A)t}$, hence $e^{tA}$ is a contraction semigroup iff $\omega(A)\leq 0$.
Since $\omega(A)=\frac{d}{dt} \|e^{tA}\| \left\vert \begin{array}{c} \! \\ \vspace{-.4cm} \! \end{array}\right._{\!\!\!\!\!\!\!t=0}$, 
the slope of the curve $t \mapsto \|e^{tA}\|$ at $t=0$ in the matrix case conveys global information on the entire curve, and the semigroup $e^{tA}$. 
This is why in the fluid flow literature it has been suggested
that minimizing $\omega(A_{cl})$ in closed loop  might be a way to prevent transition to turbulence
\cite{whidborne2007minimization,trefethen1993hydrodynamic,trefethen_embree,hinrichsen2000transient,schmid2014analysis,MQMcW2011}. Due to
$\omega(A)=\frac{1}{2} \overline{\lambda}(A+A^T)$ this would have the additional advantage of being an eigenvalue optimization problem, easier
to handle than (\ref{kreiss_program}). However, in \cite{an_kreiss} we demonstrated that minimizing $\omega(A_{cl})$ in closed loop does {\it not} have the desired effect of reducing transients.

Nonetheless, it is worthwhile  to extend $\omega(A)$ to systems as
$\omega(G) = \frac{d}{dt}\| C e^{tA}B\| \left\vert\begin{array}{c} \! \\ \vspace{-.4cm} \! \end{array}\right._{\!\!\!\!\!\!\!t=0}$,
because then $\omega(G) \leq 0$ continues to be  a necessary condition for attainment $\mathcal M_0(G)=\overline{\sigma}(CB)$.
However, unlike the matrix case, it is no longer sufficient.  Before 
showing this, we address necessity of attainment
for the Kreiss norm:

\begin{proposition}
\label{lb}
    A necessary condition for attainment of the lower bound $\mathcal K(G)=\overline{\sigma}(CB)$ is
    $\overline{\lambda}(Y+Y^T) \leq 0$, where $Y =Q^TCABB^TC^T$, and where the columns of $Q$ form an orthonormal basis of the maximum eigenspace of
    $CBB^TC^T$.
\end{proposition}

\begin{proof}
Let $A(\eta) = \frac{\eta}{2-\eta}A-I$ and put $G(\eta,s) = C(sI-A(\eta))^{-1}B$,
then (\ref{parametric}) can be re-written as $\mathcal K(G) = \max_{\eta \in [0,2]} \|G(\eta,\cdot)\|_\infty$. 
Now 
$\eta=0$ contributes the value $\overline{\sigma}(CB)$ to the maximum over $\eta\in [0,2]$, 
because
$A(0)=-I$, and therefore $G(0,s) = C(sI-A(0))^{-1}B = (s+1)^{-1}CB$, hence
$\|G(0,\cdot)\|_\infty = \max_\omega |(j\omega+1)^{-1}|\,\overline{\sigma}(CB) = \overline{\sigma}(CB)$, attained at the single frequency $\omega=0$.
In consequence,
due to our hypothesis $\mathcal K(G)=\overline{\sigma}(CB) > 0$,
the slope of
$\phi:\eta \mapsto \|G(\eta,\cdot)\|_\infty$ at $\eta = 0$ must be non-positive, as otherwise $\|G(\eta,\cdot)\|_\infty=\|C(sI-A(\eta))^{-1}B\|_\infty$
would attain values $> \overline{\sigma}(CB)$ for some small $\eta > 0$. 

To compute $\phi'(0)$, observe that
since $\|G(0,\cdot)\|_\infty$ is attained at the single frequency $\omega=0$, we have
\begin{align*}
\phi'(0)&= {\|\cdot\|_\infty}'(G(\eta,\cdot),\textstyle\frac{d}{d\eta} G(\eta,\cdot))\left\vert\begin{array}{c} \! \\ \vspace{-.4cm} \! \end{array}\right._{\!\!\!\!\!\!\!\eta=0} 
=\overline{\sigma}'(G(\eta,j0),\textstyle\frac{d}{d\eta}G(\eta,j0)) \left\vert\begin{array}{c} \! \\ \vspace{-.4cm} \! \end{array}\right._{\!\!\!\!\!\!\!\eta=0} 
\\
&= \overline{\sigma}'(CB,-C(A(\eta)^{-1} \textstyle\frac{d}{d\eta} A(\eta)A(\eta)^{-1})B) \left\vert\begin{array}{c} \! \\ \vspace{-.4cm} \! \end{array}\right._{\!\!\!\!\!\!\!\eta=0}
= \overline{\sigma}'(CB,C \textstyle\frac{1}{2}AB) \\ &= \textstyle\frac{1}{4}\, \overline{\lambda}(Q^H (CAB)P + P^H(B^TA^TC^T)Q) \\
&=\frac{1}{4 \overline{\sigma}(CB)} \overline{\lambda}(Q^H C(ABB^T + BB^TA^T)C^T Q), 
\end{align*}
where the second line uses $\frac{d}{d\eta} \left[ A(\eta)^{-1}\right]= - A(\eta)^{-1} \frac{d}{d\eta} A(\eta) A(\eta)^{-1} =
-A(\eta)^{-1} \frac{2}{(2-\eta)^2} A A(\eta)^{-1}$, which at $\eta=0$ gives $-\frac{1}{2}A$, whereas the third line uses 
Lemma \ref{sigma_prime} based on a SVD $G(0,0) =CB= \begin{bmatrix} Q&R\end{bmatrix} \begin{bmatrix} \overline{\sigma}(CB)I&\\&\Sigma\end{bmatrix}
\begin{bmatrix} P^T\\T^T\end{bmatrix}$.
The last line follows by re-substituting $Q^HCB = {\overline{\sigma}(CB)} P^H$.
\hfill $\square$
\end{proof}

Note that this leads back to $\omega(A)\leq 0$
for $C=B=I_n$.

\begin{lemma}
\label{svd}
    The Clarke subdifferential of the maximum singular value function 
    is $\partial \overline{\sigma}(G)= \{Q Y P^H: Y \succeq 0, {\rm Tr}(Y)=1\}$, where $G= \begin{bmatrix} Q & R\end{bmatrix} \begin{bmatrix} \overline{\sigma}(G)&\\&\Sigma\end{bmatrix}\begin{bmatrix} P^H\\T^H\end{bmatrix}$ is a SVD of $G$.
\end{lemma}

\begin{proof}
From $\overline{\sigma}(G)^2 = \overline{\lambda}(GG^H)$ we get $2 \overline{\sigma}(G)\partial \overline{\sigma}(G) = F'(G)^* \partial \overline{\lambda}(F(G))$, where $F:\mathbb M^{n,m} \to \mathbb S^m$ is the mapping $F(X)=XX^H$. 
Now $\partial \overline{\lambda}(GG^H) = \{QYQ^H: Y \succeq 0, {\rm Tr}(Y)=1\}$, where the columns of $Q$ in the SVD form an orthonormal basis of the maximum eigenspace of $GG^H$.
Furthermore, 
$F'(G)D = GD^H + DG^H$, hence for a test vector $S \in \mathbb S^m$ we have by the definition of the adjoint
$\langle D, F'(G)^*S\rangle = \langle F'(G)D,S\rangle = {\rm Re\, Tr\,} S(GD^H+DG^H)  = 2{\rm Re \, Tr\,}  SDG^H
= 2 {\rm Re\, Tr\,} (SG)^HD = \langle D, 2SG\rangle$, so that the action of the adjoint is
$F'(G)^*S = 2SG$. On substituting $S=QYQ^H\in \partial \overline{\lambda}(GG^H)$, we obtain 
$\partial \overline{\sigma}(G) = \frac{1}{2 \overline{\sigma}(G)} \{2Q Y Q^H G: Y \succeq 0, {\rm Tr}(Y)=1\}$. Now since $Q^HG = \overline{\sigma}(G) P^H$ from the SVD, we obtain the claimed
$\partial \overline{\sigma}(G) = \{QYP^H: Y \succeq 0, {\rm Tr}(Y)=1\}$.
    \hfill $\square$
\end{proof}

\begin{lemma}
\label{sigma_prime}
    The Clarke directional derivative 
    is $\overline{\sigma}'(G,D)= \frac{1}{2} \overline{\lambda}(Q^HDP + P^HD^HQ)$.
\end{lemma}

\begin{proof}
We have $\overline{\sigma}'(G,D) = \max\{\langle \Phi,D\rangle: \Phi\in \partial \overline{\sigma}(G)\}=
\max\{ {\rm Re\, Tr \,} \Phi^HD: \Phi \in \partial \overline{\sigma}(G)\} = \max\{{\rm Re \, Tr\,} PYQ^HD: Y \succeq 0, {\rm tr}(Y)=1\}
= \max \{ \frac{1}{2} {\rm Re\, Tr\,} Y(Q^HDP + P^HD^HQ):Y \succeq 0, {\rm Tr}(Y)=1\}= \frac{1}{2} \overline{\lambda} (Q^HDP+P^HD^HQ)$.
    \hfill $\square$
\end{proof}

On re-substituting $Q^HG = \overline{\sigma}(G)P^H$, we can also write this in the form
$\overline{\sigma}'(G,D) = \frac{1}{2 \overline{\sigma}(G)} \overline{\lambda}(Q^HDG^HQ + Q^HG D^HQ)
= \frac{1}{2 \overline{\sigma}(G)} \overline{\lambda}(Q^H   \left[DG^H + G D^H\right]Q)$.

The following is now a consequence of the finite maximum rule for the subdifferential \cite[Prop. 2.3.12]{clarke1990optimization}, along with a non-smooth chain rule
\cite[Sect. 2.8]{clarke1990optimization}. 
A similar 
argument was already used in \cite[Sect. III]{apkarian2006nonsmooth}, \cite[Sect. 4 and 6]{an:05}, and \cite[Thm. 3.2]{apkarian2006nonsmooth2}.

\begin{lemma}
    Suppose $\|G\|_\infty$ is attained at the finitely many frequencies $\omega_1, \dots,\omega_r$. Then
    $\partial \|\cdot\|_\infty (G) = \{ \sum_{k=1}^r Q_k Y_k P_k^H: Y_k\succeq 0, \sum_{k=1}^r {\rm Tr}(Y_k)=1\}$, where for every $k$ we let
    $G(j\omega_k)=\begin{bmatrix} Q_k& R_k\end{bmatrix} {\rm diag}(\|G\|_\infty , \Sigma_k) \begin{bmatrix} P_k&T_k\end{bmatrix}^H$ be a SVD of 
    $G(j\omega_k)$.
    \hfill $\square$
\end{lemma}

One also immediately gets the following description of the Clarke directional derivative of the $H_\infty$-norm:

\begin{lemma}
Suppose $\|G\|_\infty$ is attained at the finitely many frequencies $\omega_1, \dots,\omega_r$. Then
$\|\cdot\|_\infty'(G,D) = \displaystyle\max_{k=1,\dots,r} \textstyle\frac{1}{2} \overline{\lambda}(Q_k^HDP_k+P_k^HD^HQ_k)$, 
with $P_k,Q_k$ the same as above.
\hfill $\square$
\end{lemma}

\begin{remark}
    Formulas for subgradients and directional derivatives of the $H_\infty$-norm have first been given in
    \cite{an:05,AN:2007,apkarian2006nonsmooth,apkarian2006nonsmooth2}. Using the SVD as in Lemma \ref{svd} is numerically preferable
    to formulas using the subdifferential $\partial \overline{\lambda}$ directly, and we  exploited this favorably in the 
    implementation of {\tt hinfstruct} and {\tt systune} \cite{MatlabRobust}.
\end{remark}

\begin{corollary}
Condition $\overline{\lambda}(Y+Y^T)\leq 0$ is also necessary for attainment of the lower bound $\mathcal M_0(G)=\overline{\sigma}(CB)$. Moreover, 
when $C=B=I_n$, this condition is also sufficient.
\end{corollary}

\begin{proof}
    The first part of the statement follows from (\ref{KreissUB}) in tandem with Proposition \ref{lb}. One may also obtain
    it directly by computing
    $\frac{d}{dt}\| C e^{tA}B\| \left\vert\begin{array}{c} \! \\ \vspace{-.4cm} \! \end{array}\right._{\!\!\!\!\!\!\!t=0}
    = \overline{\sigma}'(CB,CAB)
    = \frac{1}{2 \overline{\sigma}(CB)} \overline{\lambda}(Q^T(CABB^TC^T+ CBB^TATC^T)Q) =
    \frac{1}{2\overline{\sigma}(CB)} \overline{\lambda}(Y+Y^T)$ using Lemma \ref{sigma_prime}.
    
    For $C=B=I_n$ we have $Q=I_n$, hence
    $\overline{\lambda}(Y+Y^T)=\overline{\lambda}(A+A^T) = 2\omega(A)$, but $\omega(A) \leq 0$ is the classical 
    necessary {\it and} sufficient condition for a contraction semi-group.
    \hfill $\square$
\end{proof}

\begin{example}
Now we show that the necessary condition $\overline{\lambda}(Y+Y^T)\leq 0$ is generally not sufficient, which contrasts with the case $C=B=I_n$, where this is true.
We take 
$A = \begin{bmatrix} 0&1\\-6&-5\end{bmatrix}$, $B= \begin{bmatrix} 0\\1\end{bmatrix}$, $C= \begin{bmatrix} -10&1\end{bmatrix}$, where
one gets $1 = \overline{\sigma}(CB) =  \mathcal K(G) < \mathcal M_0(G) = 1.5148$, $\omega(G) = -30$ with
$\overline{\lambda}(Y+Y^T) = -30$, confirming that the condition is necessary for attainment of the Kreiss norm, but not sufficient for attainment of the transient amplification.  

Another case is
$A = \begin{bmatrix}0&1\\-5&-1 \end{bmatrix}$, $B = \begin{bmatrix}0\\1\end{bmatrix}$, $C=\begin{bmatrix} -8 & 1\end{bmatrix}$, which gives
$\overline{\sigma}(CB) = 1< \mathcal K(G) = 1.13$ with $\overline{\lambda}(Y+Y^T) = -18$, showing that the condition is neither 
sufficient for attainment of the Kreiss norm, nor of transient amplification. 

\end{example}

We do not know whether there are cases with
$\overline{\sigma}(CB) < \mathcal K(G) = \mathcal M_0(G)$.

\section{More system norms for transients}
\label{sect_norms2}
In this Section we consider several alternatives to the worst case peak norm and the Kreiss
norm as its frequency approximation. Computability and use for optimization are
primordial criteria. Persistent perturbations and noise are discussed in .

\subsection{$L^1 \to L^\infty$ system norm with euclidean vector norm}
\label{sect_more}
The discussion in Section \ref{sect_kreiss} considers $\mathcal M_0(G)$ as induced operator norm $G:(L^1,\|\cdot\|_{1,2}) \to (L^\infty,\|\cdot\|_{\infty,2})$, 
with the $\ell_2$-norm
as vector norm. However, (\ref{many_others}) shows that other choices of vector norms could lead to numerically exploitable
expressions $\mathcal K, \mathcal M$. Choosing test vectors $\xi\in \ell_{p'}$, $\eta\in \ell_r$ gives
$$
\sup_{{\rm Re}(s) > 0}{\rm Re}(s) \| C(sI-A)^{-1}B\|_{r,p} \leq \sup_{t\geq 0} \| Ce^{At} B\|_{r,p},
$$
where $\|M\|_{r,p}$ is the $\ell_p \to \ell_r$ induced matrix norm. This may lead to other criteria compatible
with the goal to sensing $L^1 \to L^\infty$ amplification. 
Tractable expressions are obtained e.g. for $p=1$, $r=\infty$, which corresponds to taking
$\xi\in \ell_1$, $\eta\in \ell_1$. Here we get the estimate
$$
\sup_{{\rm Re}(s) > 0} \max_{ik} \left| c_i {\rm Re}(s) (sI-A)^{-1} b_k\right| \leq \sup_{t\geq 0} \max_{ik}  \left| c_i e^{At}b_k\right|
$$
which reads as
$$
\max_{ik} \mathcal K(c_i e^{A\bullet}b_k) \leq \max_{ik} \mathcal M_0(c_i e^{A\bullet} b_k)
$$
with a finite maximum of SISO Kreiss constants and transient growth norms involved. This practical entry-wise Kreiss norm offers potential to weigh some channels more than others. The upper bound of $\mathcal M_0$ is again  $e n \max_{ik} \mathcal K(c_i e^{A\bullet}b_k)$ from Theorem \ref{KreissUB}.

\subsection{$L^1 \to L^\infty$ system norm with  $\ell_\infty$-vector norm} 
\label{sect_other}
Now take (\ref{young}), but with 
$|\xi|_\infty \leq 1,|\eta|_\infty \leq 1$. We get on the right
$$
|\xi^T (Ce^{At}B) \eta| \leq \|(C e^{At}B)\eta\|_1 \leq \| Ce^{At}B\|_{1,\infty}
$$
because the dual norm to $\ell_\infty$ is $\ell_1$. However, this norm is not very helpful for matrices with large dimension $m$, because for $A\in \mathbb R^{m \times n}$, we have:
$$
\|A\|_{1,\infty} = \max_{r\in \{-1,1\}^m} \|Ar\|_1,
$$
where $\{-1,1\}^m$ are $m$-vectors of $\pm 1$ entries.
With the above technique,  
we easily get the following estimate
$$
\max_{r\in \{-1,1\}^m} \mathcal K(G r) \leq  \max_{r \in \{-1,1\}^m}  \mathcal M_0(Gr).
$$

\subsection{Peak-to-peak norm for persistent perturbations\label{sect-Persistent}}
In this section, we discuss the choice
$p=\infty$, $q=r=1$ in Young's inequality (\ref{young}), which will allow us to address the case of persistent perturbations
$w$ in (\ref{NL}), when for an input $\|w\|_{\infty,\infty} \leq 1$, we measure the response by the same signal norm $\|G\ast w\|_{\infty,\infty}$.
For test vectors $\xi,\eta$ and a one-dimensional signal $u$
we get from (\ref{young})
$$
|\xi^TC (sI-A)^{-1}B\eta \, u(s)| \leq \|e^{-st}\|_\infty \|\xi^TCe^{At}B\eta\|_1 \|u\|_1 =   \|\xi^TCe^{At}B\eta\|_1 \|u\|_1.$$
Letting   the scalar
signal $u(t)$ of unit $L_1$-norm approach the $\delta$-distribution,
we get
$$
|\xi^T C(sI-A)^{-1} B\eta| \leq \|\xi^TCe^{At}B\eta\|_1 = \int_0^\infty |\xi^TCe^{At}B\eta| dt.
$$
Now let $m$ be the number of outputs, $p$ the number of inputs, and let $g_{ik}(t)$ be the entries of the matrix
$Ce^{At}B$, then with $\|\xi\|_2 \leq 1$ and $\|\eta\|_2 \leq 1$ we get
\begin{align*}
\int_0^\infty |\xi^TCe^{At}B\eta| dt &= \int_0^\infty\left| \sum_{i=1}^m \sum_{k=1}^p \xi_i g_{ik}(t) \eta_k \right| dt\\
&\leq \sum_{i=1}^m |\xi_i| \sum_{k=1}^p|\eta_k|  \int_0^\infty | g_{ik}(t)|  dt  \\
&= \sum_{i=1}^m |\xi_i| \sum_{k=1}^p \| g_{ik}\|_1   |\eta_k| \\
&\leq \left(\sum_{i=1}^m |\xi_i|^2\right)^{1/2} \left( \sum_{i=1}^m  \left(\sum_{k=1}^p \|g_{ik}\|_1 |\eta_k|\right)^2\right)^{1/2}\\
&\leq \left( m \max_{i=1,\dots,m} \left(\sum_{k=1}^p \|g_{ik}\|_1\right)^2\right)^{1/2} = \sqrt{m}  \max_{i=1,\dots,m} \sum_{k=1}^p \|g_{ik}\|_1.
\end{align*}
When we recall that
the time-domain peak-to-peak, or peak-gain, system norm is defined as
$$
\|G\|_{\rm pk\_gn} = \max_{u\not=0} \frac{\|G\ast u\|_{\infty,\infty}}{\|u\|_{\infty,\infty}}= \max_{i=1,\dots,m} \sum_{j=1}^p \|g_{ij}(t)\|_1,
$$
then
we have shown the estimate $\|G\|_\infty \leq \sqrt{m} \|G\|_{\rm pk\_gn}$ for a system $G$ with $m$ outputs.  
The estimate remains true for more general systems with $g_{ik} \in L^1([0,\infty),\mathbb R^n)$, and even for 
($m \times p$)-valued Radon measures, allowing to include the case of direct transmissions. In the matrix case $\|\cdot\|_{\rm pk\_gn}$
reduces to the maximum row sum norm, i.e. the induced matrix norm $\ell_\infty-\ell_\infty$.

Let us now look at the reverse estimate, which is analogous to the right-hand estimate in the Kreiss matrix theorem
(Theorem \ref{KreissUB}). 
Consider a
stable finite-dimensional strictly proper system.
\begin{align*}
G: \quad
\left\{ \begin{array}{lll}
\dot{x} &= Ax + Bu\\
y &= Cx
\end{array} \right.
\end{align*}
where $G(t) = C e^{At}B$. Let $g_{ij}(t) = c_i e^{At}b_j$, where $c_i$ is the $i$th row of $C$, $b_j$ the $j$th column of $B$, then
$\|G\|_{\rm pk\_gn} =  \max_{i=1,\dots,m} \sum_{j=1}^p \|g_{ij}\|_1=\sum_{j=1}^p \|c_ie^{At}b_j\|_1$ for some $i$.

We now relate the peak-gain norm to the Hankel singular values of $G$.
The following was proved in the SISO case $p=m=1$ in \cite[Thm. 2]{boyd_doyle} for  discrete systems,
and in \cite[pp. 11-12]{mil} for continuous SISO systems, where in the latter reference the idea of proof is attributed to I. Gohberg.

\begin{lemma}
\label{sigmas_lemma}
Let $G$ be real-rational, strictly proper and stable, with $p$ outputs and McMillan degree $n$.
Then
\begin{equation}
\label{sigmas}
\|G\|_{\rm pk\_gn} \leq 2{p}^{1/2} \left( \sigma_{H1} + \dots + \sigma_{Hn}  \right), 
\end{equation}
where $\sigma_{Hi}$ are the Hankel singular
values of $G$. In particular, $\|G\|_{\rm pk\_gn} \leq 2n {p}^{1/2} \|G\|_\infty$.
\end{lemma}

\begin{proof}
We have for the $i\in \{1,\dots,m\}$ where the maximum is attained
\begin{align*}
\|G\|_{\rm pk\_gn} &= \sum_{j=1}^p \|c_i e^{At} b_j\|_1 
= 2 \sum_{j=1}^p \int_0^\infty  \left|c_i e^{2A\tau} b_j \right| d\tau \\
&=
2 \int_0^\infty \sum_{j=1}^p \left| (e^{{A^T}t} c_i^T)^T(e^{At}b_j) \right| dt\\
&\leq 2 \left(\int_0^\infty \sum_{j=1}^p \|e^{{A^T} t}c_i^T\|_2^2 dt\right)^{1/2} 
\left( \int_0^\infty \sum_{j=1}^p  \|e^{At}b_j\|_2^2 dt  \right)^{1/2}\\
&= 2p^{1/2} \left( \int_0^\infty  {\rm Tr}( e^{A^Tt}c_i^Tc_ie^{At}) dt \right)^{1/2} \left(  \int_0^\infty \sum_{j=1}^p {\rm Tr} (e^{At}b_jb_j^Te^{A^Tt}) dt\right)^{1/2}\\ 
&= 2 {p}^{1/2}  \left( \int_0^\infty {\rm Tr}( e^{A^Tt}c_i^Tc_ie^{At}) dt \right)^{1/2}
\left(  \int_0^\infty  {\rm Tr} (e^{At}BB^Te^{A^Tt}) dt\right)^{1/2}.%
\end{align*} 
Recall that
the observability and controllability Gramians of the system $G$
are
$$
W_o = \int_0^\infty e^{A^Tt}C^TC e^{At} dt, \quad W_c = \int_0^\infty e^{At} BB^Te^{A^Tt} dt.
$$
Now  Tr$\int_0^\infty e^{A^Tt}c_i^Tc_ie^{At} dt \leq {\rm Tr}\int_0^\infty e^{A^Tt}C^TCe^{At} dt$ 
follows from $c_i^Tc_i \preceq C^TC$ by applying a congruence
transformation with $e^{At}$.
Hence
$\|G\|_{\rm pk\_gn}  \leq 2{p}^{1/2}\left[ {\rm Tr}(W_o) {\rm Tr}(W_c) \right]^{1/2}$. Now
if we take a balanced realization, then $W_o=W_c = {\rm diag} (\sigma_{H1},\dots,\sigma_{Hn})$ for the Hankel singular
values $\sigma_{H_1} \geq \dots \geq \sigma_{Hn}$, hence 
$
\|G\|_{\rm pk\_gn} \leq 2\sqrt{p} (\sigma_{H1} + \dots + \sigma_{Hn})
\leq 2 p^{1/2} n \sigma_{H1} \leq 2p^{1/2}n \|G\|_\infty$
for a system without direct transmission. This uses $\sigma_{H1} \leq \|G\|_\infty$.
\hfill $\square$
\end{proof}

Note, however, that
by the Enns-Glover bound we have
$$
\|G\|_\infty \geq \max\{ \overline{\sigma}(D), \sigma_{H1}\}
$$
for the maximum Hankel singular value $\sigma_{H1}$ of $G=(A,B,C,D)$, so our estimate holds also for systems with direct transmission. 
Indeed, if we define the Hankel semi-norm of a system $G$ as
$$
\|G\|_H = \sup\left\{\frac{\|G\ast u\|_{L^2(T,\infty)}}{\|u\|_{L^2[0,T]}}: T > 0, u \in L^2[0,\infty)\right\},
$$
then $\|G\|_H = \sigma_{H1}$ for the maximum Hankel singular value. But with this formulation, it
is immediate  that $\|G\|_H \leq \|G\|_\infty$, when we recall that $\|G\|_\infty$ is the $L^2$-operator norm.

Therefore we get
for a system with direct transmission
$$
\|G\|_{\rm pk\_gn}
\leq \|G-D\|_{\rm pk\_gn} + \normi{D}_\infty
\leq 2p^{1/2}n\sigma_{H1} + p^{1/2} \overline{\sigma}(D) \leq (2n+1)p^{1/2} \|G\|_\infty
$$
using the fact that  $\sigma_{H1} \leq \|G\|_\infty$ and $\overline{\sigma}(D) \leq \|G\|_\infty$.
Here $\normi{D}_\infty$ is the maximum row sum norm $\max_i \sum_j |d_{ij}|$, which is the $\ell_\infty-\ell_\infty$ induced
matrix norm satisfying $\normi{D}_\infty\leq p^{1/2} \overline{\sigma}(D)$. 
Altogether, we have proved the following estimates for the $H_\infty$- and peak-gain norms stated in \cite{mixedApkarianNoll}:
\begin{theorem}
    Let $G$ be a stable real-rational system with $n$ poles, $p$ inputs and $m$ outputs. Then
    $$
    m^{-1/2} \|G\|_\infty \leq \|G\|_{\rm pk\_gn} \leq (2n+1)p^{1/2}\|G\|_\infty.
    $$
\end{theorem}
A large variety of synthesis experiments based on the peak-gain norm $\|G\|_{\rm pk\_gn}$ has been presented
in \cite{mixedApkarianNoll}, so that our experiments here may focus on $L_1$-disturbances.

\subsection{Noise as perturbation}
\label{sect_noise}
In this section we consider the case $w\in L^2$, $G\ast w\in L^\infty$, where we can rely on \cite{induced_norms}.
Consider for instance
$\|G\|_{(\infty,2) (2,2)} = \lambda_{\max}(C QC^T)$, where $Q \succeq 0$ is the unique solution of the Lyapunov equation
$AQ+QA^T + BB^T=0$. This norm can be optimized directly using a technique similar to
\cite{dao}.

\section*{Organization of the applications}

In the following, we present applications illustrating the use of the Kreiss system norm for both analysis and feedback control design. The methods are general and can be applied to a wide range of nonlinear controlled systems. It should be emphasized that in all tests the results are certified a posteriori, since Kreiss norm optimization itself relies on a heuristic.
The material is organized as follows:

\begin{description}
  \item[Section~\ref{sect-LCA}] 
  Kreiss norm minimization is applied to control nonlinear dynamics with periodic orbits:
  \begin{itemize}
    \item a two-dimensional limit cycle (Section~\ref{sect-brunton}),
    \item a four-dimensional periodic orbit (Section~\ref{sect-brunton2}).
  \end{itemize}

  \item[Section~\ref{sect-ChaosFP}] 
  Nonlinear regimes in the Lorenz model are considered, including chaotic  (Section~\ref{sect-lorenz1}) and fixed points dynamics (Section~\ref{sect-fiedP}). These are investigated using:
  \begin{itemize}
    \item the QC approach (Sections~\ref{sect-QC1}, \ref{sect-QC2}),
    \item Kreiss norm minimization (Sections~\ref{sect-Kreiss12}, \ref{sect-Kreiss22}).
  \end{itemize}
\end{description}

\section{Applications to nonlinear dynamics with periodic orbit attractors \label{sect-LCA}}

\subsection{Study of $2$nd-order dynamics with limit cycle  attractor\label{sect-brunton}}
We start with
the model of Brunton and Noack \cite{brunton2015closed}, which is a low-order illustration of a dynamic mechanisms 
known in oscillator flow, 
observed for instance on a larger scale in Navier-Stokes equations, see also \cite{Taira2017AIAA_Overview}.
Examples of this type include fluid flow around a cavity or a cylinder \cite{leclercq2019linear,illingworth_morgans_rowley_2012}. 
The model is of the form 
\begin{eqnarray}
\label{eq-Brunton1}
\left\{ \begin{array}{lll} \dot x &=& \begin{bmatrix}\sigma_u & -\omega_ u\\\omega_ u & \sigma_u 
\end{bmatrix} x  + B_w w + B u, \;\; x \in \mathbb R^n,\;  n = 2\\
w & = & \phi(x)  \\ 
y & = & C x  \end{array}\right. \,,
\end{eqnarray}
with $B_w := I$, $B:= [0 \,\, g]^T$ $C:=[0 \,\, 1]$, 
$$
\phi(x):=  \alpha_u \| x \|^2 \begin{bmatrix}-\beta_u & -\gamma_ u\\\gamma_ u & -\beta_u  
\end{bmatrix} x \,,
$$
and  $\alpha_u, \beta_u > 0$. 
Signals $u$ and $y$ are control input and measured output, respectively. It is easy to verify that the triple $(A,B,C)$ is stabilizable and detectable. 

Unlike  transitional amplifier  flows, oscillator flows are 
characterized by an unstable fixed point at the origin and a globally attractive limit cycle,
here with radius $\sqrt{\sigma_u/\alpha_u \beta_u }$. This is shown in Fig. \ref{fig-BruntonOL}  
for two initial conditions inside and outside the asymptotic limit cycle for data $\alpha_u=1$, $\beta_u=1$, $\omega_u=1$, $\gamma_u=0$,  $\sigma_u=0.1$ and $g=1$. 

The goal is to compute a feedback controller $u = K(s) y$ with two main design requirements. 
Firstly, $K$ has to stabilize the origin, often called the base flow. 
Secondly, trajectories trapped in the limit cycle should be driven back to the origin with limited oscillations.  
Additional insight into this model in terms of fluid flow interpretation can be found in \cite{brunton2015closed}. 

In order to
mitigate the effects of nonlinearity, we minimize the Kreiss system norm in closed loop. 
This leads to the following min-max constrained program
\begin{align}
\label{eq-synth1}
\begin{array}{ll}
\displaystyle\mbox{minimize} & \displaystyle\max_{\delta \in [-1,1]} \left\| J^T \left( sI- \left(\textstyle \frac{1-\delta}{1+\delta} A_{cl}(K)-I \right)\right)^{-1} J \right\|_\infty\\
\mbox{subject to}& K \mbox{ robustly stabilizing}, \, K \in \mathscr K \\
& \alpha(A_{cl}(K)) \leq -\eta \\
& \|W(s)GK(I+GK)^{-1}\|_\infty \leq  1 \,.
\end{array}
\end{align}
Here $K\in \mathscr K$ means that the controller has a prescribed structure,  which could be a PID, 
observed-based or low-order controller, a decentralized controller,
as well as any control architecture assembling simple control components. 
The robust stability constraint on $K$ in (\ref{eq-synth1})  demands stability of the entire set of matrices 
$\left\{\frac{1-\delta}{1+\delta} A_{cl}-I: \;\delta \in [-1,1]\right\}$,  and in particular, for $\delta=0$ that of
$A_{cl}(K)$. Matrix $J$ is a restriction matrix to the space of physical plant states since transient amplification of controller states is not relevant.  We have $J:=I_n$ with $n$ the plant state dimension for a static feedback controller and $J:= [I_n,\,0_{n\times n_K}]^T$ for an $n_K$-order  output-feedback controller (see also  example \ref{example7}). 
In terms of dimension, we have  $n=2$, $n_K = 3$ when $K$ is a  $3$rd-order controller, $p_z = 2$ since $z=x$, $m_w=2$ for $w = \phi(z)$, and $p=m=1$. The overall state dimension is therefore $N=n+n_K+n_W =7$ in closed loop for a $3$rd-order controller and $2$nd-order filter $W$. The latter is used to weigh the noise to measurement transfer $T:= GK(I+GK)^{-1}$. 

The notation $\alpha(\cdot)$ refers to the spectral abscissa, and 
the constraint $\alpha(A_{cl})\leq -\eta$ in (\ref{eq-synth1}) therefore
imposes a convergence rate to the origin for the linear dynamics in closed loop. In our experiments we have chosen $\eta = 0.1$.

These constraints  are readily implemented from the closed-loop nonlinear system: 
\begin{eqnarray}
\label{eq-CL1}
\begin{array}{lll}
\dot x_{cl} &=& A_{cl}  x_{cl} +  B_{w,cl} \phi_{cl}(x_{cl}),\quad  x_{cl}:= [x^T, x_K^T]^T\,, \\
    \end{array}
\end{eqnarray}
where  
\begin{eqnarray} 
\label{eq-CL2}
A_{cl} := \begin{bmatrix} A+B D_K C & B C_K \\ B_K C & A_K  \end{bmatrix}, \, 
\phi_{cl}(x_{cl}): = \phi(x), \, B_{w,cl}:= J=[I_2 ,0_{2 \times n_K}]^T \,,\end{eqnarray}
and where the controller dynamics are
\begin{eqnarray}
\label{eq-K1}
 \left\{   \begin{array}{lll}
\dot x_K \!\!&\!\!=&\!\! A_K x_K + B_K y, \quad x_K \in \mathbb R^{n_K} \\
 u \!\!&\!\! = &\!\! C_K x_K + D_K y\, \\
    \end{array} \right. \,.
\end{eqnarray}
We exclude high-gain feedback in the high-frequency range by adding a constraint on the complementary sensitivity function $\|WT\|_\infty\leq 1$, where $W$ is a high-pass weighting filter 
$W(s):= (1\text{e}06 s^2 + 1\text{e}04 s + 24.99)/(s^2 + 10000 s + 2.5\text{e}07)$.

Program (\ref{eq-synth1}) was solved for controller orders: $0$, $1$ and $3$. All controllers achieve 
nearly the same Kreiss norm of $1.005$, but differ in terms of the remaining performance/robustness constraints. 
This is seen by plotting  transient amplifications versus time  in Fig. \ref{fig-BruntonTA}.
Peak values $\mathcal M_0(J^T (sI-A_{cl})^{-1} J)$ are all close to $1.10$ with $\|J^T J\| = 1$ as lower bound, see Proposition \ref{bound}. 

The static controller $K = -0.20$ gives a spectral abscissa of $\alpha(A_{cl})= -1.9899\text{e-}04$ with badly damped modes and a strong roll-off violation of $\|WT\|_\infty = 20.03 $. The $1$st-order controller $K(s)= (0.001071 s - 2.247)/(s + 1.483)$  meets the roll-off constraint and achieves a decay rate constraints $\alpha(A_{cl})= -0.393$.
As expected, the $3$rd-order controller $K(s)=(-0.008068 s^3 - 6.391 s^2 + 83.2 s - 1673)/(s^3 + 27.97 s^2 + 252.8 s + 1333)$ provides the best results in terms of decay rate $\alpha(A_{cl})=-0.811$. 

Simulations in closed loop for identical initial conditions are given in Fig. 
\ref{fig-BruntonCL}. Controllers are switched on at $t=50$ seconds when the limit cycle is well engaged. The static controller leads to a spiral trajectory barely converging to the origin, 
a stint which   is overcome by increasing the controller order. 

Global stability of the origin is established in appendix \ref{appendix-A}.

\begin{figure}[t]
\includegraphics[height=0.3\textheight]{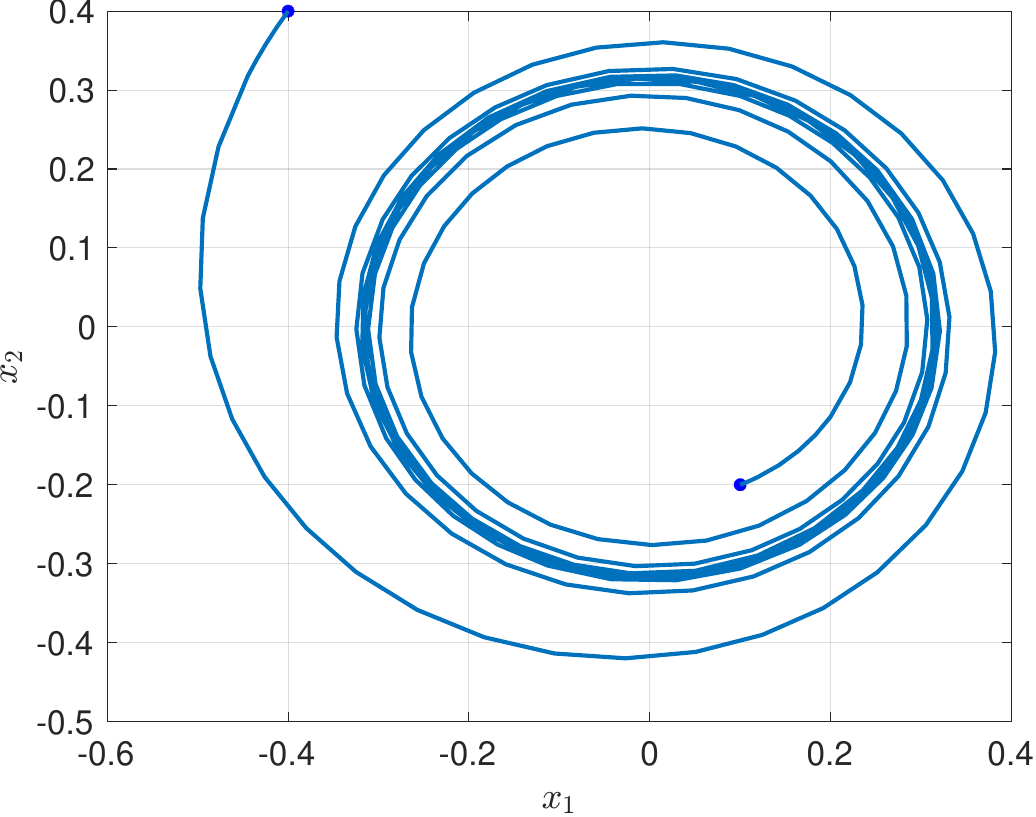}
\centering
\caption{Limit cycle attractor of Brunton and Noack  model\label{fig-BruntonOL} }
\end{figure}

\begin{figure}[!htbp]
\centering
\includegraphics[height=0.25\textheight, width = 0.30\textwidth]{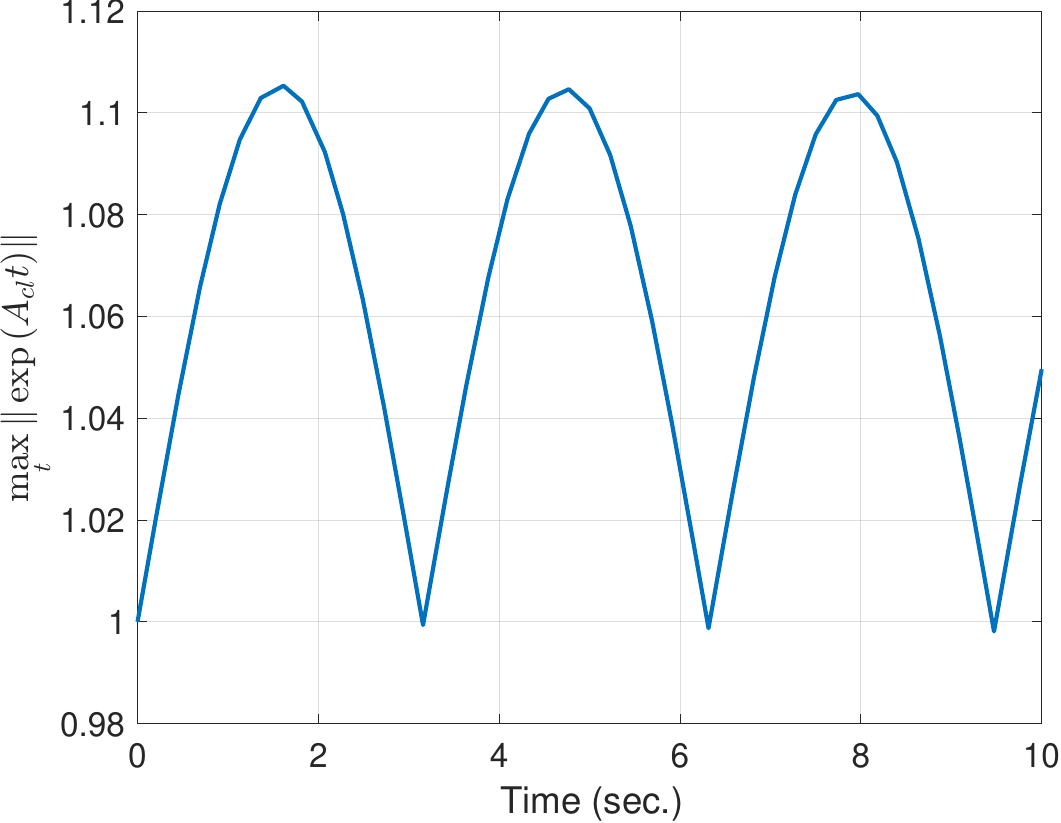}
\includegraphics[height=0.25\textheight, width = 0.30\textwidth]{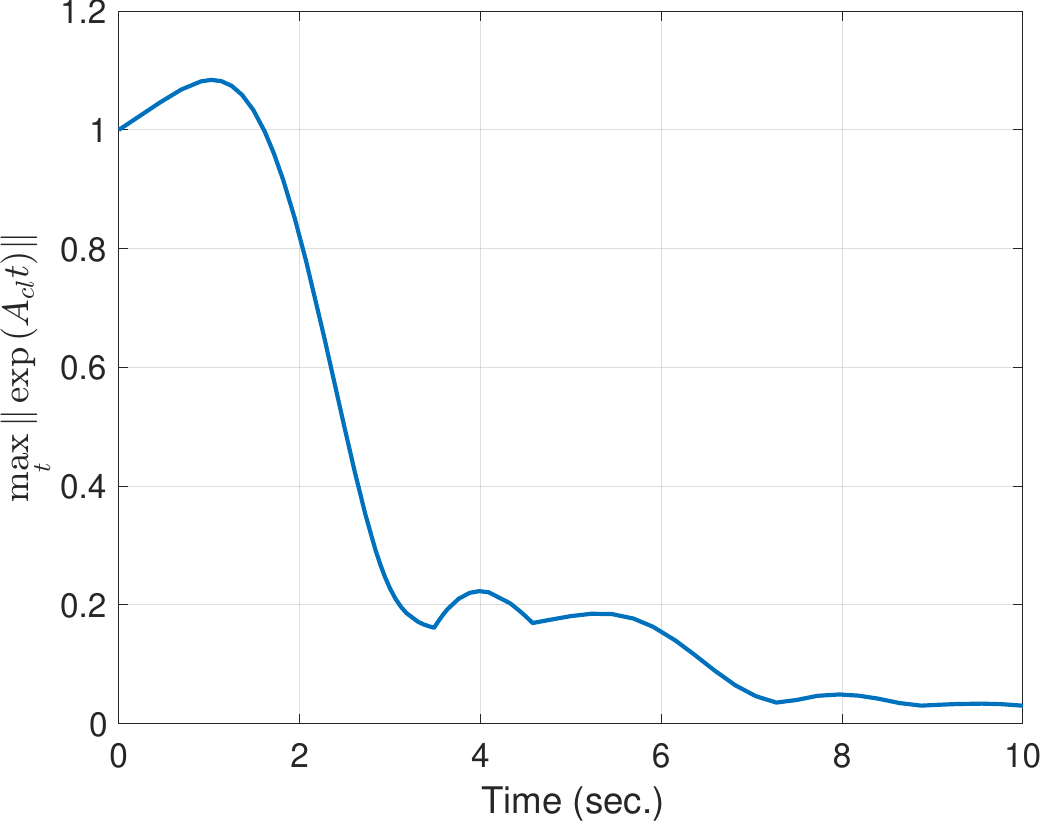} 
\includegraphics[height=0.25\textheight, width = 0.30\textwidth]{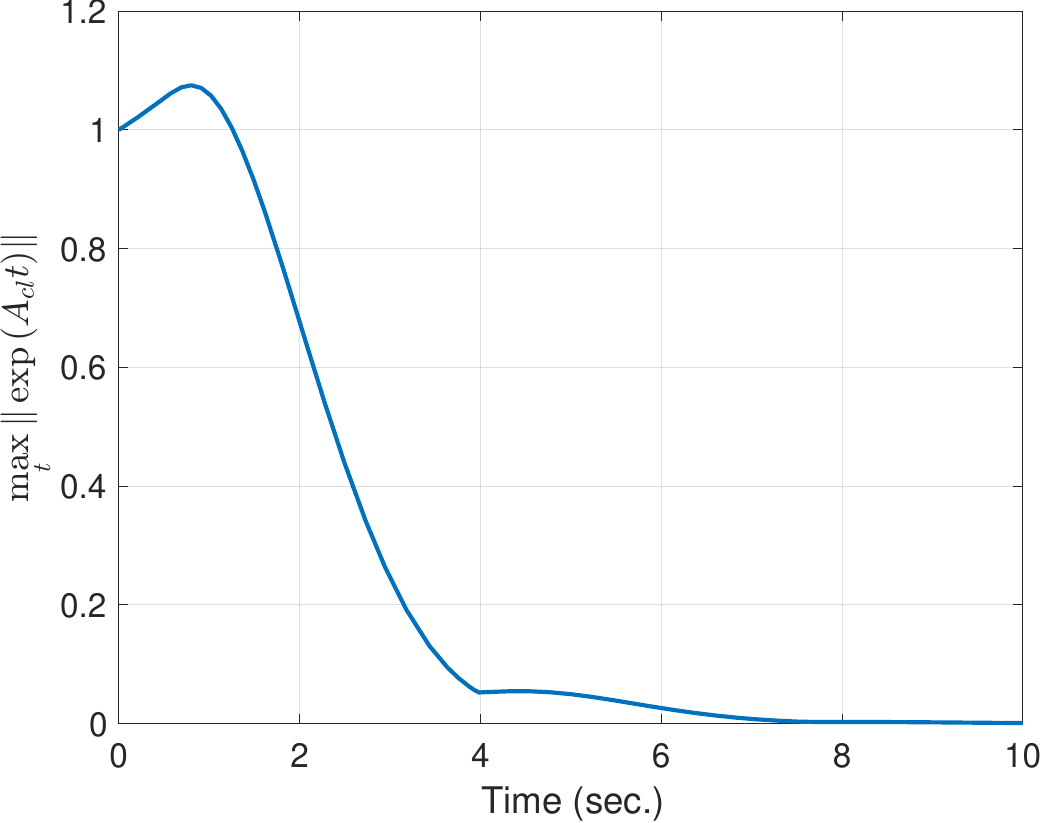}
\caption{From left to right, time evolution of transient amplification for static, $1$st-order and $3$rd-order controller\label{fig-BruntonTA}}
\end{figure}

\begin{figure}[!htbp]
\centering
\includegraphics[height=0.27\textheight, width = 0.325\textwidth]{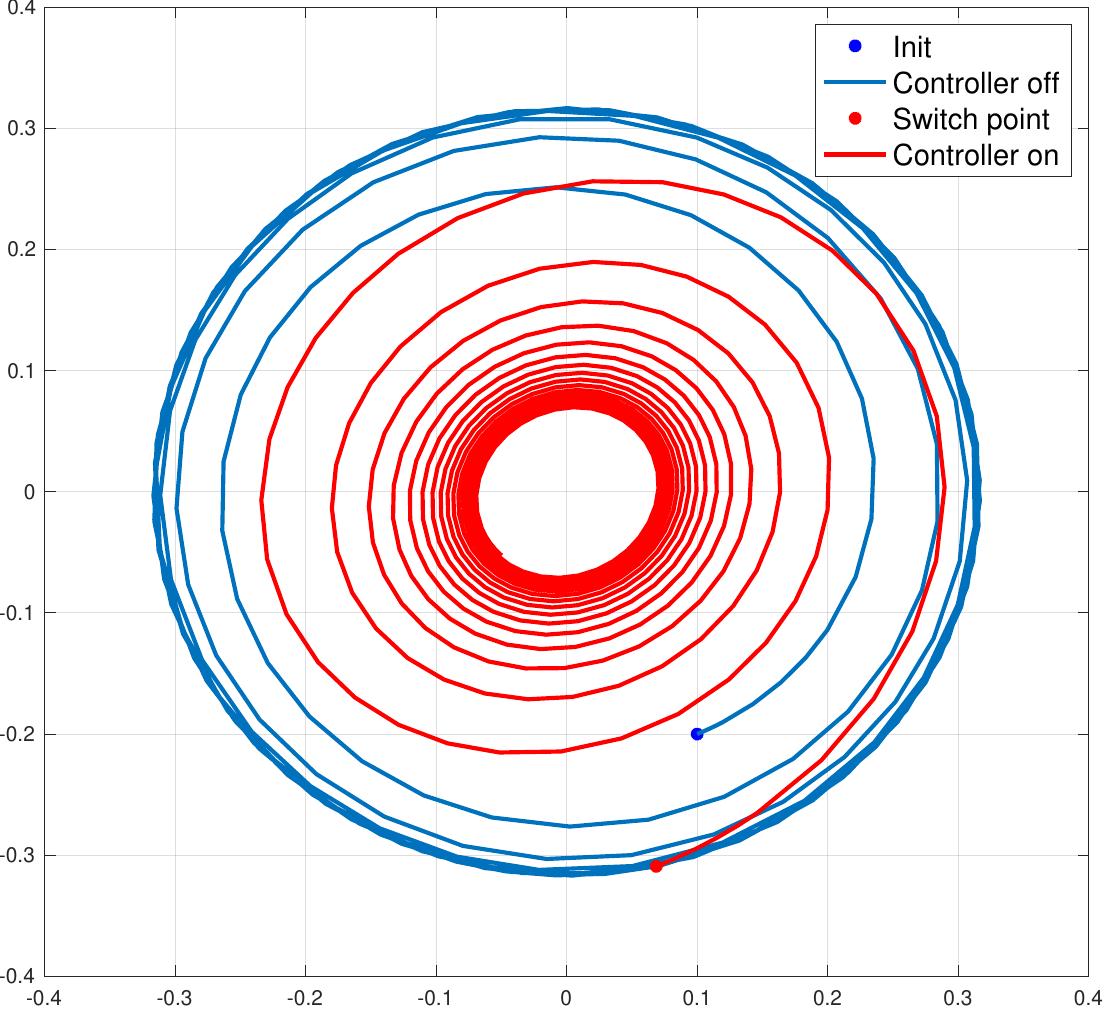}
\includegraphics[height=0.27\textheight, width = 0.325\textwidth]{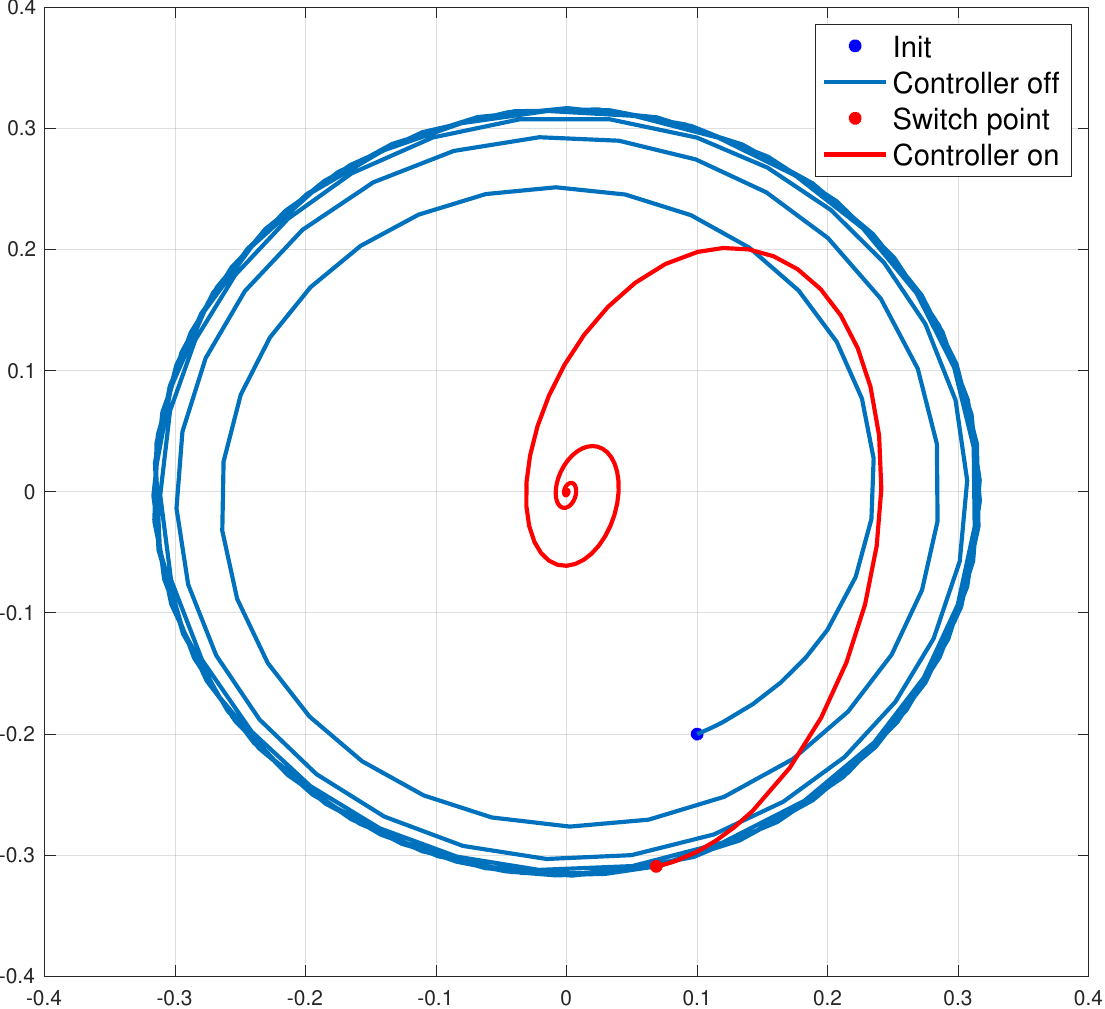} 
\includegraphics[height=0.27\textheight, width = 0.325\textwidth]{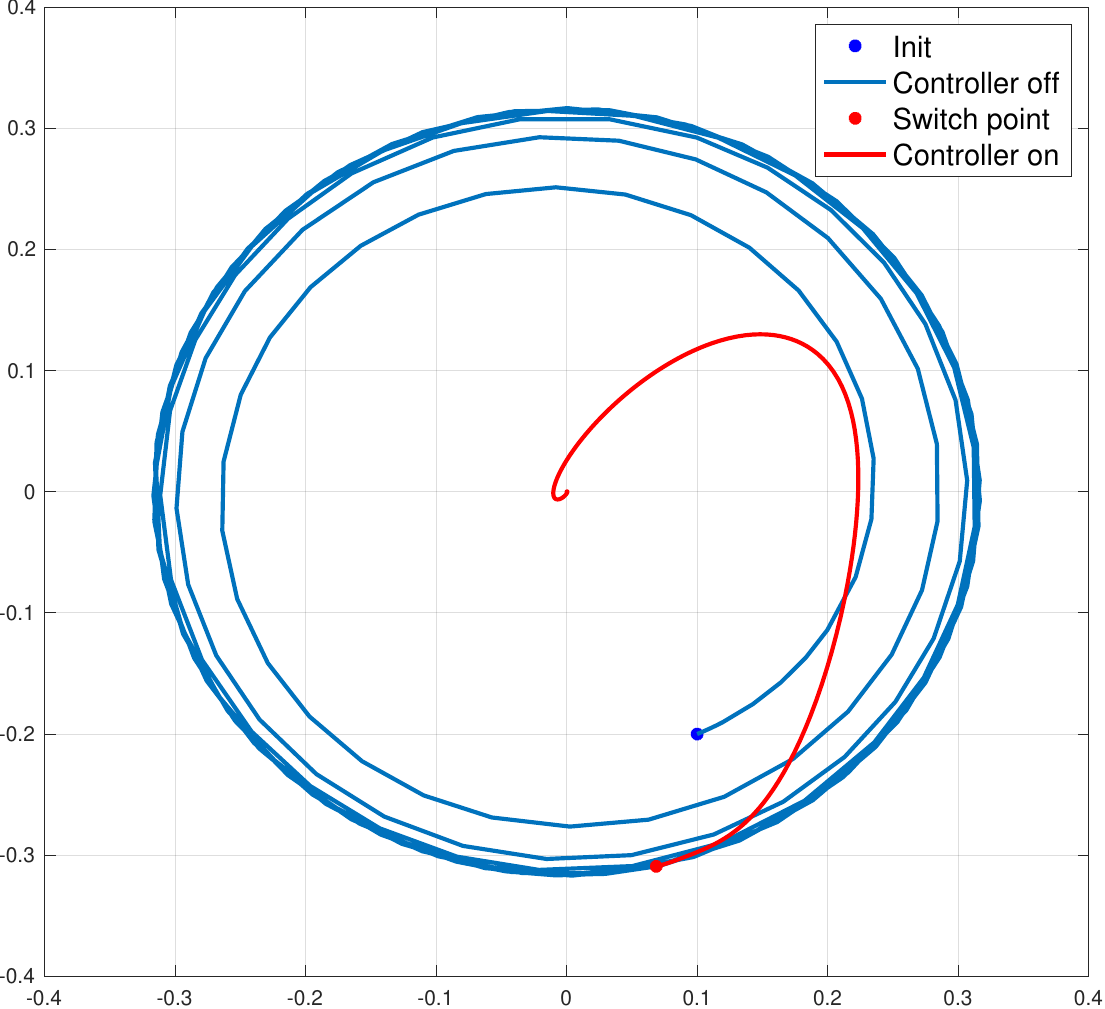}
\caption{From left to right, closed-loop free responses  for static, $1$st-order and $3$rd-order controller\label{fig-BruntonCL}}
\end{figure}

\subsection{Study of fourth-order dynamics with $4D$ periodic orbit attractor. \label{sect-brunton2}}
The fourth-order model of  Brunton and Noack is described as

\begin{eqnarray}
\label{eq-Brunton2}
\left\{ \begin{array}{lll} \dot x &=& A x  + B_w w + B u,  \;\; x \in \mathbb R^{n}, \; n= 4 \\
w & = & \phi(x)  \\ 
y & = & C x  \end{array}\right. \,,
\end{eqnarray}
with 
$$
\phi(x):= (\alpha_u(x_1^2+x_2^2)A_5 + \alpha_a(x_3^2+x_4^2)A_6))x
$$
where
$$
A:= \diag\left(\begin{bmatrix}\sigma_u &  -\omega_u \\ \omega_u &  \sigma_u \end{bmatrix}, \begin{bmatrix}\sigma_a & -\omega_a \\ \omega_a &  \sigma_a \end{bmatrix}\right), \; B_w := I, \, B:= [0\; g\; 0\; g]^T,\, C:= [1\; 0\; 1\; 0],
$$
$$
A_5:= \diag \left(\begin{bmatrix}-\beta_{uu} &  -\gamma_{uu} \\ \gamma_{uu} &  -\beta_{uu} \end{bmatrix}, \begin{bmatrix}-\beta_{au} &  -\gamma_{au} \\ \gamma_{au} &  -\beta_{au} \end{bmatrix}\right), \;
A_6:= \diag\left(\begin{bmatrix}-\beta_{ua} &  -\gamma_{ua} \\ \gamma_{ua} &  -\beta_{ua} \end{bmatrix}, \begin{bmatrix}-\beta_{aa} &  -\gamma_{aa} \\ \gamma_{aa} &  -\beta_{aa} \end{bmatrix}\right),
$$
with data given in \cite{brunton2015closed}. 
In this application, we have  $n=4$,  $n_K= 1$ for a $1$st-order controller, $p_z = 4$ since $z=x$, $m_w=4$ for $w = \phi(z)$, and $p= m=1$. The overall state  dimension is therefore $n+n_K+n_W = 4 + 1 +2 = 7$ for a $2$nd-order filter $W$.  

The open-loop dynamics are characterized by an unstable fixed point at the origin and an attractive $4D$ periodic orbit. 
A $1$st-order controller is computed to minimize the Kreiss norm as in program (\ref{eq-synth1}). The roll-off filter $W$ is unchanged. 
The optimal controller $K(s):=(0.03538 s - 0.5306)/(s + 0.667)$ achieves a Kreiss norm of $1.004$ with decay rate and roll-off constraints all met.  

Despite the apparent complexity of the dynamics, experience shows that it is possible to bring points of the periodic orbit back to the origin with a fairly large class of {\it linear} controllers. As an instance, a controller designed using a mixed-sensitivity approach \cite[p. 141]{ZDG:96} with weight $W_1:=\frac{0.001 s + 5}{s + 0.05}$ for $S$ and $W$ as above for $T$ also drives
points from the periodic orbit to the origin. A $1$st-order controller of this type was obtained as $K(s):=(34.31 s + 168.1)/(s + 32.47)$ with a closed-loop Kreiss constant of $1.54$. Simulations show that closed-loop trajectories undergo large deviations before heading back to the origin. See Fig. \ref{fig-BruntonCL4}.  This remains risky, as attractors when 
still present may capture trajectories. The controller based on the Kreiss norm corrects such undesirable transients as 
corroborated in Figs. \ref{fig-BruntonCL4} and   \ref{fig-Transient4}, where worst-case transients have been plotted. Finally, all controllers globally stabilize the origin and this can be established as was done for the $2$nd-order system in appendix \ref{appendix-A}.

\begin{figure}[!htbp]
\centering
\includegraphics[height=0.35\textheight, width = 1\textwidth]{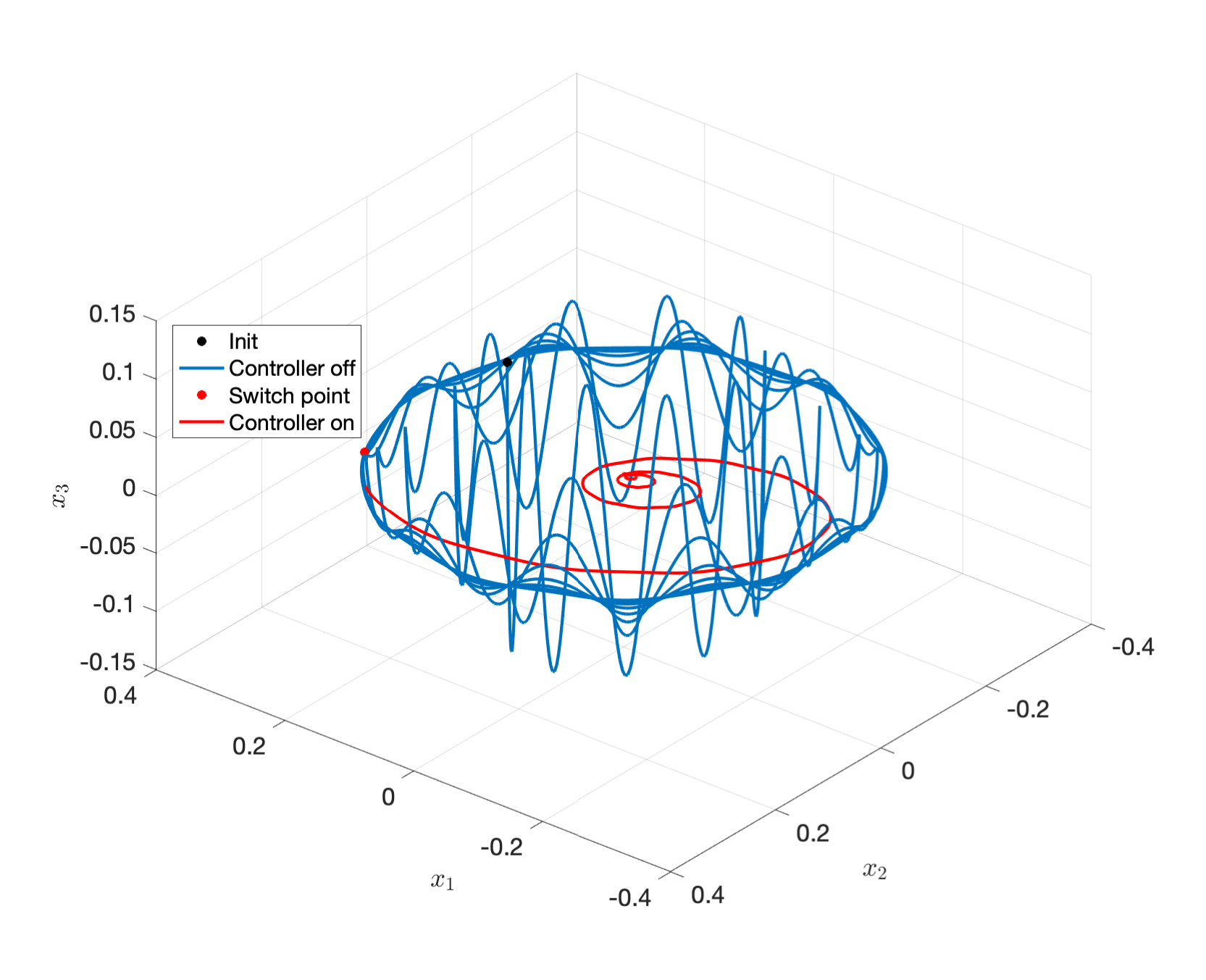}
\includegraphics[height=0.35\textheight, width = 1\textwidth]{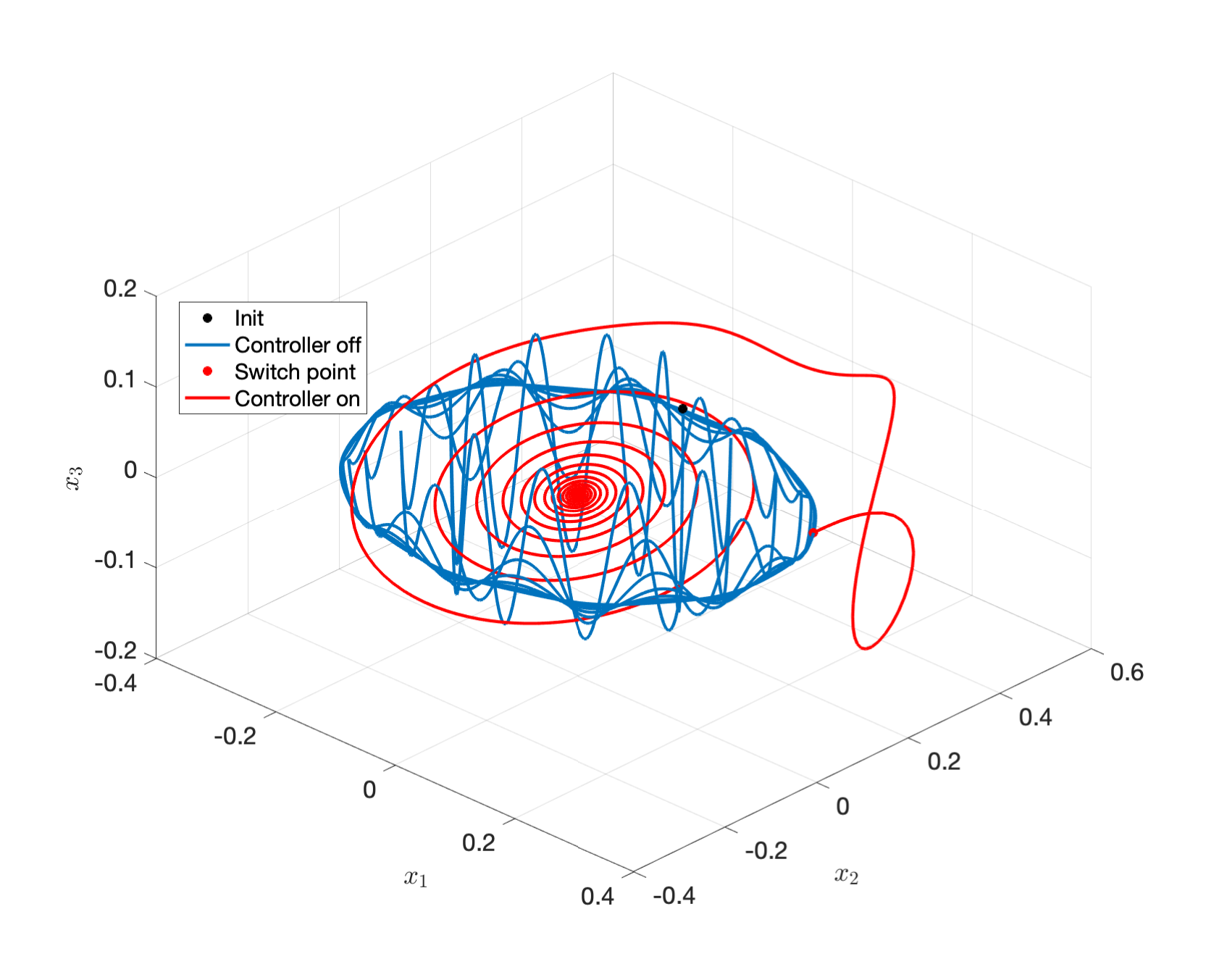}
\caption{Brunton and Noack model. Free open- and closed-loop responses projected in $(x_1,x_2,x_3)$-space with $1$st-order  controllers. Top: Kreiss controller. Bottom: mixed-sensitivity controller. \label{fig-BruntonCL4}}
\end{figure}

\begin{figure}[!htbp]
\centering
\includegraphics[height=0.25\textheight, width = 0.45\textwidth]{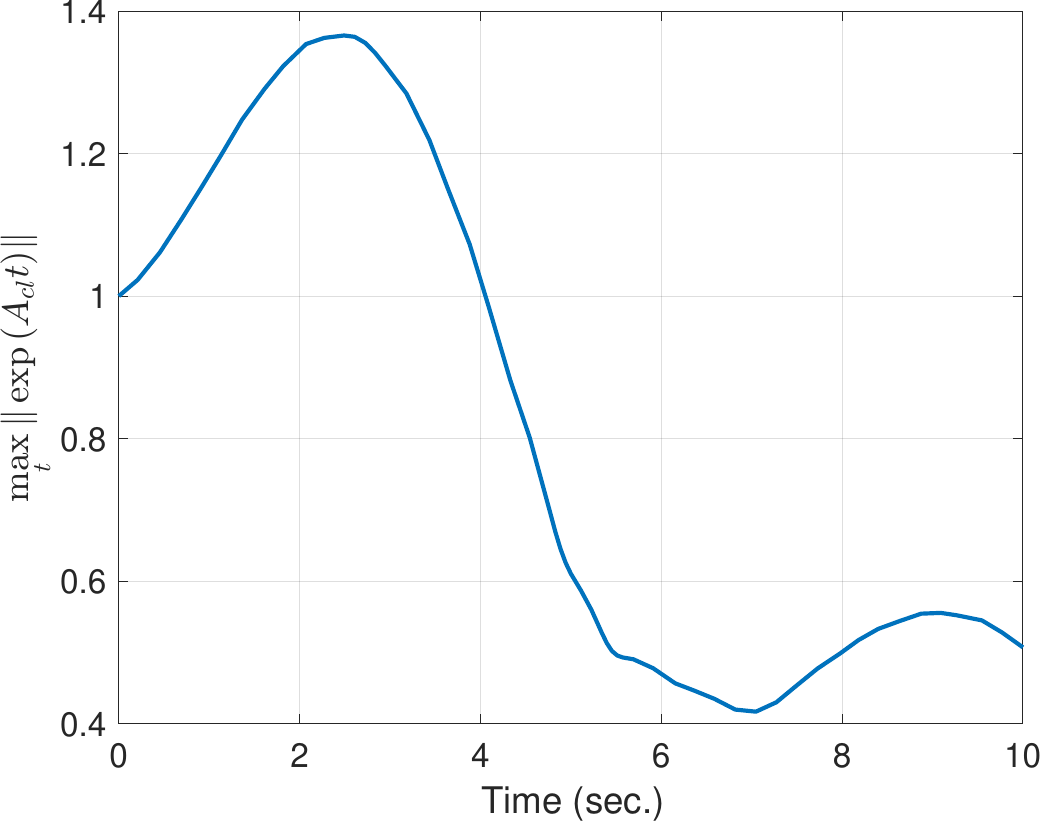}
\includegraphics[height=0.25\textheight, width = 0.45\textwidth]{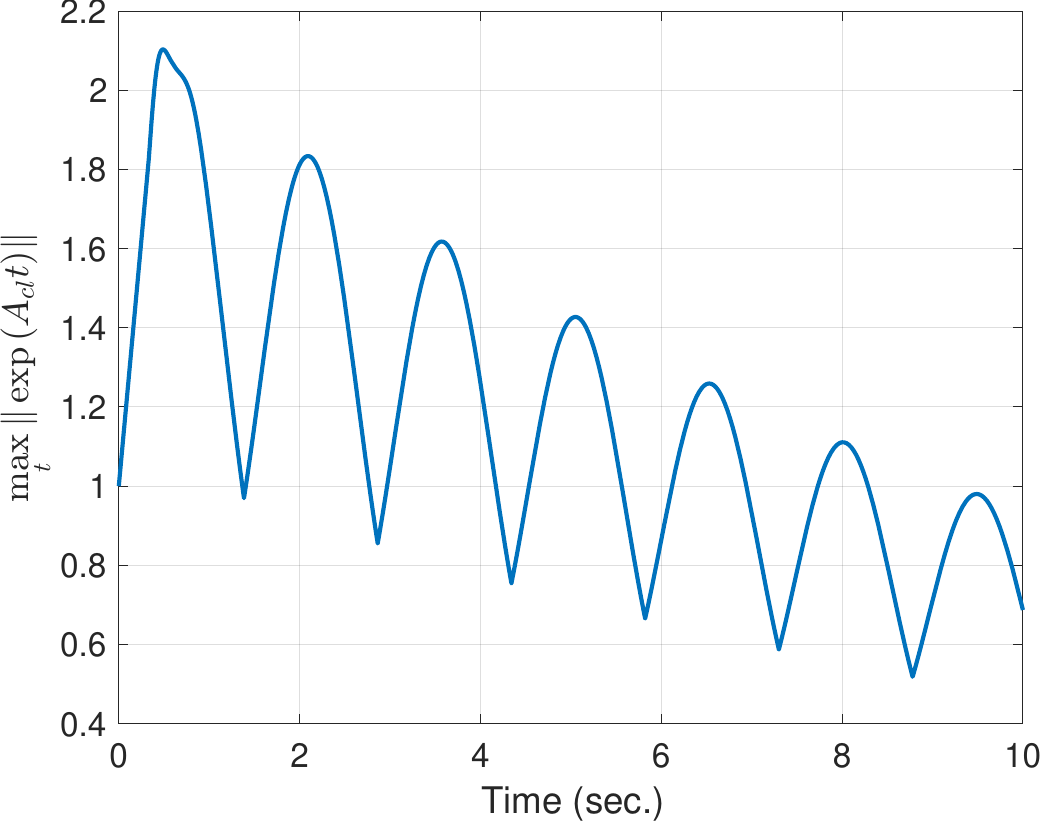}
\caption{Brunton and Noack model. Transient amplifications. Left: Kreiss controller. Right: mixed-sensitivity controller. \label{fig-Transient4}}
\end{figure}

\section{Applications to nonlinear dynamics with chaos and fixed points \label{sect-ChaosFP}}
In this section, we consider suppression of undesirable nonlinear regimes such as chaos and fixed points for the Lorenz model. 

\subsection{Study of the Lorenz system with chaotic attractor\label{sect-lorenz1}}
The Lorenz system \cite{lorenz1963deterministic} has three coupled first-order nonlinear differential equations

\begin{eqnarray}
\label{eq-lorenz1}
 \left\{   \begin{array}{lll}
\dot x_1 &=& p(x_2-x_1) \\
\dot x_2& = &Rx_1 - x_2 -x_1x_3 \\
\dot x_3 &= & - bx_3 + x_1x_2  \,,
    \end{array} \right.
\end{eqnarray}
where $p$, $R$, and $b$ are given parameters. In this study, we will use $p=10$ and $b=1$, while $R$ will be 
varied to illustrate different nonlinear asymptotic regimes. To begin with, we take $R=28$, where the Lorenz model has three unstable fixed points with coordinates  
\begin{equation} \label{eq-fixedPoints}
(0,0,0),\; (\sqrt{R-1},\sqrt{R-1},R-1),\; (-\sqrt{R-1},-\sqrt{R-1},R-1)\,. \end{equation}
For any initial condition $x(0)=x_0$, a repelling effect of these fixed points is observed and trajectories are quickly captured by a chaotic attractor of double-scroll type, shown in Fig. \ref{fig-scroll}. 

\begin{figure}[t]
\includegraphics[height=0.3\textheight]{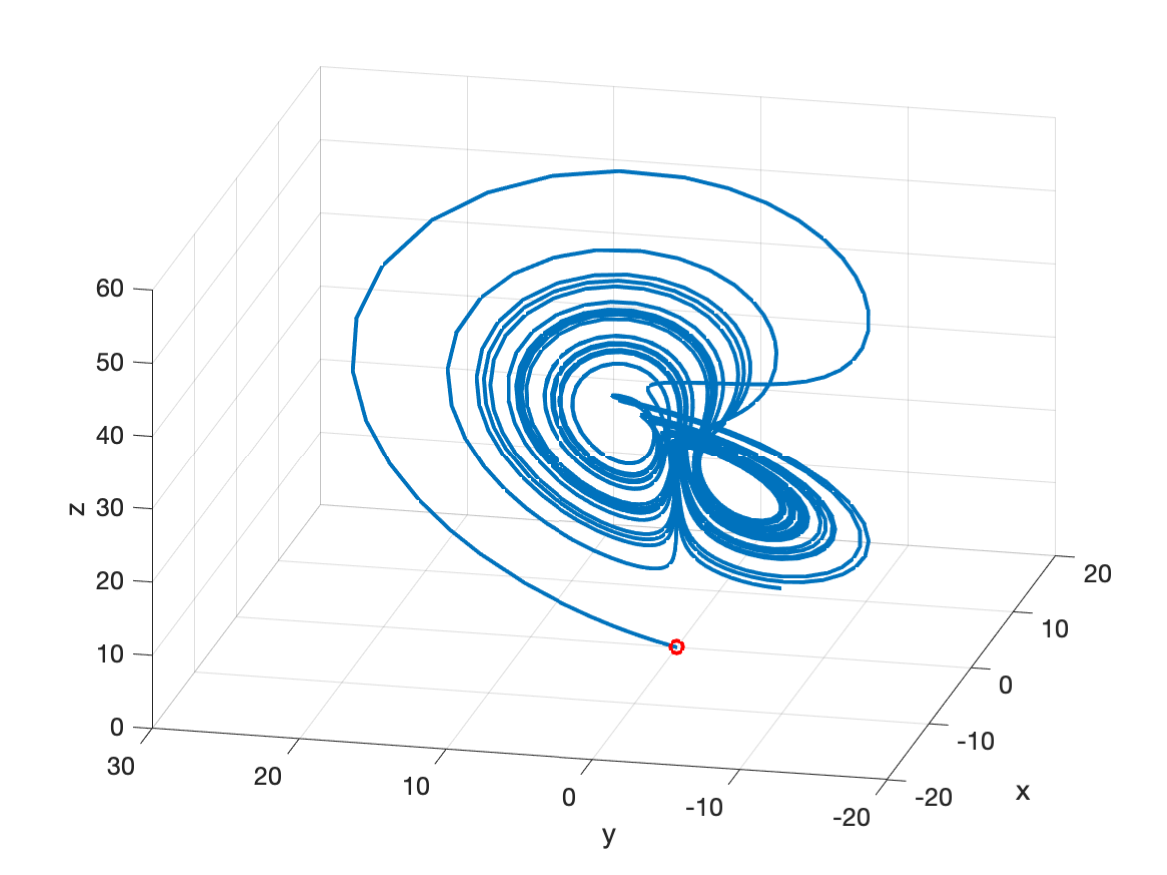}
\centering
\caption{Double-Scroll chaotic attractor of the Lorenz model\label{fig-scroll} }
Response to state initial condition.
\end{figure}

Our feedback goal is  therefore suppression of the chaotic attractor and stabilization of the origin through various feedback control strategies. We complement (\ref{eq-lorenz1}) by adding actuation and sensing,
letting $B=[0,\,1,\,0]^T$ and discussing several cases $C$, where $(A,B,C)$ is stabilizable and detectable. 
The Lorenz model is then rewritten as 
\begin{eqnarray}
\label{eq-lorenz2}
 \left\{   \begin{array}{lll}
\dot x &=& A x + B_w w + B u,\;\;  x \in \mathbb R^{n}, \; n= 3  \\
w & = & \phi(x) \\
y & = & C x \,,
    \end{array} \right.
\end{eqnarray}
where $u$ is the control input, $y$ the measurement output. Matrix $A$ collects the linear terms in (\ref{eq-lorenz1}),   
$\phi(x):=[-x_1x_3,\,x_1x_2]^T$ the nonlinearity, and $B_w:=[0_{2\times 1},I_2]^T$.
As observed before, the origin is unstable in the absence of feedback. 

In this example,  $n = 3$, $p_z = 3$ since $z = x$, $m_w = 2$ for $w = \phi(z)$  and $m=1$. The number of measurements $p$ depends on the control strategy used. We get $p =1$ when a single measurement is used and $p=3$ for full state measurement $y=x$. As before, we also investigate various controller orders $n_K$.

When a linear feedback controller $u = K(s) y$ is used with
\begin{eqnarray}
\label{eq-K}
 \left\{   \begin{array}{lll}
\dot x_K \!\!&\!\!=&\!\! A_K x_K + B_K y, \quad x_K \in \mathbb R^{n_K} \\
 u \!\!&\!\! = &\!\! C_K x_K + D_K y\,, \\
    \end{array} \right.
\end{eqnarray}
the Lorenz model in closed loop becomes: 
\begin{eqnarray}
\label{eq-CL}
\begin{array}{lll}
\dot x_{cl} &=& A_{cl}  x_{cl} +  B_{w,cl} \phi_{cl}(x_{cl}),\quad  x_{cl}:= [x^T, x_K^T]^T\,, \\
    \end{array}
\end{eqnarray}
where  
\begin{eqnarray} 
\label{eq-CLlorenz}
A_{cl} := \begin{bmatrix} A+B D_K C & B C_K \\ B_K C & A_K  \end{bmatrix}, \, 
\phi_{cl}(x_{cl}): = \phi(x), \, B_{w,cl}:=[0_{2\times 1},I_2,0_{2 \times n_K}]^T \,.\end{eqnarray}

\subsubsection{Chaos dynamics: design with the QC approach \label{sect-QC1}}
Here we
assess the stability properties of the closed loop (\ref{eq-CL}) using the Lyapunov Quadratic Constraints (QC) approach 
of \cite{mushtaq2022feedback,kalur2021nonlinear,liu2020input}.

A particularity of the Lorenz system is the so-called {\em lossless property}
\begin{equation}
\label{eq-lossless}
x_{cl}^T B_{w,cl} w = 0 \, \mbox{ for all }\,  x_{cl},  w = \phi(x) \,,
\end{equation}
which holds globally in  state space. 
The QC approach to stability analysis now relies on the existence of a Lyapunov function $V(x_{cl}) = x_{cl}^T X_{cl} x_{cl}$, with $X_{cl}$ a positive definite matrix, such that 
$$\dot V(x_{cl}) \leq -\epsilon V(x_{cl}), \epsilon > 0$$ for all $x_{cl}$, $w$ such that the quadratic constraint in  (\ref{eq-lossless}), when disregarding $w = \phi(x)$, holds. This is clearly a sufficient possibly conservative condition because of the chosen quadratic in $V(x_{cl})$, and also because the specific dependence of $w$ on the states $x$ is ignored.  Using a $S$-procedure argument \cite{kalur2021nonlinear,boyd1994linear} to aggregate the lossless constraint (\ref{eq-lossless}),  this is rewritten as 
$$\dot V(x_{cl}) + \mu_0 x_{cl}^T B_{w,cl} w \leq -\epsilon V(x_{cl})$$ for all $x_{cl}, w$, where $\mu_0$ is a $S$-procedure parameter (sometimes called a Lagrange multiplier), which here is unsigned, as the constraint (\ref{eq-lossless}) is an equality.  The following equivalent  matrix inequality constraints are obtained:   

\begin{equation}
\label{eq-LMI}
\begin{bmatrix}
    A_{cl}^T X_{cl} + X_{cl}  A_{cl} + \epsilon X_{cl} & X_{cl} B_{w,cl} + \mu_0  B_{w,cl}  \\
    B_{w,cl}^T X_{cl} +  \mu_0  B_{w,cl}^T & 0
\end{bmatrix} \preceq 0, \; X_{cl} \succ 0 \,.
\end{equation}

We have the following:

\begin{theorem}\label{theoLMI}
There exist a linear time-invariant controller {\rm (\ref{eq-K})}  such that the sufficient global stability conditions {\rm (\ref{eq-LMI})} hold, if and only if there exist solutions $X=X^T$ and $Y=Y^T$ in $\mathbb R^{(n-n_\phi)\times (n-n_\phi)}$ to the 
following LMIs:
\begin{equation}\begin{split}\label{eq-LMIXY}
    N_C^T\left( A^T \begin{bmatrix}X & 0 \\0 & I \end{bmatrix}+  \begin{bmatrix}X & 0 \\0 & I \end{bmatrix} A + \epsilon \begin{bmatrix}X & 0 \\0 & I \end{bmatrix} \right)N_C \prec 0 \\
    N_B^T\left( A \begin{bmatrix}Y & 0 \\0 & I \end{bmatrix}+  \begin{bmatrix}Y & 0 \\0 & I \end{bmatrix} A^T + \epsilon \begin{bmatrix}Y & 0 \\0 & I \end{bmatrix} \right)N_B \prec 0 \\
    \begin{bmatrix}X & 0 &I & 0 \\0 & I & 0 & I \\I & 0 & Y&0 \\0 & I &0 & I \end{bmatrix} \succeq 0 \,,
\end{split}\end{equation}
where $N_C$ and $N_B$ are bases of the null space of $C$ and $B^T$, respectively.

Moreover, the controller order is determined by the rank of 
$I_{n-n_\phi} - XY$
with $n_\phi$ the vector dimension of the nonlinearity. 
\end{theorem}
\begin{proof}
 See appendix \ref{appendix-C}.    
\end{proof}

Note that Theorem \ref{theoLMI} applies to any nonlinear system with similar structure for which (\ref{eq-lossless}) holds.  

For the Lorenz model, we have a loss of rank of at least $n_\phi = 2$, the vector dimension of the nonlinearity. 
For the QC approach this means that controllers  can have order at most $n-n_\phi = 1$. For problems with nonlinearity of dimension $n$, the plant order, 
only static output feedback controllers can be computed. In that case the BMI (\ref{eq-BMI}) reduces to the LMI feasibility problem:
$$(A+BKC)^T + (A+BKC) \prec 0\,,$$
or equivalently, to minimization of the numerical abscissa $\omega(A+BKC)$, 
defined as $\omega(M):= 1/2{\lambda_{\max}}(M+M^T)$. 
This is in line  with the results in \cite{kalur2021nonlinear} for transitional flow studies. 
On the other end, when $n_\phi = 0$, the plant is linear and the controller can be of full order. 
The last step is
construction of the controller given $X$ and $Y$ from (\ref{eq-LMIXY}), which is standard and found in \cite{gahinet1994linear}. 

Application to  the Lorenz model with  $x-1$-measurement, 
yields a $1$st-order controller  $K(s) = -(306.5+2809)/(s+0.1044)$. 
Simulation in closed loop is shown in Fig. \ref{fig-sfxy} (top left corner). The feedback controller is switched on after $15$ seconds, when the chaotic regime is well engaged.

Characterization of state-feedback controllers is easily derived from the second projection LMI in  (\ref{eq-LMIXY}), 
or using $u = K x$ and $C = I$  in the BMI (\ref{eq-LMI}): 
\begin{equation}
\label{eq-sf0}(A+BK)^T \diag(X,  I)  +(.)^T \prec -\epsilon \diag(X,  I),\; X \succ 0 \,,
\end{equation}
or equivalently, using a congruence transformation $\diag(Y,  I ) =  \diag(X,  I ) ^{-1}$, 
on the left- and right-hand sides of the first matrix inequality in (\ref{eq-sf0})
\begin{equation}
\label{eq-sf}(A+BK) \diag(Y,  I ) + (.)^T \prec -\epsilon \diag(Y,  I ),\; Y \succ 0\,.
\end{equation}
 
The
constraint (\ref{eq-sf}) is turned into an  LMI feasibility program using the standard change of variable $V:= K \diag(Y, I )$:
\begin{equation}
\label{eq-sf2}A \diag(Y, I )  + B V + (.)^T \prec -\epsilon \diag(Y,  I ),\; Y \succ 0 \,.
\end{equation}

All
LMI characterizations derived so far can be solved by standard convex SDP software as {\it LMILab}  \cite{MatlabRobust} or {\it SeDuMi}
\cite{sturm1999using}. Solving  (\ref{eq-sf2}) for the Lorenz model yields a globally stabilizing state-feedback controller $K = V \diag(P_{11}, I)= [-154,40   0.245, 0] $. A simulation is shown in Fig. \ref{fig-sfxy} top right.

The fact that the state-feedback controller does not use the $x_3$-measurement suggests that even simpler controller structures 
should be satisfactory, e.g. using 
static output feedback in $x_1$ or $x_2$. For $x_1$-measurement alone, we have $C = [1,0,0]$ and the BMI characterization is the same as in (\ref{eq-sf0}) with $A+BKC$ replacing $A+BK$. For a scalar $K$ this is easily solved by sweeping an interval of $K$ values and solving for the resulting LMIs with $K$ fixed. We obtain $K= -27.01$ with search interval $[-100, 100]$.  Simulations are displayed in Fig. \ref{fig-sfxy}, bottom left.
\begin{figure}[!htbp]   
\centering
\includegraphics[height=0.35\textheight, width = 0.45\textwidth]{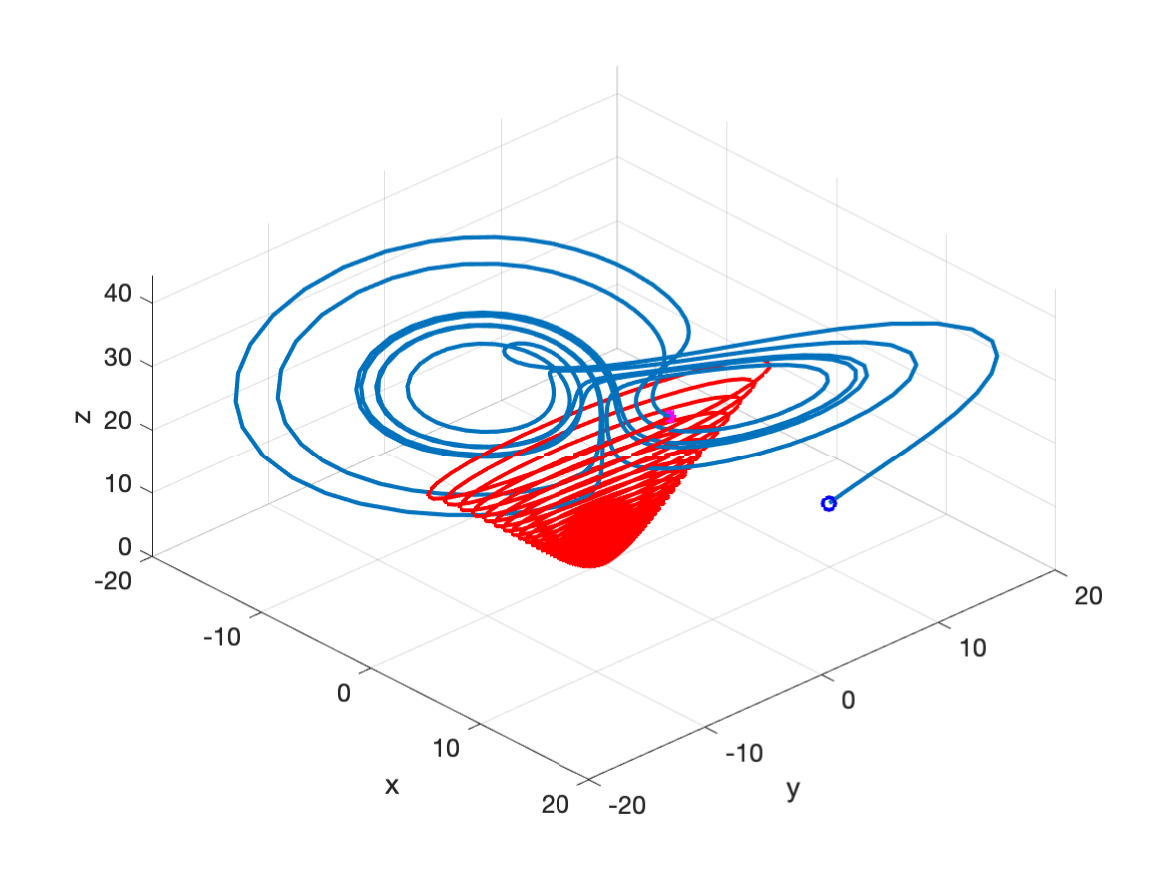}
\includegraphics[height=0.35\textheight, width = 0.45\textwidth]{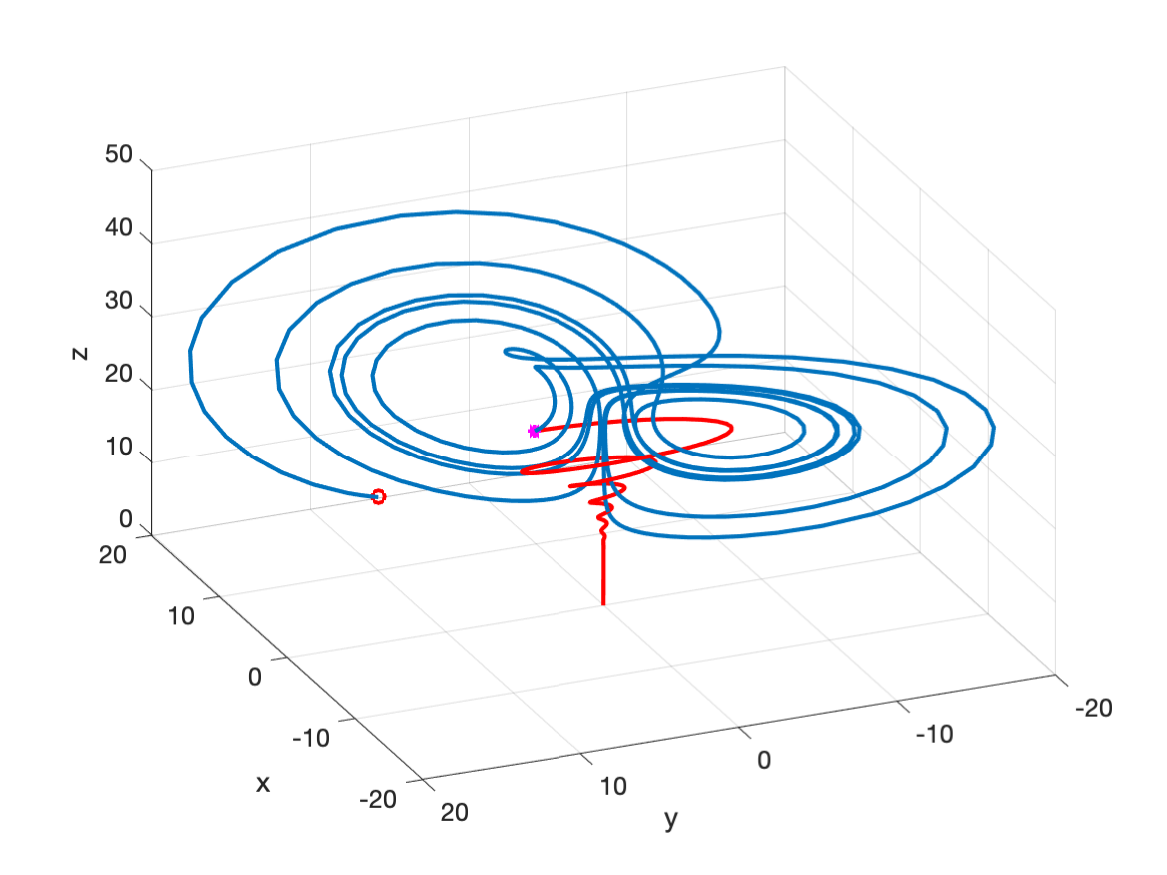} \\
\vspace{-.2cm}
\includegraphics[height=0.35\textheight, width = 0.45\textwidth]{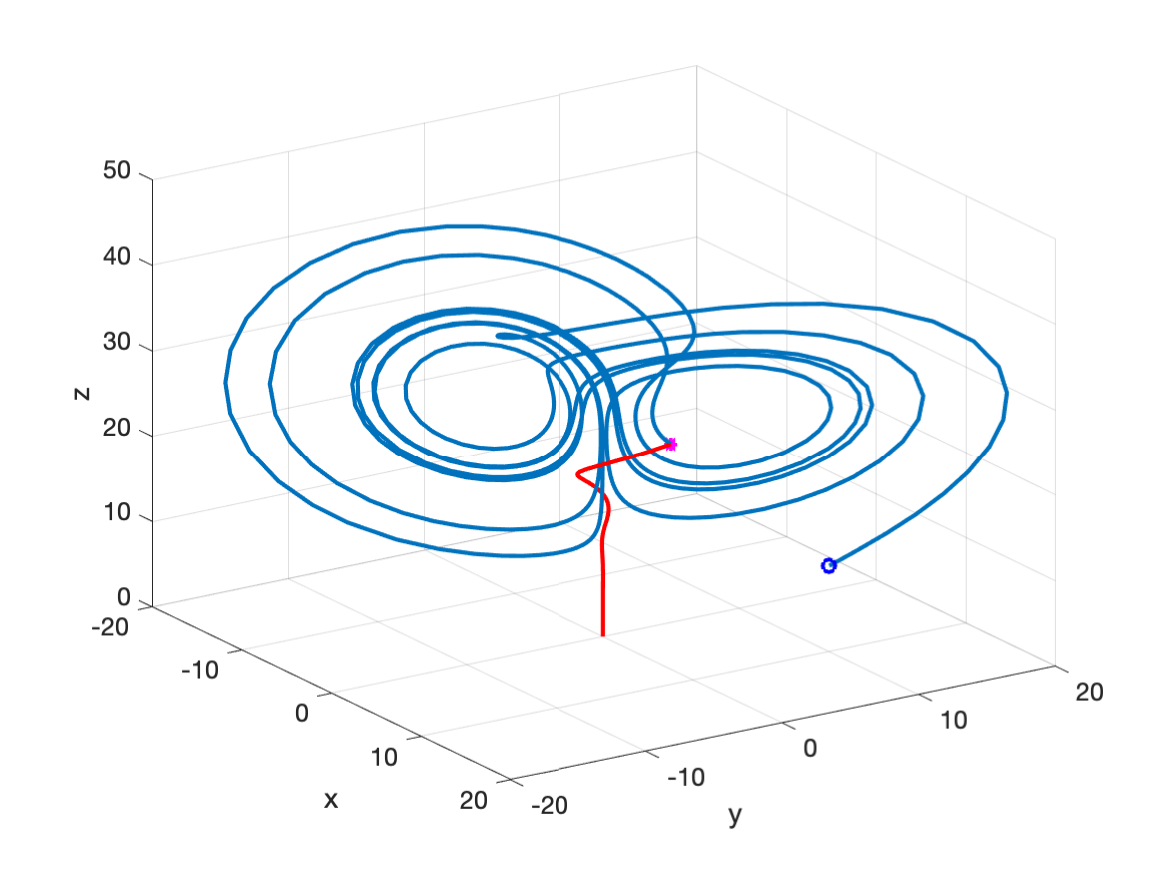}
\includegraphics[height=0.35\textheight,width = 0.45\textwidth]{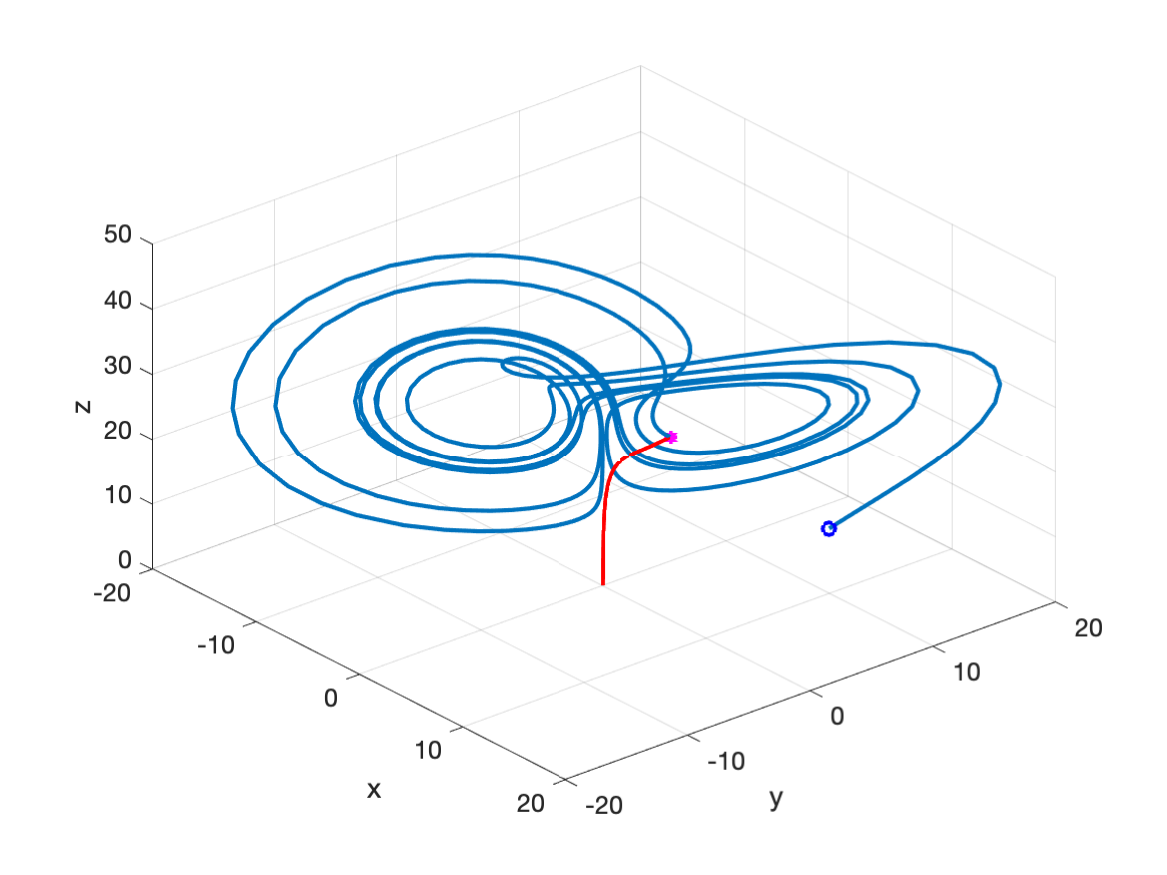}
\caption{Suppression of Lorenz double-scroll chaotic attractor using QC approach\label{fig-sfxy}} Top left: $x_1$-measurement dynamic  feedback, Top right: state feedback \\
Bottom left: : $x_1$-measurement static feedback, Bottom right: : $x_2$-measurement static feedback \\ Open loop in blue: response to state initial condition. Closed loop in red: response when controller is on. \\ 
\end{figure}

Similar results can be obtained with $x_2$-measurement feedback alone. The gain value is  $K = -27.01$, 
and simulations are given in Fig. \ref{fig-sfxy}, bottom right.

\subsubsection{Chaos dynamics: 
Kreiss norm minimization \label{sect-Kreiss12}}
We now investigate
whether similar results can be achieved with controllers minimizing the Kreiss system norm. 
Here we follow a different strategy which is to decouple the linear dynamics 
$\dot x = A x$ from the  nonlinearity $\phi$  by way of
mitigating transients due to initial conditions or $L^1$ disturbances. While this is a heuristic in the first place, 
it can of course in a second step be certified rigorously using the same QC approach, now for analysis. This has the 
advantage
that BMIs are replaced by LMIs.  In addition,  the technique  is applicable in a much more general context beyond the Lorenz model
as seen in sections \ref{sect-brunton} and \ref{sect-brunton2} when the QC approach turns out too conservative. 

Controllers based on minimizing the Kreiss norm alone are computed through the following min-max program
\begin{align}
\label{eqsynth}
\begin{array}{ll}
\displaystyle\mbox{minimize} & \displaystyle\max_{\delta \in [-1,1]} \left\| J^T \left( sI- \left(\textstyle \frac{1-\delta}{1+\delta} A_{cl}(K)-I \right)\right)^{-1} J \right\|_\infty\\
\mbox{subject to}& K \mbox{ robustly stabilizing}, \, K \in \mathscr K,
\end{array}
\end{align}
with the  definitions already given for program (\ref{eq-synth1}).

Program (\ref{eqsynth}) was solved for four controller structures: $x_1$-measurement  dynamic  feedback,  state feedback, static $x_1$-measurement feedback, and   $x_2$-measurement feedback. Controller gains were computed  as $K(s) = -(47.06 s  + 715.7)/(s + 17.95)$,  $[-41.07, -13.78,   0]$ , $-34.70$ and $-32.55$, respectively. In each case a Kreiss constant of unit value 
with $\mathcal M_0(G) =1$ was achieved, meaning that the linear dynamics do no longer amplify 
transients in the Lorenz model. Note that unlike the matrix case,  $\mathcal M_0(G) = 1$ cannot be inferred directly from $\mathcal K(G) = 1$,
but can be certified a posteriori.
Naturally, all controllers have been tested for global stability of the Lorenz model,
which for $K$ fixed 
uses the characterization in (\ref{eq-LMI}) and requires solving a convex SDP. 
Simulations  are given in Fig. \ref{fig-KreissSimu1}.

\begin{figure}[!htbp]   
\centering
\includegraphics[height=0.35\textheight, width = 0.45\textwidth]{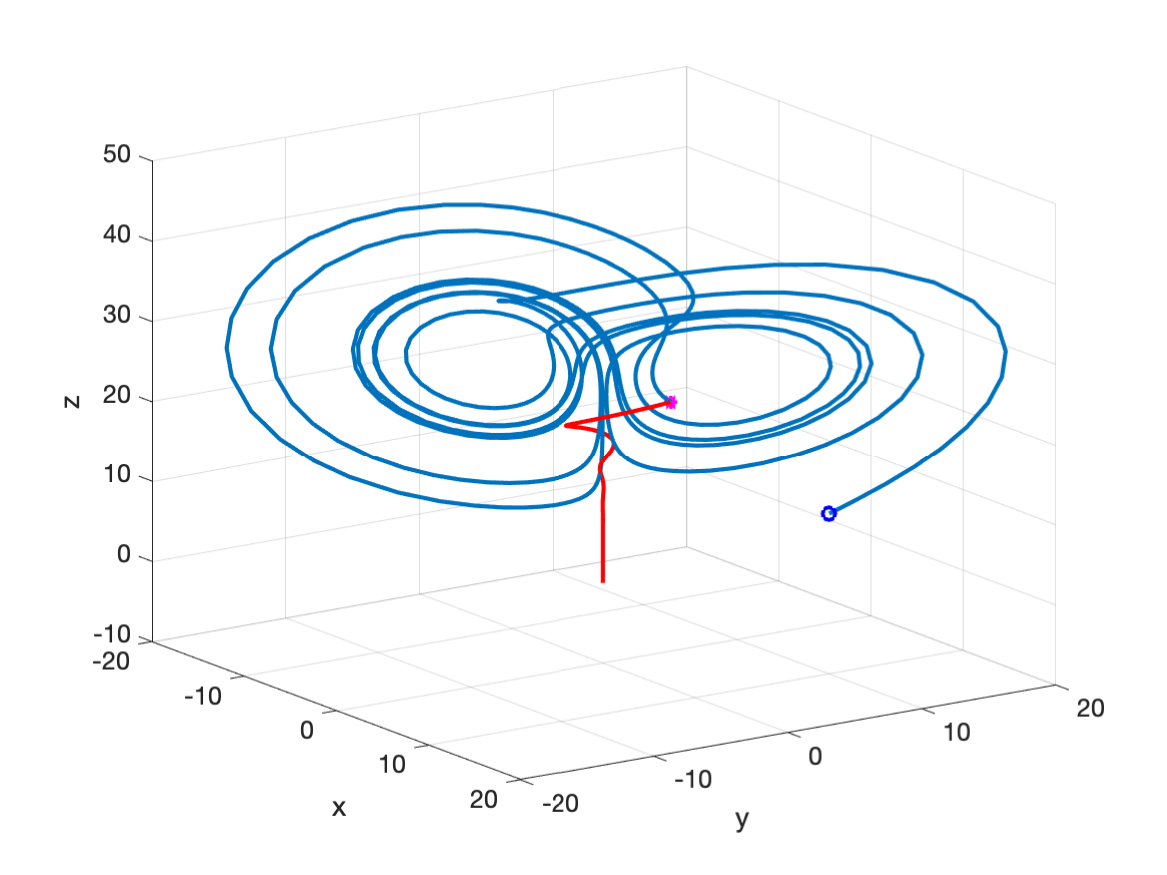}
\includegraphics[height=0.35\textheight, width = 0.45\textwidth]{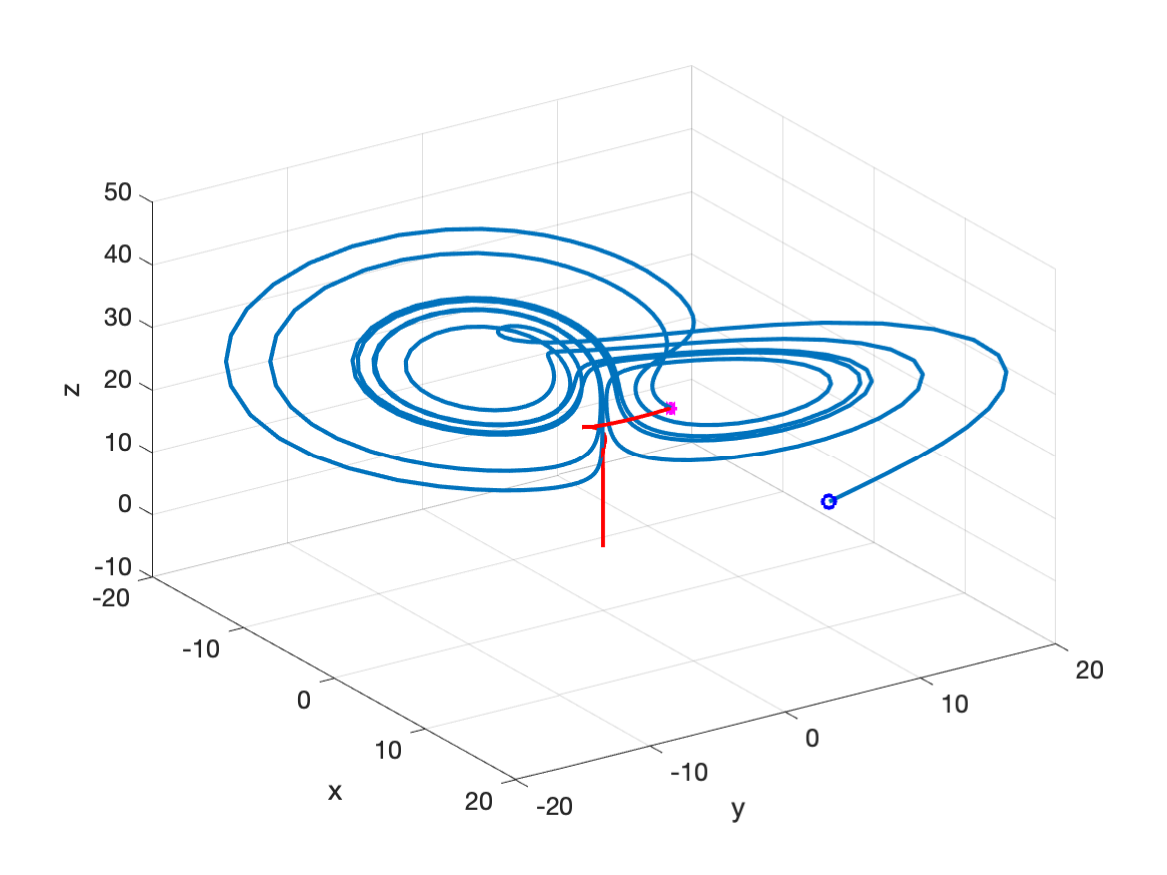} \\
\includegraphics[height=0.35\textheight, width = 0.45\textwidth]{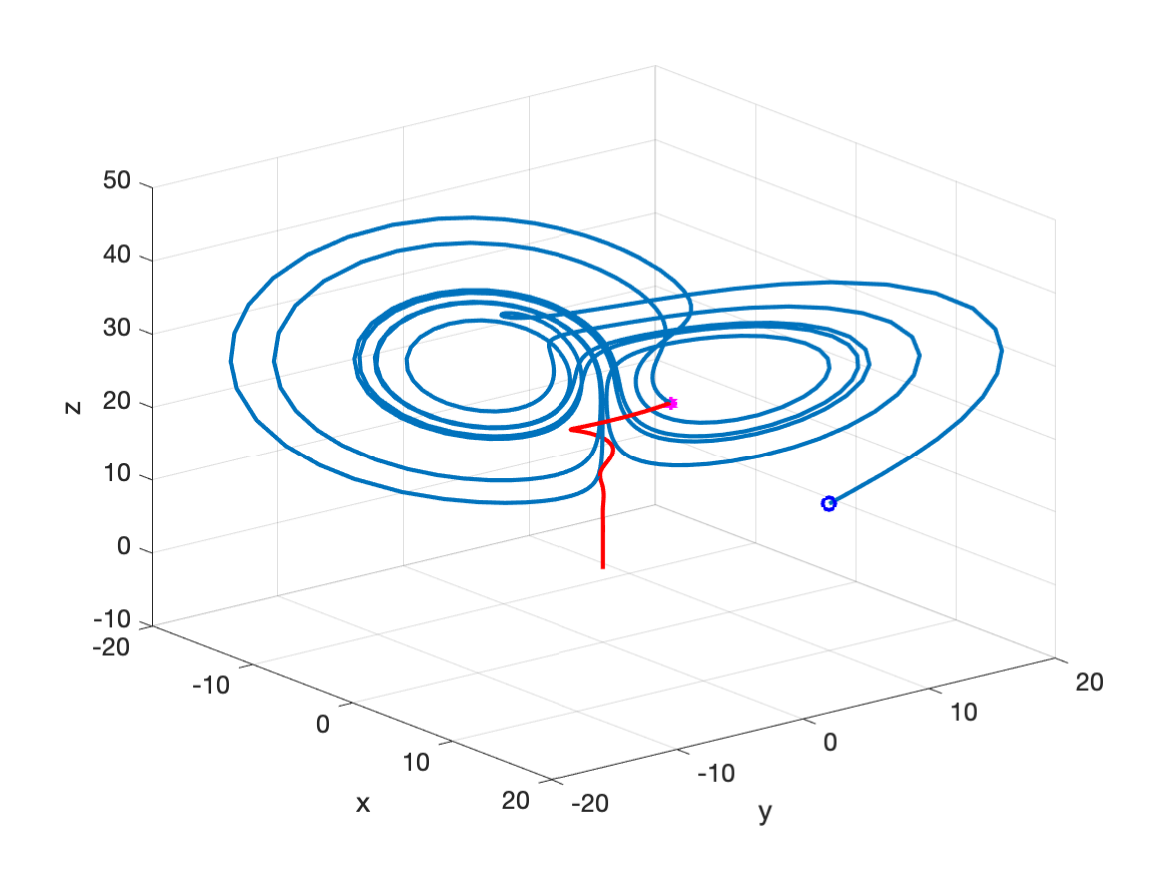}
\includegraphics[height=0.35\textheight,width = 0.45\textwidth]{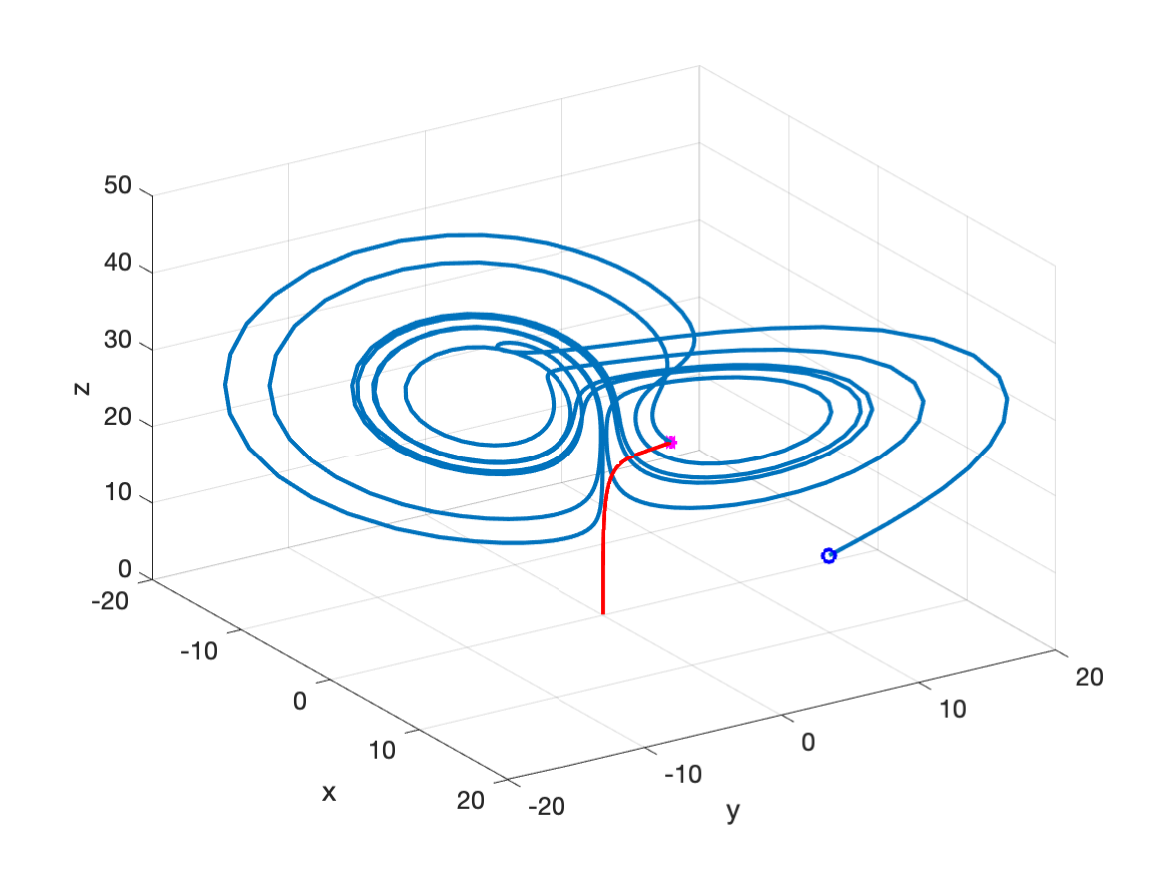}
\caption{Suppression of  Lorenz chaotic attractor using Kreiss norm minimization \label{fig-KreissSimu1}}Top left: $x_1$-measurement dynamic  feedback, Top right: state feedback \\
Bottom left: $x_1$-measurement static feedback, Bottom right: : $x_2$-measurement static feedback \\ Open loop: blue curve, Closed loop: red curve. 
\end{figure}

\subsection{Study of the Lorenz system with fixed points \label{sect-fiedP}}
For $R < 1$, the origin is the only stable equilibrium of (\ref{eq-lorenz1}) and the Lorenz model is then globally stable. 
When the Lorenz parameter is chosen as $1 < R < 17.5$, the chaotic attractor disappears and is replaced with   fixed points. 
For instance, when $R=10$, the Lorenz model has an unstable fixed point at the origin and two stable fixed points given in (\ref{eq-fixedPoints}). A typical illustration of that situation is shown 
in Fig. \ref{LorenzFixedPoints1}.  Trajectories with initial conditions arbitrarily close to $0$ are quickly captured by one of the fixed points. 

Despite this quite different pattern of the attracting regime, synthesis proceeds
along similar lines  as in section \ref{sect-lorenz1}. 
We remove the undesirable fixed points  and stabilize the origin globally
using static state-feedback, and dynamic and static output-feedback,  
comparing QC approach and Kreiss norm minimization. 

\subsubsection{Fixed-point dynamics: design with the QC approach \label{sect-QC2}}
As before, we start with the QC approach. 
A  state-feedback controller was computed as 
$K =[-136.40,  0.24,  0]$. Again the $x_3$ measurement is not used. That leads us to computing static output feedback controllers given as $K = -9.01$ and $K = -9.01$ for the $x_1$ and $x_2$ measurements alone, respectively. A dynamic $1$st-order $x_1$-measurement output feedback controller was computed as 
$K=-(288.5s + 2807)/(s+0.104)$ based on Theorem \ref{theoLMI}. All computed controllers globally stabilize the origin. 
This is illustrated in Fig. \ref{fig2-sfxy} for two initial conditions.

\begin{figure}[!htbp]
\centering
\includegraphics[height=0.3\textheight]{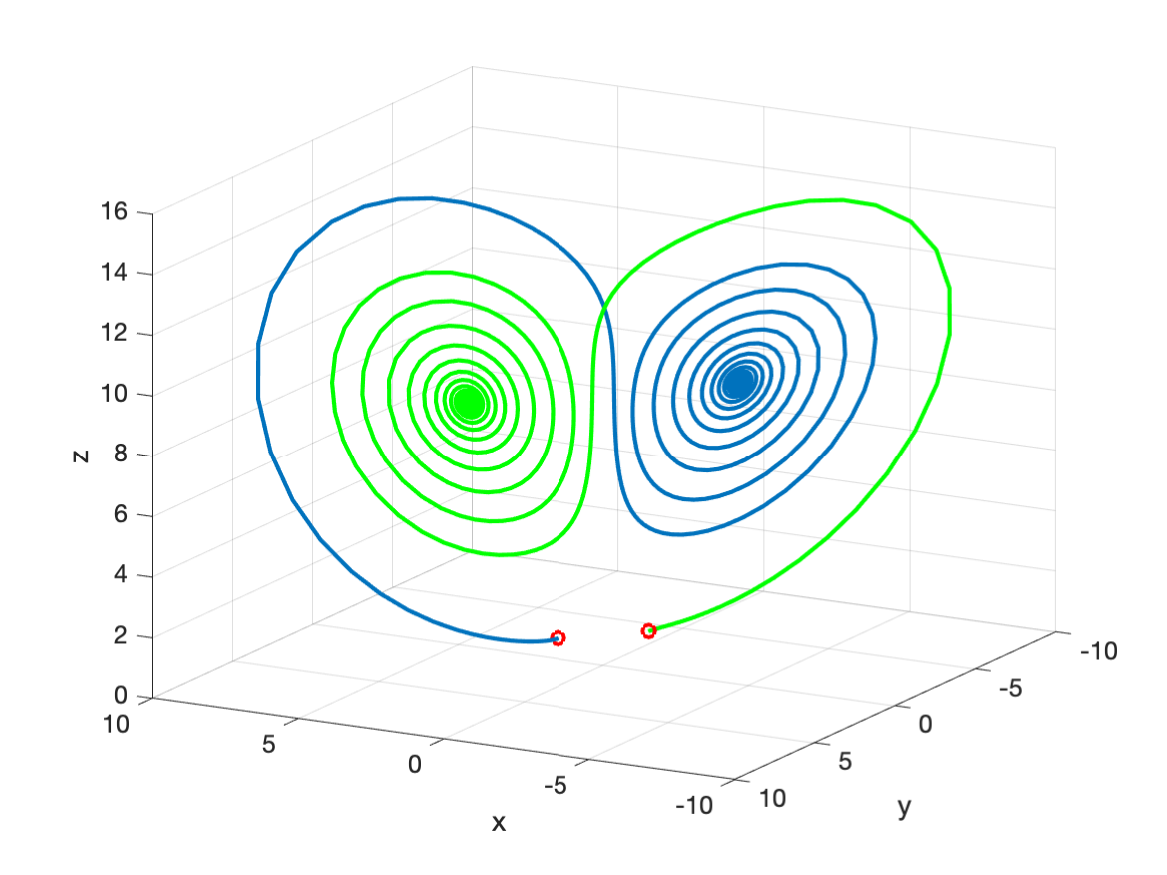}
\caption{Lorenz model for $1 < R < 17.5$ \label{LorenzFixedPoints1}} Unstable origin and two stable  fixed points
\end{figure}

\subsubsection{Fixed-point dynamics:  
Kreiss system norm\label{sect-Kreiss22}}
Controllers with identical structure were computed using Kreiss norm minimization. 
Dynamic  $1$st-order $x_1$-measurement output feedback, full state, $x_1$-measurement  and $x_2$-measurement static feedback  
were obtained as $K=-(12.23s+67.63)/(s+5.541)$, 
$K=[-4.47,  -6.92, 0]$, $K= -26.32$ and $K=-11.53$, respectively. All controllers were certified to stabilize the origin globally through feasibility of the LMI  (\ref{eq-LMI}). Simulations are shown in Fig. \ref{fig2-sfxyKreiss} and should be compared to Fig. \ref{fig2-sfxy}.

\begin{figure}[!htbp]
\centering
\includegraphics[height=0.35\textheight, width = 0.45\textwidth]{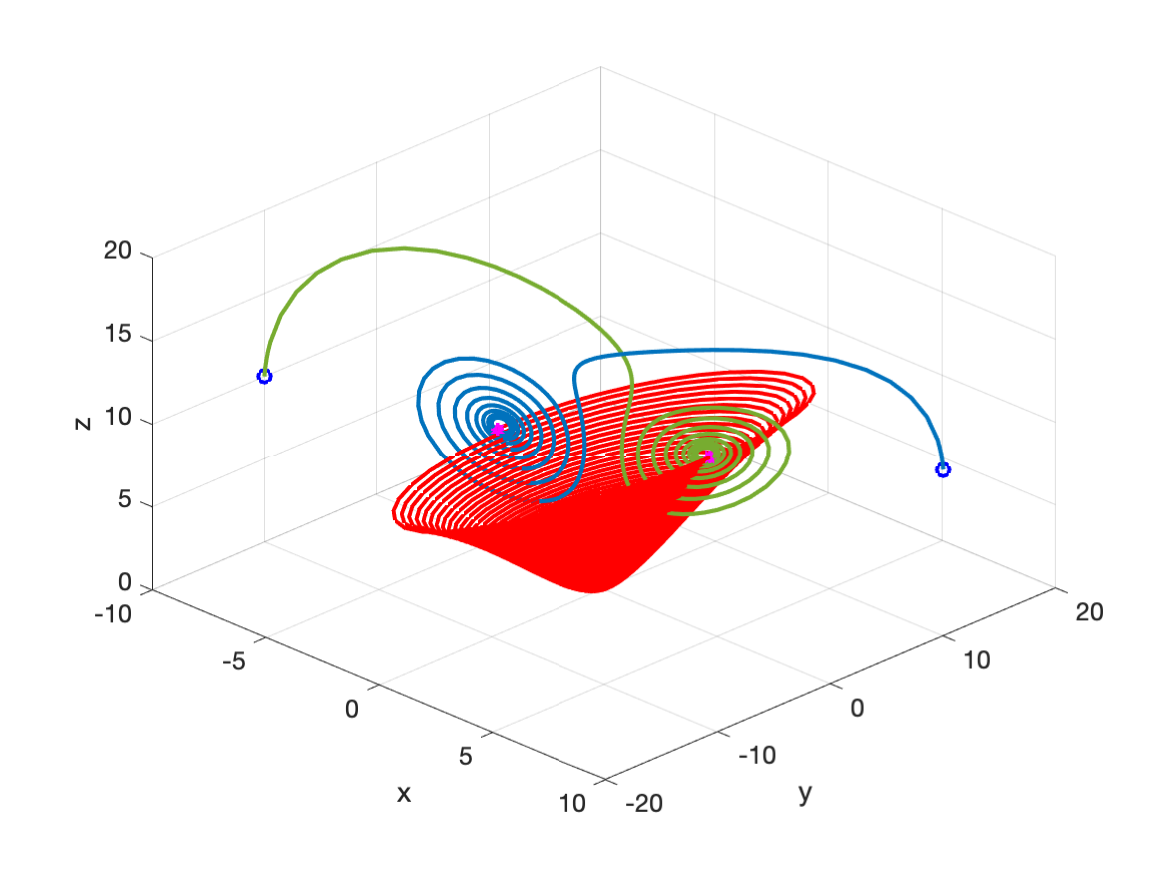}
\includegraphics[height=0.35\textheight, width = 0.45\textwidth]{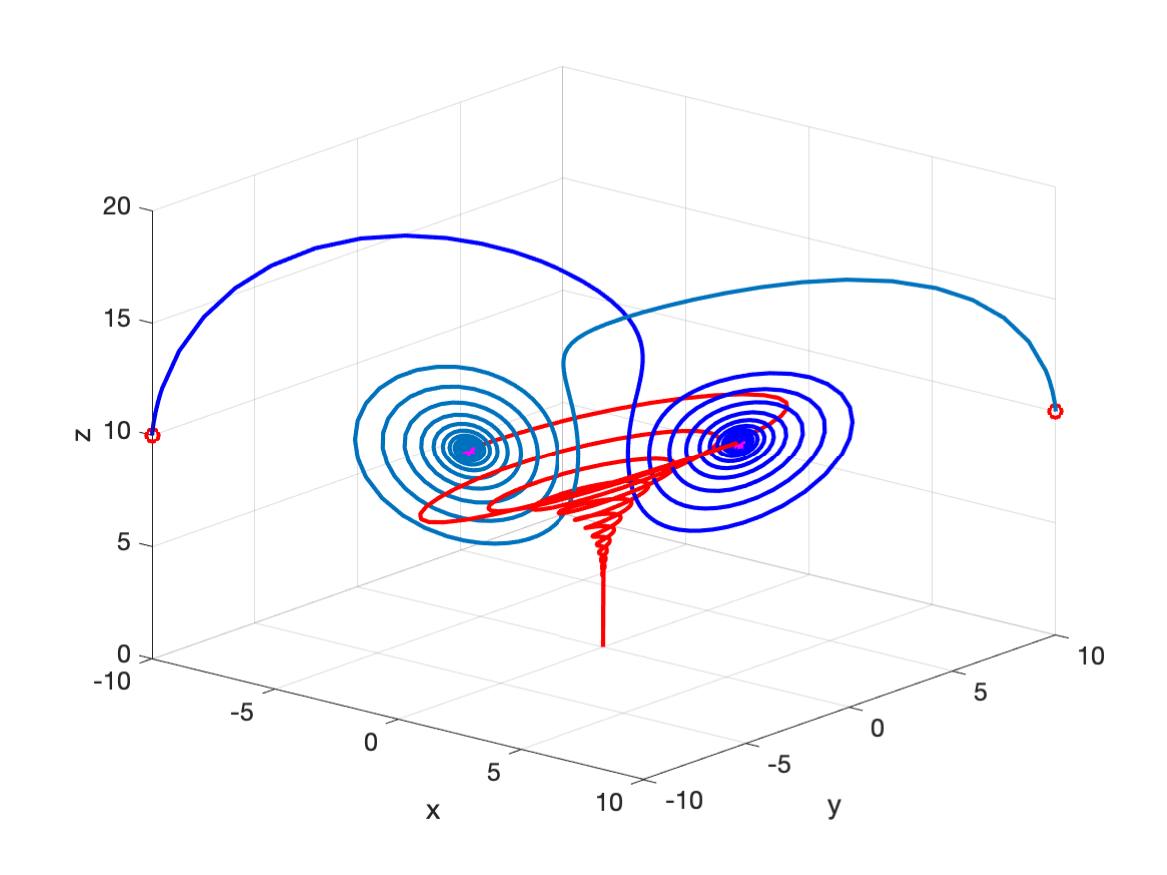}
\includegraphics[height=0.35\textheight, width = 0.45\textwidth]{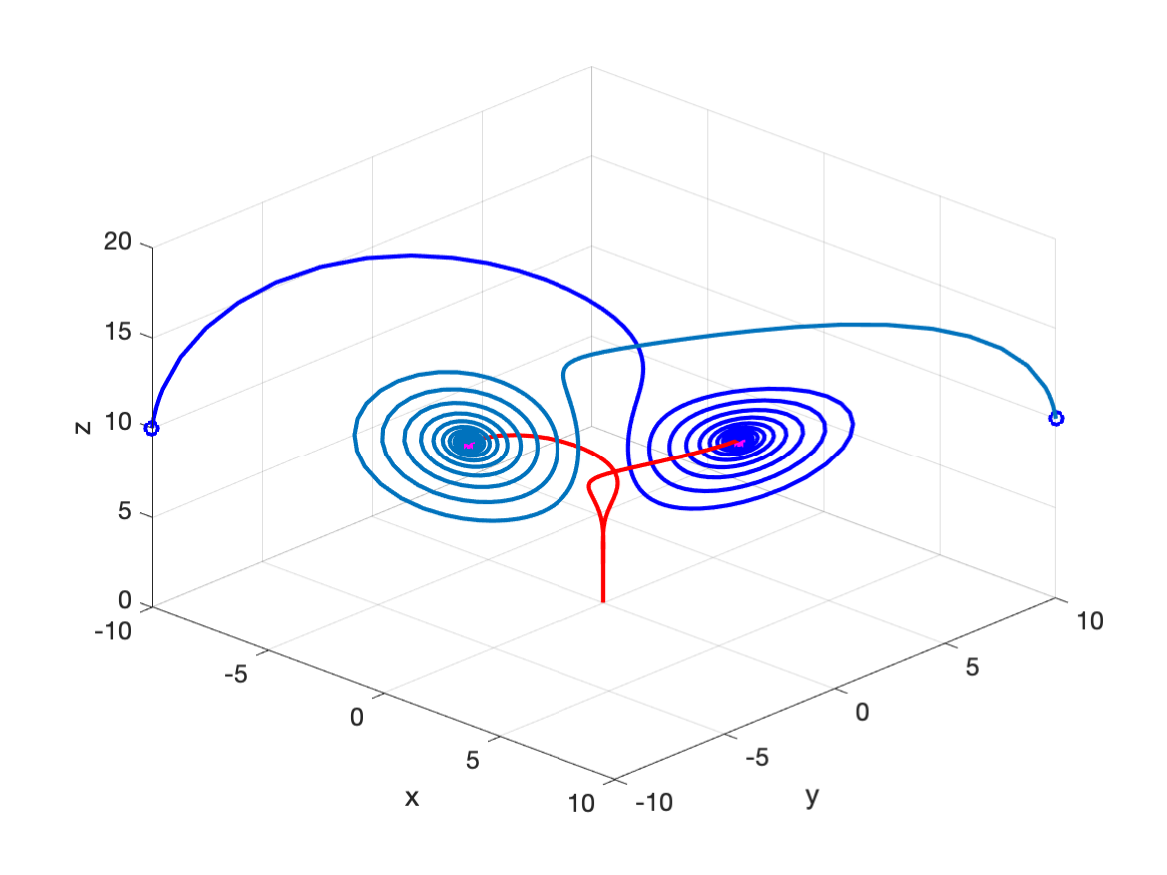}
\includegraphics[height=0.35\textheight,width = 0.45\textwidth]{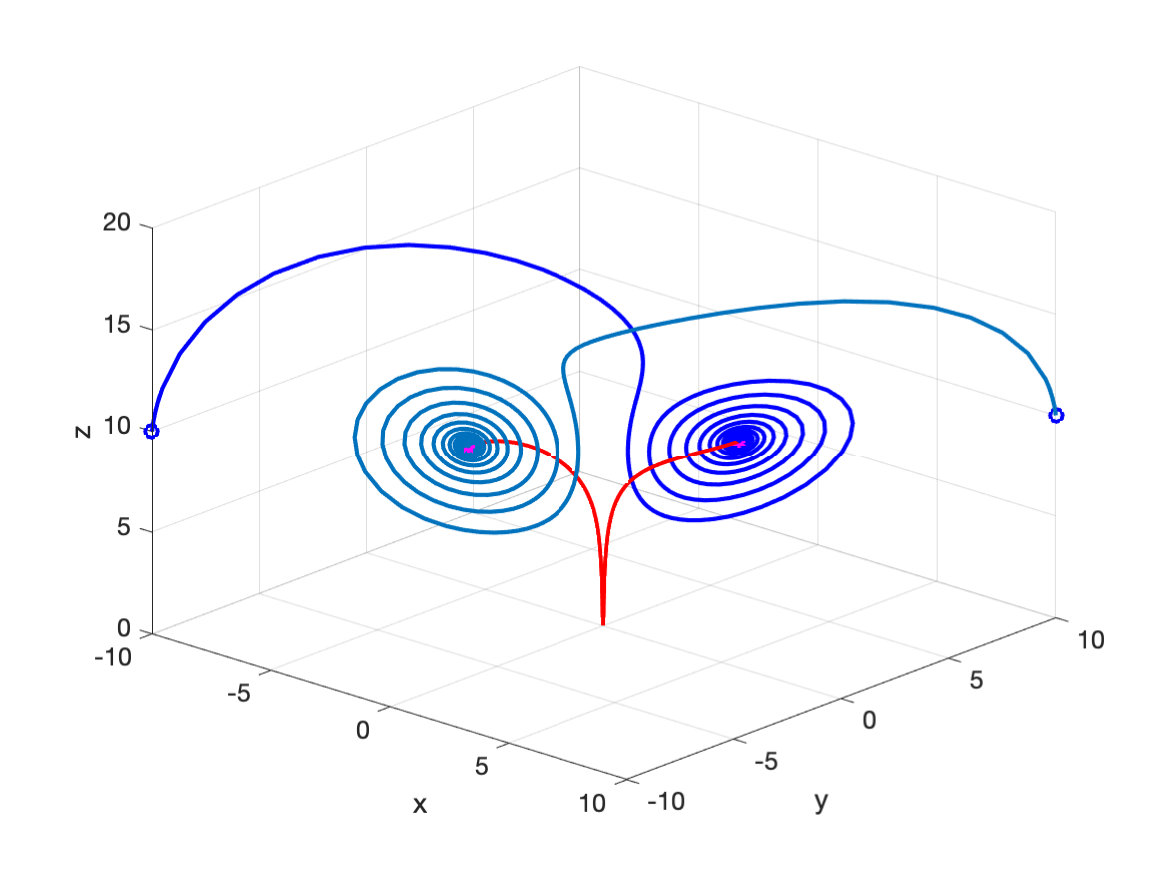}
\caption{Suppression of fixed point attractors using QC approach\label{fig2-sfxy}}Top left: $x_1$-measurement dynamic feedback, Top right: state feedback \\
Bottom left: : $x_1$-measurement static feedback, Bottom right: : $x_2$-measurement static feedback \\ Open loop: blue curve. Closed loop: red curve.  
\end{figure}

\begin{figure}[!htbp]
\centering
\includegraphics[height=0.35\textheight, width = 0.45\textwidth]{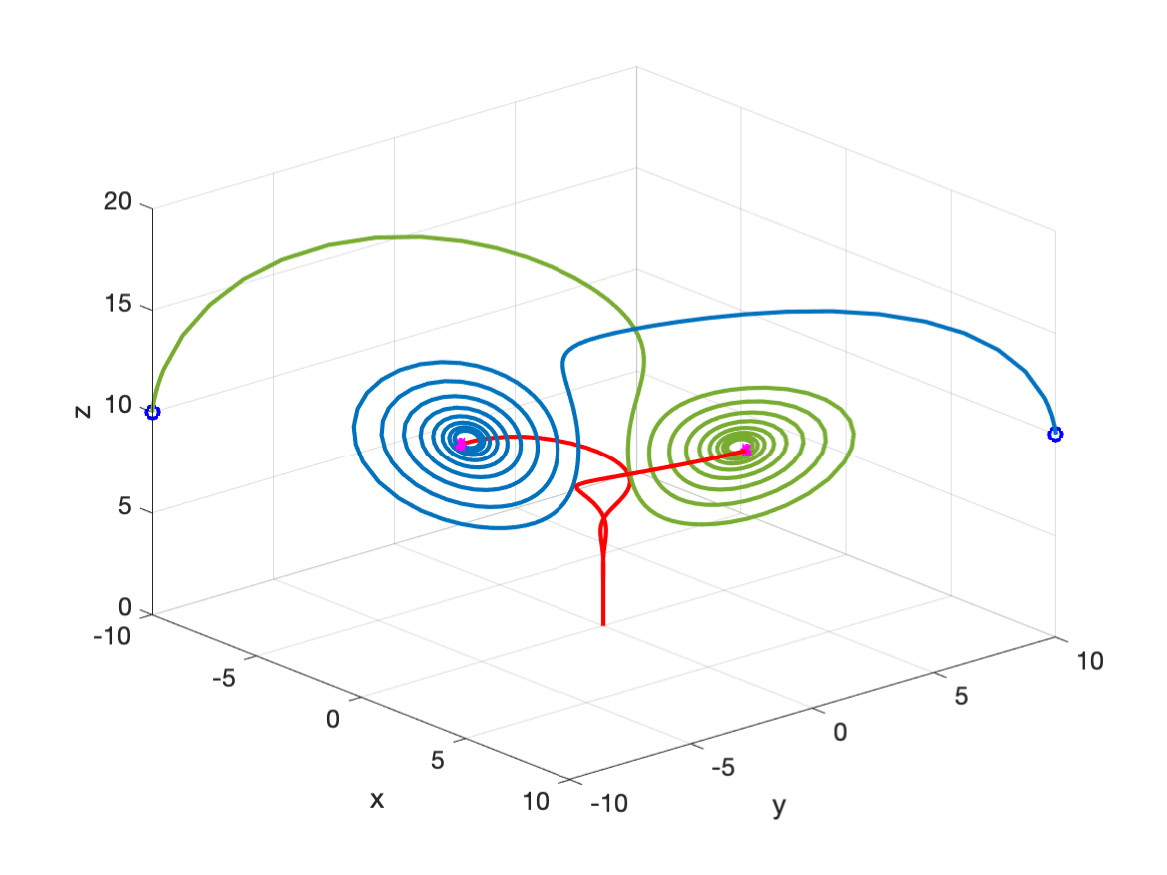}
\includegraphics[height=0.35\textheight, width = 0.45\textwidth]{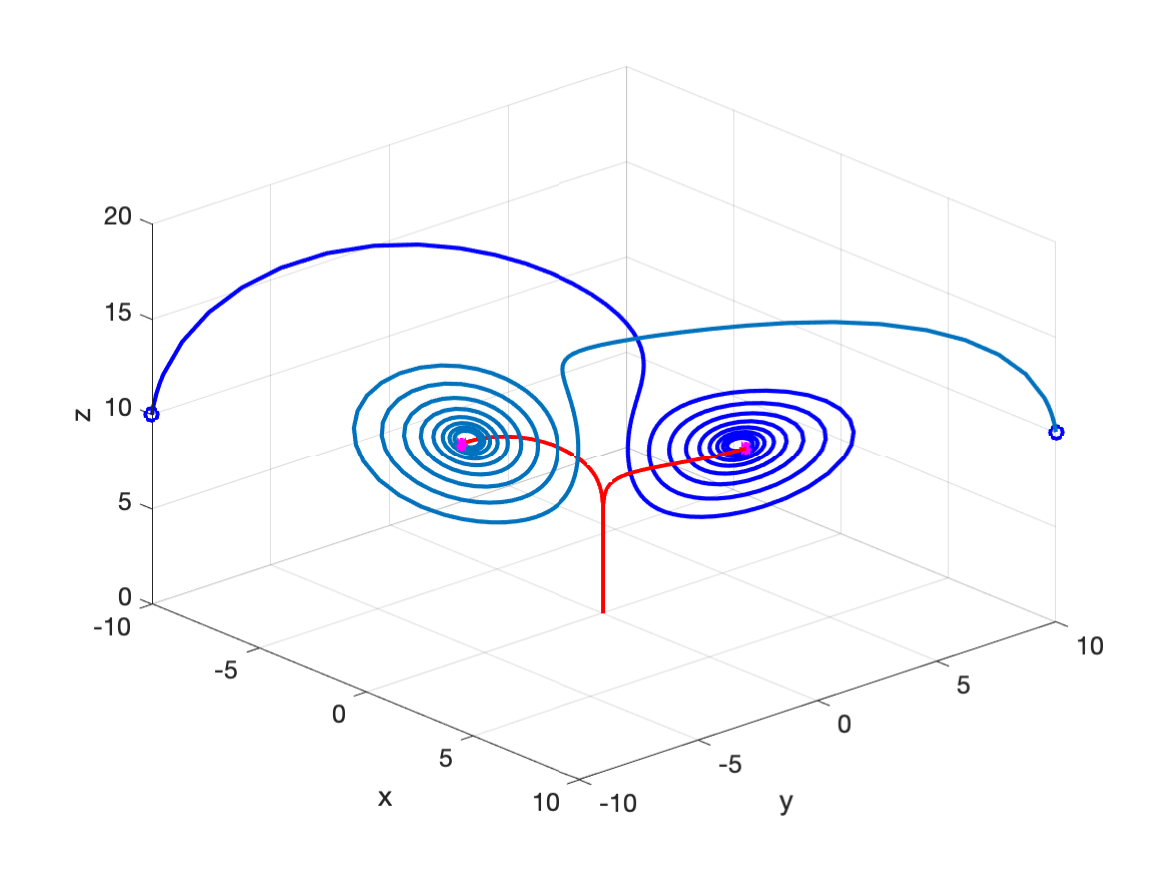} \\
\includegraphics[height=0.35\textheight, width = 0.45\textwidth]{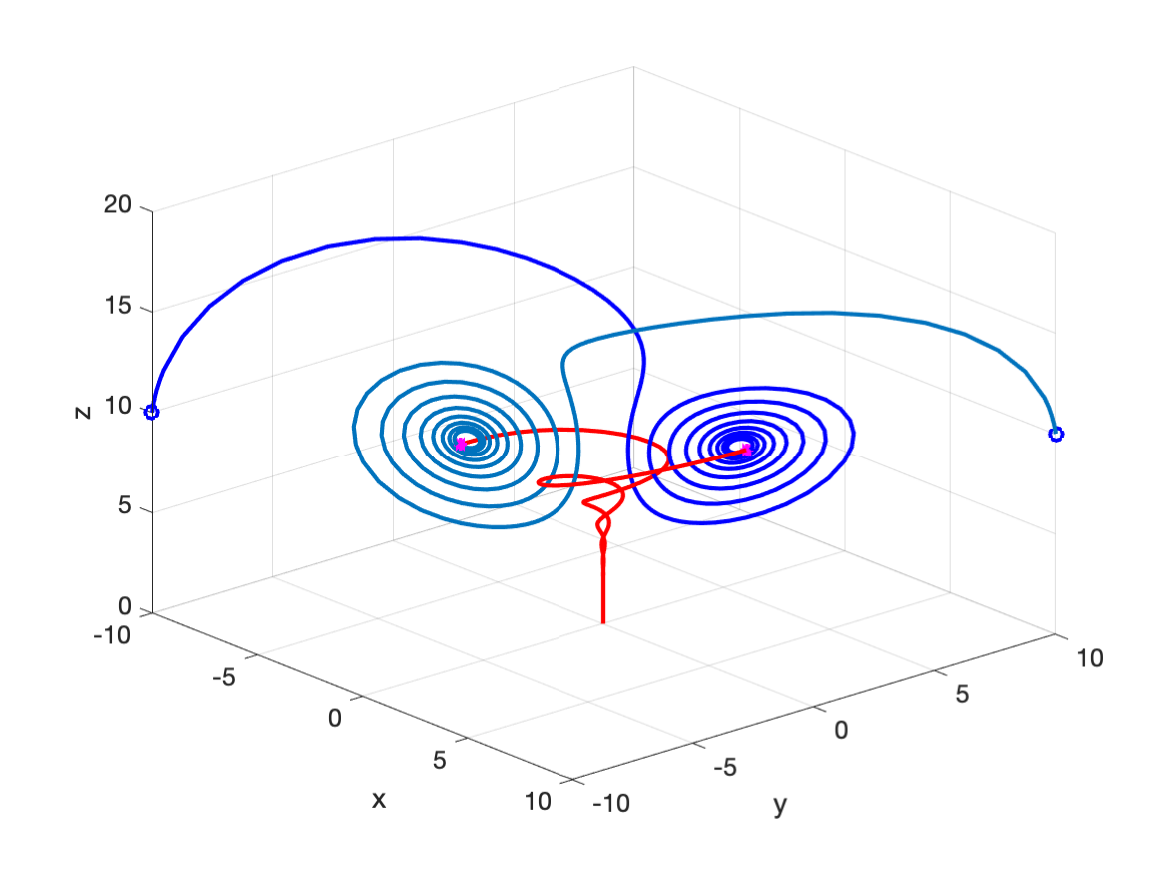}
\includegraphics[height=0.35\textheight,width = 0.45\textwidth]{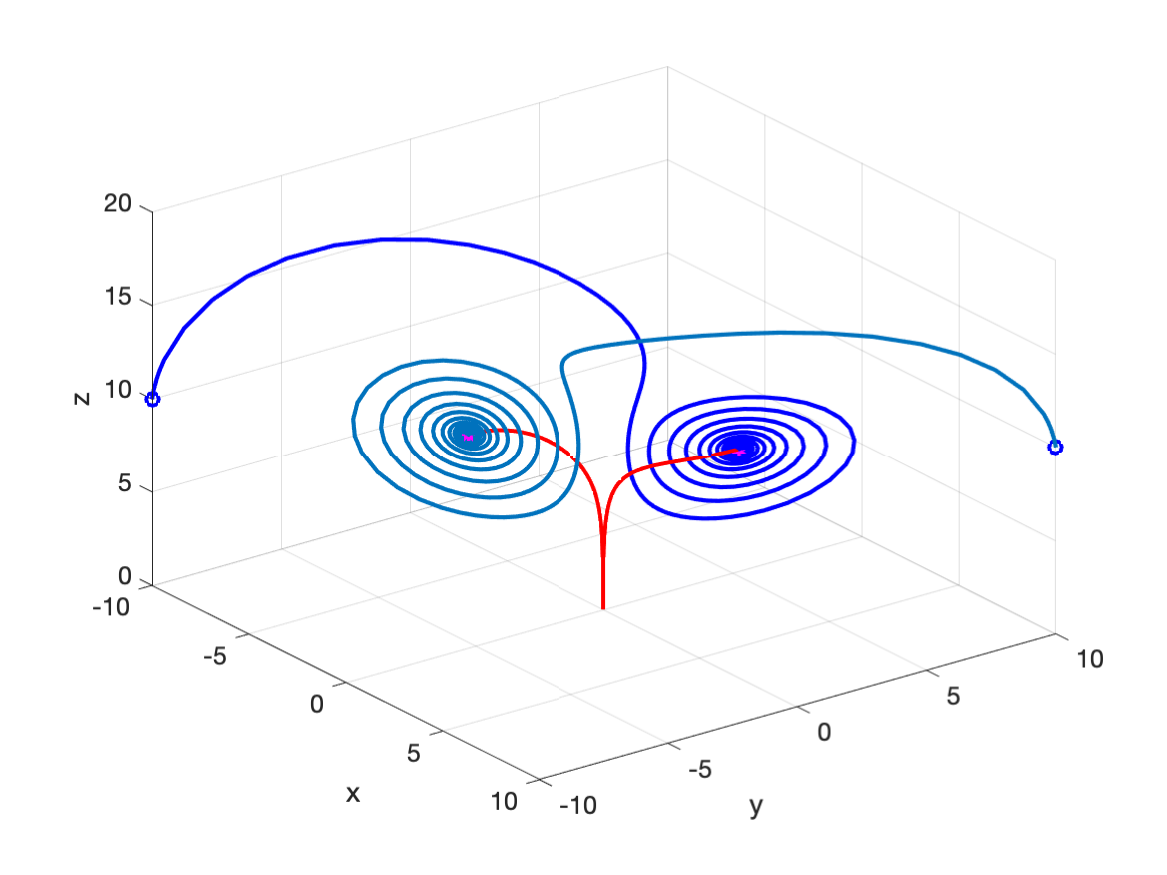}
\caption{Suppression of fixed point attractors using Kreiss norm minimization\label{fig2-sfxyKreiss}} Top left: $x_1$-measurement dynamic  feedback, Top right: state feedback \\
Bottom left: : $x_1$-measurement static feedback, Bottom right: : $x_2$-measurement static feedback \\ Open loop: blue curve, Closed loop: red curve.
\end{figure}

\section{Conclusion \label{sect-Conclusion}}
The idea to stabilize nonlinear systems in closed loop by mitigating transients of
the linearized closed loop was investigated, the rationale being that large transients are responsible
for driving the 
nonlinear dynamics outside the region of local stability.
Heuristic approaches tailored to transients caused by
noise, persistent perturbations, and finite consumption disturbances were obtained, opening up new possibilities for analysis and control of linear and nonlinear systems. 

The time-domain
worst case transient peak norm $\mathcal M_0(G)$ was identified
as suitable to assess transients caused by $L_1$-disturbances. The Kreiss system norm $\mathcal K(G)$ was introduced and studied as
a frequency domain approximation of $\mathcal M_0(G)$, better suited for the purpose of optimization due to its 
representation as a parametric robust control problem.
In each case the Kreiss-norm objective was effectively combined with other performance and robustness specifications as used in practice, underlining its relevance.
In our numerical testing, Kreiss norm optimization was evaluated by matching it, in small to medium size cases where  possible, with 
a properly extended QC approach.

The Kreiss norm approach is particularly effective for plants with up to several hundred states. 
This concerns the evaluation of Kreiss norm on the one hand, which has polynomial complexity, 
but also feedback design, which is naturally more challenging, but still manageable at such sizes.
Future work
may strive to enable Kreiss norm minimization for large-dimensional plants, 
such as discretizations of realistic fluid flow models or other PDE models. 
While challenging, this may be within reach when  model sparsity is exploited 
and specialized linear algebra is used. In contrast, LMI techniques and SOS certificates such as seen in the application section
are no longer viable options for such huge dimensions.

Over the past two years our optimizer  \cite{an:05,apkarian2006nonsmooth,apkarianNoll2017worst},
available through {\it systune} in \cite{MatlabRobust}, has been used regularly in
industrial applications, see e.g.
\cite{li2023multivariable,lim2022active,Vilarino2022,LeGuehennec2020,Maheshwari2022,deSousa2021,Asghar2023}. A 
specific interest for practitioners
is that it allows parametric robustness in tandem with multi-objective synthesis and controllers of designer-chosen  structure.

\appendix
\renewcommand{\thesection}{\Alph{section}}

\section{}\label{appendix-A}
We consider the closed-loop system (\ref{eq-Brunton1}) in polar coordinates
\begin{align}
\label{start}
\begin{split}
    \dot{r}& =  \sigma r - \alpha \beta r^3 + g K r\sin^2 \phi\\
     \dot{\phi} &= \omega + \alpha \gamma r^2  + g K \cos\phi \sin\phi 
     \end{split}
\end{align} 
First observe that $r(t)$ must be bounded. Indeed, we have $\dot{r} \leq 0$ for
$$
r^2 \geq \frac{\sigma+gK\sin^2\phi}{\alpha \beta}
$$
which due to $K < 0$ means that states $r$ with
$$
r^2 > \frac{\sigma}{\alpha\beta}  =: r_0^2
$$
cannot be reached (from below).
Namely if $r(0) < r_0$, then the trajectory may never reach values $r(t) > r_0$, as this would require derivatives
$\dot{r}>0$ in between $r_0$ and $r(t) > r_0$. Even when $r(0) > r_0$, then $\dot{r} < 0$ on some $[0,\epsilon)$, so the trajectory 
decreases until $r(t) = r_0$ is reached, and then the previous argument shows that it cannot rebounce to values $>r_0$. 
In conclusion, the trajectories of the system are bounded.

Let us look for steady states $(x^*,y^*)$.
In the original $(x,y)$-system we have (with $r^2=x^2+y^2$)
\begin{align*}
    0&= (\sigma-\alpha\beta r^2) x - \omega y - \alpha\gamma r^2 y\\
    0& =\omega x + \alpha\gamma r^2 x +  \sigma y-\alpha\beta r^2 y+gKy
\end{align*}
and this can be written
$$
A(r):=\begin{bmatrix}
    \sigma-\alpha \beta r^2 & -\omega -\alpha\gamma r^2 \\
     \omega+\alpha \gamma r^2 & \sigma- \alpha\beta r^2 + gK
\end{bmatrix}
\begin{bmatrix} x\\y\end{bmatrix}=\begin{bmatrix}0\\0\end{bmatrix}
$$
For this system to have a non-zero solution $(x^*,y^*)\not=(0,0)$, the determinant of the system matrix
$A(r)$
must vanish, which leads to
$$
(\sigma-\alpha\beta r^2)^2 + gK(\sigma-\alpha\beta r^2) + (\omega+\alpha \gamma r^2)^2=0.
$$
This quadratic equation in $\sigma-\alpha \beta r^2$ has no real solution
for $g^2K^2-4(\omega+\alpha\gamma r^2)^2 < 0$, which gives the following
\begin{proposition}
    Suppose $-K < \frac{2\omega}{g}$. Then the only steady state of the closed-loop system is $(0,0)$.
\end{proposition}

The origin is locally exponentially stable, so there exists a largest ball $B(0,\rho)$ such that all trajectories
starting in $B(0,\rho)$ converge to $(0,0)$. Suppose $\rho < \infty$, then there exists $(x_0,y_0)\not\in B(0,\rho)$ such that
the trajectory starting at $(x_0,y_0)$ does not enter the ball $B(0,\rho)$. Since it is a bounded trajectory,
the Poincar\'e-Bendixon theorem implies that it must approach a limit cycle. For a limit cycle to exist, the system
must admit a periodic solution.

We therefore look for conditions which allow to exclude the existence of a periodic solution. The Bendixon condition
tells that this is the case when $P_x+Q_y$ does not change sign, where $P,Q$ are the right hand sides
of (\ref{eq-Brunton1}) with the loop $u=Ky$ closed. We get
\begin{align*}
    P_x+Q_y = 2 \sigma - 4 \alpha \beta r^2 + gK
\end{align*}
and this has negative sign for $K < -\frac{2\sigma}{g}$. We conclude the

\begin{proposition}
    Suppose $K \in \mathbb R$ satisfies $K < -\frac{2\sigma}{g}$ and $-K < \frac{2\omega}{g}$. 
    Then {\rm (\ref{eq-Brunton1})}
    is globally stabilized by the static controller $u=Ky$.
\end{proposition}

\subsection{Dynamic controllers}
Consider the case of  dynamic controllers. Closed-loop dynamics are obtained as follows ($r^2=x^2+y^2$):
$$
\begin{bmatrix} \dot x \\ \dot y \\ \dot x_K \end{bmatrix} = 
\begin{bmatrix}
(\sigma-\alpha \beta r^2) &  -(\omega + \alpha r^2 \gamma)   & 0  \\
(\omega + \alpha \gamma r^2) & (\sigma + g D_K-\alpha \beta r^2) & g C_K \\
0 & B_K  & A_K 
\end{bmatrix}
\begin{bmatrix} x \\ y \\x_K \end{bmatrix}\,.
$$
The equilibrium equations give $x_K = -A_K^{-1} B_K y$, assuming that $A_K$ is invertible. This leads to 
$$
\begin{bmatrix} (\sigma-\alpha \beta r^2) &  (\omega + \alpha r^2 \gamma) \\
(\omega + \alpha \gamma r^2) & (\sigma -\alpha \beta r^2 + g (D_K -C_K A_K^{-1} B_K))
\end{bmatrix} \begin{bmatrix} x\\y\end{bmatrix}=\begin{bmatrix}0\\0\end{bmatrix}\,,
$$
which as before,  has $(0,0)$ as unique solution if and only if the system matrix is invertible. The determinant quadratic equation in $\sigma-\alpha \beta r^2$ has no real solution and is thus non-zero when  
$$ ( g (D_K -C_K A_K^{-1} B_K))^2-4(\omega + \alpha \gamma r^2)^2 < 0\,,$$
which is guaranteed when
$$|D_K -C_K A_K^{-1} B_K| < 2\omega/g  \,.$$
Note the latter involves a  constraint  on the DC gain of the dynamic controller $K(s)= C_K(sI-A_K)^{-1}B_K + D_K$. 

The polar form of these differential equations for $(x,y)$ is obtained as
\begin{align}
\label{polar2}
\begin{split}
    \dot{r}& =  \sigma r - \alpha \beta r^3 + g D_K r\sin^2 \phi + gC_K x_K \sin \phi  \\
     \dot{\phi} &= \omega + \alpha \gamma r^2  + g D_K \cos\phi \sin\phi + g C_K x_K \cos \phi  \\
     \dot{x}_K &= A_K x_K + B_K r \sin \phi
     \end{split}
\end{align} 

Assuming that 
$A_K$ is Hurwitz as is the case for all controllers based on the Kreiss norm, the third equation in (\ref{polar2}) gives us on every finite interval $[0,t_0]$ 
an estimate of the form $\max_{0 \leq t \leq t_0}|x_K(t)| \leq c \max_{0\leq t \leq t_0} r(t)$ for a constant $c >0$
independent of $t_0$.
Indeed, $x_K(t) = \exp(tA_K)x_0+ \int_0^t \exp((s-t)A_K) B_K \sin \phi(s) r(s)ds$, hence from Young's inequality
(with $q=r=\infty$, $p=1$), we get
$$\max_{0 \leq t \leq t_0} |x_K(t)| \leq c_1 + \|B_K\|\| \exp(t A_K)\|_1 \max_{0 \leq t\leq t_0}r(t)\leq c_1 + c_2 \max_{0\leq t\leq t_0} r(t) \leq c \max_{0\leq t \leq t_0} r(t).$$
Therefore by the comparison theorem, (see Lemma \ref{comparison} below), applied to the first equation in (\ref{polar2}),
$r(t)$ 
is bounded above by the solution of the equation $\dot{r} = (\sigma + g|D_K| + g\|C_K\|c)r - \alpha\beta r^3$.
The latter, however, is globally bounded, as the negative term $-\alpha\beta r^{3}$ dominates for large $r>0$.
Having established global boundedness of $r(t)$, we go back into the equation $\dot{x}_K=A_Kx_K+B_K r\sin \phi$, from which we now
derive global boundedness of $x_K$, and so altogether trajectories of (\ref{polar2}) remain bounded.
\begin{lemma}
\label{comparison}
{\rm (See e.g. \cite[Thm. 2.1, p. 93]{zabczyk2020mathematical})}.
    Suppose $\phi$ satisfies $|\phi(t,x)-\phi(t,x')| \leq M|x-x'|$ for all $t\in [t_0,t_1]$ and $x,x'$, and is jointly continuous. Let $v(t)$ be an absolutely continuous function
    such that $\dot{v}(t) \leq \phi(t,v(t))$ for almost all $t\in [t_0,t_1]$. Then
    $v(t) \leq u(t)$ on $[t_0,t_1]$, where $u(t)$ is the solution of $\dot{u}(t)=\phi(t,u(t))$ with initial value $u(t_0)$
 satisfying   $v(t_0)\leq u(t_0)$. \hfill $\square$
\end{lemma}

We are now in the situation addressed in \cite[Corollary]{yorke1970theorem}, which says that if a $C^1$-function $V(x)$ can be found satisfying
\begin{equation}
    \label{yorke}
\dot{V}(x) + \ddot{V}(x) \not= 0 \mbox{ for all } x\not=0
\end{equation}
then trajectories either converge $x(t)\to 0$, or escape to infinity $|x(t)|\to \infty$. 
Since we have already ruled out the latter, we
have then a certificate of global asymptotic stability.
For this model, we have used the more restrictive condition $\dot V(x) < 0$, with
$V(x) = V_1(x) + \dot V_2(x)$ and $V_1$, $V_2$ are chosen as multivariate polynomials. See \cite{ahmadi2011higher} for details. The polynomials are then sought using {\it sostools} \cite{sostools}.

For both the $1$st- and $3$rd-order controllers, a solution was obtained with  $V_1$ and $V_2$  sums of  monomials of degree $2$. For the simpler $1$st-order controller, with $x_{cl} = (x,y,x_K)$  this reads 

\begin{equation*}\begin{split}
V_1(x_{cl}) & = 2.556x^2 - 1.389x y - 0.02803 x x_K + 2.897 y^2 - 3.846e\text{-}5y x_K + 0.003159 x_K^2 \\
V_2(x_{cl}) & = - 0.2061 x^2 + 0.008941 x y - 1.324e\text{-}6 x x_K - 0.1787y^2 +
1.641e\text{-}5 y x_K - 0.008169 x_K^2\,.
\end{split}\end{equation*}
We have established that  $x(t)\to 0$.

\section{}\label{appendix-C}
Since the first matrix in (\ref{eq-LMI}) has a zero principal sub-matrix, 
the corresponding row and column terms should be zero for this
matrix to be negative semi-definite. This leads to $ X_{cl} B_{w,cl} + \mu_0  B_{w,cl} = 0 \,
$. Also, the $(1,1)$ sub-matrix should be negative semi-definite.   Using a partitioning in $X_{cl}$ conformable to that of $B_{w,cl}$ in (\ref{eq-CLlorenz}), we have 
$$X_{cl} = \begin{bmatrix}X & X_{12} & X_{13} \\ X_{12}^T & X_{22} & X_{23} \\ X_{13}^T & X_{23}^T & X_{33} \end{bmatrix}, \mbox{ with } X \in \mathbb R^{(n-n_\phi)\times (n-n_\phi)} ,\, X_{22} \in \mathbb R^{n_\phi\times n_\phi},\, X_{33} \in \mathbb R^{n_K\times n_K}  $$
with $n_\phi$ the vector dimension of the nonlinearity $\phi$.
This gives $X_{12}= 0$, $X_{22}=  -\mu_0 I$ and $X_{23}=0$. Due to homogeneity of the problem, 
$\mu_0$ is set to $-1$, and since $X_{22}$ should be positive definite, we get $X_{22} = I$. 
Also, non-strict feasibility can be replaced with strict feasibility by reducing $\epsilon$ if necessary. 
Summing up, assessing global stabilization with a dynamic controller $K(s)$ reduces to a specially structured Lyapunov inequality
\begin{equation}
\label{eq-BMI}
A_{cl}^T X_{cl} + X_{cl}  A_{cl} + \epsilon X_{cl} \prec 0, \;  \;X_{cl} = \begin{bmatrix}X & 0 & X_{13} \\ 0 &  I & 0 \\ X_{13}^T & 0 & X_{33} \end{bmatrix}, X_{cl} \succ 0 \,. 
\end{equation}

This  is  rewritten in the  familiar form: 
\begin{equation}
\label{eq-BMI2}
\Psi + P^T \Theta Q + Q^T \Theta P \prec 0, \;  \;X_{cl} = \begin{bmatrix}X & 0 & X_{13} \\ 0 &  I & 0 \\ X_{13}^T & 0 & X_{33} \end{bmatrix}, X_{cl} \succ 0 \,,
\end{equation}
with appropriate matrices $\Psi$, $P$, $Q$ depending on $X,A,B,C$ and controller data gathered in 
$$\Theta := \begin{bmatrix}A_K & B_K \\ C_K & D_K \end{bmatrix}\,.$$
We can then apply the Projection Lemma \cite{gahinet1994linear} to eliminate $\Theta$, which
leads to LMI solvability conditions. There exist controllers of order $n_K$ if and only if $W_P^T \Psi W_P \prec 
 0$ and  $W_Q^T \Psi W_Q \prec 
 0 $, for some $X_{cl} \succ 0$.  Introducing the inverse of $X_{cl}$ as 
$$Y_{cl}:=X_{cl}^{-1} = \begin{bmatrix}Y & 0 & Y_{13} \\ 0 &  I & 0 \\ Y_{13}^T & 0 & Y_{33} \end{bmatrix}\,, $$ and following \cite{gahinet1994linear}, the two projection inequalities  are computed as 
\begin{equation}\label{proj1}\begin{split}
    N_C^T\left( A^T \begin{bmatrix}X & 0 \\0 & I \end{bmatrix}+  \begin{bmatrix}X & 0 \\0 & I \end{bmatrix} A + \epsilon \begin{bmatrix}X & 0 \\0 & I \end{bmatrix} \right)N_C \prec 0 \\
    N_B^T\left( A \begin{bmatrix}Y & 0 \\0 & I \end{bmatrix}+  \begin{bmatrix}Y & 0 \\0 & I \end{bmatrix} A^T + \epsilon \begin{bmatrix}Y & 0 \\0 & I \end{bmatrix} \right)N_B \prec 0 
\end{split}\end{equation}
where $N_C$ and $N_B$ are bases of the null space of $C$ and $B^T$, respectively. Also, completion of $X_{cl} = Y_{cl}^{-1} \succ 0$ and 
$X_{cl} \in \mathbb R^{(n+n_K)\times (n+n_K)}$ from $X$ and $Y$ is equivalent to \cite{packard1991collection,gahinet1994linear}
\begin{equation}\label{ineqXY}\begin{bmatrix}X & 0 &I & 0 \\0 & I & 0 & I \\I & 0 & Y&0 \\0 & I &0 & I \end{bmatrix} \succeq 0, \; \operatorname{rank} \left(I_{n} - \begin{bmatrix}Y & 0 \\0 & I_{n_\phi} \end{bmatrix} \begin{bmatrix}X & 0 \\0 & I_{n_\phi} \end{bmatrix}\right) \leq n_K\,.\end{equation}

Clearly, the maximal rank is $\operatorname{rank} (I-YX) \leq n-n_\phi$ and determines the controller order. Finally, for $X$ and $Y$ solutions to  (\ref{proj1}) and (\ref{ineqXY}), the full matrix $X_{cl}$ can be reconstructed as well as controller state-space data $(A_K,B_K,C_K,D_K)$ \cite{gahinet1994linear}.


\begin{thebibliography}{10}

\bibitem{sussmann1991peaking}
H.~Sussmann and P.~Kokotovic, ``The peaking phenomenon and the global
  stabilization of nonlinear systems,'' {\em IEEE Transactions on Automatic
  Control}, vol.~36, no.~4, pp.~424--440, 1991.

\bibitem{francis1978bounded}
B.~Francis and K.~Glover, ``Bounded peaking in the optimal linear regulator
  with cheap control,'' {\em IEEE Transactions on Automatic Control}, vol.~23,
  no.~4, pp.~608--617, 1978.

\bibitem{lin2022co}
Z.~Lin, ``Co-design of linear low-and-high gain feedback and high gain observer
  for suppression of effects of peaking on semi-global stabilization,'' {\em
  Automatica}, vol.~137, p.~110124, 2022.

\bibitem{Taira2017AIAA_Overview}
K.~Taira, S.~L. Brunton, S.~T.~M. Dawson, C.~W. Rowley, T.~Colonius, B.~J.
  McKeon, O.~T. Schmidt, S.~Gordeyev, V.~Theofilis, and L.~S. Ukeiley, ``Modal
  analysis of fluid flows: An overview,'' {\em AIAA Journal}, vol.~55, no.~12,
  pp.~4013--4041, 2017.

\bibitem{Taira2019AIAA_Applications}
K.~Taira, M.~S. Hemati, S.~L. Brunton, Y.~Sun, K.~Duraisamy, S.~Bagheri,
  S.~T.~M. Dawson, and C.-A. Yeh, ``Modal analysis of fluid flows: Applications
  and outlook,'' {\em AIAA Journal}, vol.~58, no.~3, pp.~998--1022, 2020.

\bibitem{an_kreiss}
P.~Apkarian and D.~Noll, ``Optimizing the {K}reiss constant,'' {\em SIAM
  Journal on Control and Optimization}, vol.~58, no.~6, pp.~3342--3362, 2020.

\bibitem{hinrichsen}
D.~Hinrichsen and A.~Pritchard, ``On the transient behaviour of stable linear
  systems,'' in {\em Proc. 14th International Symposium of Mathematical Theory
  of Networks and Systems (MTNS 2000), Perpignan}, pp.~19--23, 2000.

\bibitem{krakov}
H.~Krakovska, C.~Kuehn, and I.~Longo, ``Resilience of dynamical systems,'' {\em
  European Journal of Applied Math.}, vol.~35, no.~1, pp.~1--46, 2023.

\bibitem{ghanbari24}
M.~Ghanbari and J.~Jiang, ``Resilience metrics in power and control systems:
  Comparative analysis,'' in {\em IFAC Papers Online}, vol.~58, pp.~699--704,
  2024.

\bibitem{demmer23}
T.~Demmer, J.~Kahlen, and D.~Lichte, ``The use of control theory to enhance
  systems towards resiliance,'' in {\em Proceedings of the 33rd European Safety
  and Reliability Conference}, pp.~1242--1249, ESREL2023, Singapore, 2023.

\bibitem{bouvier23}
J.-B. Bouvier and M.~Ornik, ``Resilience of linear systems to partial loss of
  control authority,'' {\em Automatica}, vol.~152, no.~june 23, p.~110985,
  2023.

\bibitem{boerner21}
J.~B\"orner and F.~Steinke, ``Measuring {LTI} system resiliance against
  adversarial disturbances based on efficient eigenvalue computation,'' in {\em
  60th CDC21}, 2021.

\bibitem{Packard2003b}
A.~Packard and P.~Seiler, ``{IQC}s and {LMI}s for analysis and synthesis of
  uncertain systems,'' {\em IEEE Transactions on Automatic Control}, vol.~48,
  no.~7, pp.~1127--1143, 2003.

\bibitem{khong_24}
S.~Khong and A.~Lanzon, ``Connections between integral quadratic constraints
  and dissipativity,'' {\em IEEE Trans. Autom. Contr.}, vol.~69, no.~8,
  pp.~5672--5677, 2024.

\bibitem{khong25}
S.~Khong, C.~Chen, and A.~Lanzon, ``Feedback stability analysis via
  dissipativity with dynamic supply rates,'' {\em Automatica}, vol.~172,
  p.~112000, 2025.

\bibitem{Veenmann2013}
J.~Veenmann and C.~Scherer, ``Stability analysis with integral quadratic
  constraints: A dissipativity based proof,'' in {\em 2013 IEEE 52nd Annual
  Conference on Decision and Control (CDC)}, (Florence, Italy), pp.~6289--6294,
  Dec 2013.

\bibitem{seiler_2015}
P.~Seiler, ``Stability analysis with dissipation inequalities and integral
  quadratic constraints,'' {\em IEEE Trans. Autom. Control.}, vol.~60, no.~6,
  pp.~1704--1709, 2015.

\bibitem{gahinet20}
M.~Xia, P.~Gahinet, N.~Abroug, C.~Buhr, and E.~Laroche, ``Sector bounds in
  stability analysis and control design,'' {\em Int. J. Rob. Nonlin. Control},
  vol.~30, pp.~7857--7882, 2020.

\bibitem{simoes20}
V.~Cavalcanti and A.~Sim$\tilde{\rm o}$es, ``{IQC}-synthesis under structural
  constraints,'' {\em Int. J. Robust Nonlin. Control}, vol.~30, pp.~4880--4905,
  2020.

\bibitem{an_iqc}
P.~Apkarian and D.~Noll, ``{IQC} analysis and synthesis via nonsmooth
  optimization,'' {\em Systems and Control Letters}, vol.~55, no.~12,
  pp.~971--981, 2006.

\bibitem{kalur2021nonlinear}
A.~Kalur, P.~Seiler, and M.~S. Hemati, ``Nonlinear stability analysis of
  transitional flows using quadratic constraints,'' {\em Physical Review
  Fluids}, vol.~6, no.~4, p.~044401, 2021.

\bibitem{mushtaq2022feedback}
T.~Mushtaq, P.~J. Seiler, and M.~Hemati, ``Feedback stabilization of
  incompressible flows using quadratic constraints,'' in {\em AIAA AVIATION
  2022 Forum}, p.~3773, 2022.

\bibitem{astolfi22}
D.~Astolfi, L.~Marcoui, and A.~Teel, ``Low-power peaking-free high-gain
  observers for non-linear systems,'' 2022.

\bibitem{whidborne2007minimization}
J.~F. Whidborne and J.~McKernan, ``On the minimization of maximum transient
  energy growth,'' {\em IEEE Transactions on Automatic Control}, vol.~52,
  no.~9, pp.~1762--1767, 2007.

\bibitem{MQMcW2011}
F.~Martinelli, M.~Quadrio, J.~McKernan, and J.~F. Whidborne, ``Linear feedback
  control of transient energy growth and control performance limitations in
  subcritical plane poiseuille flow,'' {\em Phys. Fluids}, vol.~23, no.~1,
  p.~014103, 2011.

\bibitem{ray21}
A.~Ray, A.~Pal, D.~Ghosh, S.~Dana, and C.~Hens, ``Mitigating long transient
  time in deterministic systems by resetting,'' {\em Chaos}, vol.~31,
  p.~011103, 2021.

\bibitem{boyd1994linear}
S.~Boyd, L.~El~Ghaoui, E.~Feron, and V.~Balakrishnan, {\em Linear matrix
  inequalities in system and control theory}.
\newblock SIAM, 1994.

\bibitem{whidborne2005minimization}
J.~F. Whidborne, J.~McKernan, and A.~J. Steer, ``Minimization of maximum
  transient energy growth by output feedback,'' {\em IFAC Proceedings Volumes},
  vol.~38, no.~1, pp.~283--288, 2005.

\bibitem{quenon2021control}
P.~Qu{\'e}non and J.~F. Whidborne, ``Control of plane {P}oiseuille flow using
  the {K}reiss constant,'' in {\em International Conference Cyber-Physical
  Systems and Control}, pp.~41--51, Springer, 2021.

\bibitem{induced_norms}
V.~Chellaboina, W.~M. Haddad, D.~S. Bernstein, and D.~A. Wilson, ``Induced
  convolution operator norms of linear dynamical systems,'' {\em Mathematics of
  Control, Signals and Systems}, vol.~13, pp.~216--239, 2000.

\bibitem{clarke1990optimization}
F.~H. Clarke, {\em Optimization and nonsmooth analysis}.
\newblock SIAM, 1990.

\bibitem{leveque}
R.~J. LeVeque and L.~N. Trefethen, ``On the resolvent condition in the {K}reiss
  matrix theorem,'' {\em BIT Numerical Mathematics}, vol.~24, no.~4,
  pp.~584--591, 1984.

\bibitem{mitchell1}
T.~Mitchell, ``Computing the {K}reiss constant of a matrix,'' {\em SIAM Journal
  on Matrix Analysis and Applications}, vol.~41, no.~4, pp.~1944--1975, 2020.

\bibitem{mitchell2}
T.~Mitchell, ``Fast interpolation-based globality certificates for computing
  {K}reiss constants and the distance to uncontrollability,'' {\em SIAM Journal
  on Matrix Analysis and Applications}, vol.~42, no.~2, pp.~578--607, 2021.

\bibitem{trefethen_embree}
L.~N. Trefethen and M.~Embree, {\em Spectra and pseudospectra, the behavior of
  nonnormal matrices and operators}.
\newblock Princeton University Press, 2005.

\bibitem{Krakovska2024Resilience}
H.~Krakovsk{\'a}, C.~K\"uhn, and I.~Longo, ``Resilience of dynamical systems,''
  {\em European Journal of Applied Mathematics}, 2024.

\bibitem{LeeMarcus2023JFM}
S.~Lee and P.~S. Marcus, ``Linear stability analysis of wake vortices by a
  spectral method using mapped legendre functions,'' {\em Journal of Fluid
  Mechanics}, vol.~967, p.~A2, 2023.

\bibitem{ShcherbakovDabbene2022EJC}
P.~S. Shcherbakov and F.~Dabbene, ``A probabilistic point of view on peak
  effects in linear difference equations,'' {\em European Journal of Control},
  vol.~63, pp.~107--115, 2022.

\bibitem{GarciaHilares2023PhD}
N.~A. Garc{\'i}a~Hilares, {\em Mathematical Modeling and Dynamic Recovery of
  Power Systems}.
\newblock PhD thesis, Virginia Polytechnic Institute and State University,
  Blacksburg, VA, USA, 2023.

\bibitem{Lee2024Thesis}
S.~Lee, ``Linear stability of a wake vortex and its transient growth,''
  Master's thesis, University of California, Berkeley, 2024.

\bibitem{DudarenkoEtAl2023IA}
N.~Dudarenko, N.~Vunder, V.~Melnikov, and A.~Zhilenkov, ``Minimization of peak
  effect in the free motion of linear systems with restricted control,'' {\em
  Informatics and Automation}, vol.~22, no.~3, 2023.

\bibitem{ApkarianNoll2021_OptBasedControl}
P.~Apkarian and D.~Noll, ``Optimization‐based control design techniques and
  tools,'' in {\em Encyclopedia of Systems and Control} (J.~Baillieul and
  T.~Samad, eds.), pp.~1626--1637, Springer, Cham, 2021.

\bibitem{AN2015}
P.~Apkarian, M.~N. Dao, and D.~Noll, ``Parametric robust structured control
  design,'' {\em Automatic Control, IEEE Transactions on}, vol.~60, no.~7,
  pp.~1857--1869, 2015.

\bibitem{apkarianNoll2017worst}
P.~Apkarian and D.~Noll, ``Worst-case stability and performance with mixed
  parametric and dynamic uncertainties,'' {\em International Journal of Robust
  and Nonlinear Control}, vol.~27, no.~8, pp.~1284--1301, 2017.

\bibitem{apkarian2006nonsmooth}
P.~Apkarian and D.~Noll, ``Nonsmooth {$H_\infty$} synthesis,'' {\em IEEE
  Transactions on Automatic Control}, vol.~51, no.~1, pp.~71--86, 2006.

\bibitem{apkarian2006nonsmooth2}
P.~Apkarian and D.~Noll, ``Nonsmooth optimization for multidisk {$H_\infty$}
  synthesis,'' {\em European Journal of Control}, vol.~12, no.~3, pp.~229--244,
  2006.

\bibitem{best_young}
H.~J. Brascamp and E.~H. Lieb, ``Best constants in {Y}oung's inequality, its
  converse, and its generalization to more than three functions,'' {\em
  Advances in Mathematics}, vol.~20, no.~2, pp.~151--173, 1976.

\bibitem{boyd_barratt}
S.~Boyd and C.~Barratt, {\em Linear Controller Design: {L}imits of
  Performance}.
\newblock Prentice-Hall, 1991.

\bibitem{spijker}
M.~Spijker, ``On a conjecture by {L}eveque and {T}refethen related to the
  {K}reiss matrix theorem,'' {\em BIT Numerical Mathematics}, vol.~31,
  pp.~551--555, 1991.

\bibitem{kreiss_himself}
H.-O. Kreiss, ``{\"U}ber die {S}tabilit{\"a}tsdefinition f{\"u}r
  {D}ifferenzengleichungen die partielle {D}ifferentialgleich\-ungen
  approximieren,'' {\em BIT Numerical Mathematics}, vol.~2, pp.~153--181, 1962.

\bibitem{engel2000one}
K.-J. Engel, R.~Nagel, and S.~Brendle, {\em One-parameter semigroups for linear
  evolution equations}, vol.~194.
\newblock Springer, 2000.

\bibitem{swaroop}
D.~Swaroop and D.~Neimann, ``On the impulse response of {LTI} systems,'' in
  {\em Proceedings of the 2001 American Control Conference.(Cat. No.
  01CH37148)}, vol.~1, pp.~523--528, IEEE, 2001.

\bibitem{trefethen1993hydrodynamic}
L.~N. Trefethen, A.~E. Trefethen, S.~C. Reddy, and T.~A. Driscoll,
  ``Hydrodynamic stability without eigenvalues,'' {\em Science}, vol.~261,
  no.~5121, pp.~578--584, 1993.

\bibitem{hinrichsen2000transient}
D.~Hinrichsen and A.~Pritchard, ``On the transient behaviour of stable linear
  systems,'' {\em Proc. Int. Symp. Math. Theory Networks \& Syst. Perpignan,
  France. CDROM - paper B218.}, vol.~2, p.~2, 2000.

\bibitem{schmid2014analysis}
P.~J. Schmid and L.~Brandt, ``Analysis of fluid systems: Stability,
  receptivity, sensitivity. {L}ecture {N}otes from the flow-nordita summer
  school on advanced instability methods for complex flows, {S}tockholm,
  {S}weden, 2013,'' {\em Applied Mechanics Reviews}, vol.~66, no.~2, p.~024803,
  2014.

\bibitem{an:05}
P.~Apkarian and D.~Noll, ``Controller design via nonsmooth multi-directional
  search,'' {\em SIAM J. on Control and Optimization}, vol.~44, no.~6,
  pp.~1923--1949, 2006.

\bibitem{AN:2007}
P.~Apkarian and D.~Noll, ``Nonsmooth optimization for multiband frequency
  domain control design,'' {\em Automatica}, vol.~43, no.~4, pp.~724 -- 731,
  2007.

\bibitem{MatlabRobust}
``Robust control toolbox 6.11,'' 2021.
\newblock The MathWorks, Natick, MA, USA.

\bibitem{boyd_doyle}
S.~Boyd and J.~Doyle, ``Comparison of peak and {RMS} gains for discrete-time
  systems,'' {\em Systems \& Control Letters}, vol.~9, no.~1, pp.~1--6, 1987.

\bibitem{mil}
J.~Doyle and C.~C. Chu, ``Robust control of multivariable and large scale
  systems,'' Tech. Rep. AD-A175 058, Honeywell Systems and Research center,
  1986.

\bibitem{mixedApkarianNoll}
P.~Apkarian and D.~Noll, ``Mixed {$L_1/H_\infty$}-synthesis for
  {$L_\infty$}-stability,'' {\em International Journal of Robust and Nonlinear
  Control}, vol.~32, no.~4, pp.~2119--2142, 2022.

\bibitem{dao}
M.~Dao and D.~Noll, ``Minimizing the memory of a system,'' {\em Mathematics of
  Control, Signals and Systems}, vol.~27, no.~1, pp.~77--110, 2015.

\bibitem{brunton2015closed}
S.~L. Brunton and B.~R. Noack, ``Closed-loop turbulence control: Progress and
  challenges,'' {\em Applied Mechanics Reviews}, vol.~67, no.~5, 2015.

\bibitem{leclercq2019linear}
C.~Leclercq, F.~Demourant, C.~Poussot-Vassal, and D.~Sipp, ``Linear iterative
  method for closed-loop control of quasiperiodic flows,'' {\em Journal of
  Fluid Mechanics}, vol.~868, pp.~26--65, 2019.

\bibitem{illingworth_morgans_rowley_2012}
S.~J. Illingworth, A.~S. Morgans, and C.~W. Rowley, ``Feedback control of
  cavity flow oscillations using simple linear models,'' {\em Journal of Fluid
  Mechanics}, vol.~709, p.~223–248, 2012.

\bibitem{ZDG:96}
K.~Zhou, J.~C. Doyle, and K.~Glover, {\em {Robust and Optimal Control}}.
\newblock Prentice Hall, 1996.

\bibitem{lorenz1963deterministic}
E.~N. Lorenz, ``Deterministic nonperiodic flow,'' {\em Journal of {A}tmospheric
  {S}ciences}, vol.~20, no.~2, pp.~130--141, 1963.

\bibitem{liu2020input}
C.~Liu and D.~F. Gayme, ``Input-output inspired method for permissible
  perturbation amplitude of transitional wall-bounded shear flows,'' {\em
  Physical Review E}, vol.~102, no.~6, p.~063108, 2020.

\bibitem{gahinet1994linear}
P.~Gahinet and P.~Apkarian, ``A linear matrix inequality approach to
  {$H_\infty$} control,'' {\em International Journal of Robust and Nonlinear
  Control}, vol.~4, no.~4, pp.~421--448, 1994.

\bibitem{sturm1999using}
J.~F. Sturm, ``Using sedumi 1.02, a matlab toolbox for optimization over
  symmetric cones,'' {\em Optimization methods and software}, vol.~11, no.~1-4,
  pp.~625--653, 1999.

\bibitem{li2023multivariable}
J.~S. Li, Y.~L. Zhou, S.~B. Shi, and H.~Q. Liu, ``Multivariable control of a
  supercritical boiler with load regulation using loop shaping and {LMI}-based
  robust control,'' {\em IEEE Access}, vol.~11, pp.~17822--17834, 2023.

\bibitem{lim2022active}
D.~H. Lim, S.~H. Lee, and J.~Y. Kim, ``Active vibration control of a high-speed
  railway pantograph using {$H_\infty$} and systune methods,'' {\em Journal of
  Mechanical Science and Technology}, vol.~36, no.~10, pp.~4773--4781, 2022.

\bibitem{Vilarino2022}
D.~L. Vilarino, F.~B. G.~A. de~Sousa, and E.~H.~C. Junior, ``Multivariable
  control of a wind turbine to minimize loads using the
  {$H_\infty$}/{$\mu$}-synthesis and the systune functions,'' {\em Journal of
  Control, Automation and Electrical Systems}, vol.~33, no.~5, pp.~1117--1129,
  2022.

\bibitem{LeGuehennec2020}
G.~Le~Guehennec, J.~Verron, and N.~Petit, ``Robust {$H_\infty$} control design
  for a satellite attitude control system,'' {\em IFAC-PapersOnLine}, vol.~53,
  no.~2, pp.~11175--11181, 2020.

\bibitem{Maheshwari2022}
R.~Maheshwari, A.~Kumar, and S.~Sharma, ``Robust flight control design for a
  launch vehicle using structured {$H_\infty$} synthesis,'' {\em Journal of
  Aerospace Engineering}, vol.~35, no.~2, p.~04021110, 2022.

\bibitem{deSousa2021}
F.~B. G.~A. de~Sousa, M.~R.~V. de~Castro, and M.~W.~S. Maudsley, ``Robust
  control law for a european launcher in the ascent phase,'' {\em Journal of
  Aerospace Engineering}, vol.~34, no.~4, p.~04021029, 2021.

\bibitem{Asghar2023}
M.~I. Asghar, S.~S. Kim, and I.~H. Lim, ``Multivariable robust control of a
  large wind turbine using the {$H_\infty$} synthesis technique,'' {\em Journal
  of Control, Automation and Systems}, vol.~21, no.~10, pp.~3129--3138, 2023.

\bibitem{zabczyk2020mathematical}
J.~Zabczyk, {\em Mathematical {C}ontrol {T}heory: an introduction}.
\newblock Birkh\"auser, Boston, 2008.

\bibitem{yorke1970theorem}
J.~A. Yorke, ``A theorem on {L}iapunov functions using {$\ddot V$},'' {\em
  Mathematical Systems Theory}, vol.~4, no.~1, pp.~40--45, 1970.

\bibitem{ahmadi2011higher}
A.~A. Ahmadi and P.~A. Parrilo, ``On higher order derivatives of {L}yapunov
  functions,'' in {\em Proceedings of the 2011 American Control Conference},
  pp.~1313--1314, IEEE, 2011.

\bibitem{sostools}
A.~Papachristodoulou, J.~Anderson, G.~Valmorbida, S.~Prajna, P.~Seiler, P.~A.
  Parrilo, M.~M. Peet, and D.~Jagt, {\em {SOSTOOLS}: Sum of squares
  optimization toolbox for {MATLAB}}.
\newblock \texttt{http://arxiv.org/abs/1310.4716}, 2021.
\newblock Available from \texttt{https://github.com/oxfordcontrol/SOSTOOLS}.

\bibitem{packard1991collection}
A.~Packard, K.~Zhou, P.~Pandey, and G.~Becker, ``A collection of robust control
  problems leading to {LMIs},'' in {\em [1991] Proceedings of the 30th IEEE
  Conference on Decision and Control}, pp.~1245--1250, IEEE, 1991.

\end{thebibliography}

\end{document}